\documentclass[reqno,11pt]{amsart}

\usepackage{srcltx}
\usepackage{pdfsync}

\usepackage{amssymb,enumerate}
\usepackage[pagebackref]{hyperref}
\usepackage[portrait,a4paper,margin=3cm]{geometry}

\usepackage{graphicx,amssymb,amsmath}

\usepackage{tikz}
\usetikzlibrary{arrows}

\newtheorem{theorem}{Theorem}[section]
\newtheorem{definition}[theorem]{Definition}
\newtheorem{lemma}[theorem]{Lemma}
\newtheorem{proposition}[theorem]{Proposition}

\newtheorem{corollary}[theorem]{Corollary}
\newtheorem{remark}[theorem]{Remark}

\newcommand\bp[1]{\noindent {\em Proof{#1}.} 
}
\def\ep{\hfill $\Box$}

\newcommand{\dE}{\mathbb{E}}

\newcommand{\dN}{\mathbb{N}}
\newcommand{\dP}{\mathbb{P}}
\newcommand{\dR}{\mathbb{R}}

\newcommand{\dZ}{\mathbb{Z}}

 \newcommand{\bbN}{{\mathbb N}}

\newcommand{\cA}{\mathcal{A}}\newcommand{\cB}{\mathcal{B}}
\newcommand{\cC}{\mathcal{C}}\newcommand{\cD}{\mathcal{D}}
\newcommand{\cE}{\mathcal{E}}\newcommand{\cF}{\mathcal{F}}
\newcommand{\cG}{\mathcal{G}}\newcommand{\cH}{\mathcal{H}}

\newcommand{\cK}{\mathcal{K}}\newcommand{\cL}{\mathcal{L}}
\newcommand{\cM}{\mathcal{M}}
\newcommand{\cP}{\mathcal{P}}
\newcommand{\cR}{\mathcal{R}}
\newcommand{\cS}{\mathcal{S}}\newcommand{\cT}{\mathcal{T}}

\newcommand{\cX}{\mathcal{X}}

\newcommand{\bD}{\mathbf{D}}

\newcommand{\bM}{\mathbf{M}}

\newcommand{\bg}{\mathbf{g}}
\newcommand{\ABS}[1]{{{\left| #1 \right|}}} 
\newcommand{\BRA}[1]{{{\left\{#1\right\}}}} 
\newcommand{\PAR}[1]{{{\left(#1\right)}}} 
\newcommand{\SBRA}[1]{{{\left[#1\right]}}} 
\newcommand{\so}{o}

\newcommand{\LIMSUP}{\limsup}
\newcommand{\LIMINF}{\liminf}

\renewcommand{\a}{\alpha} 
\renewcommand{\b}{\beta}  
\renewcommand{\c}{\chi} 
\renewcommand{\d}{\delta}  
\newcommand{\veps}{\varepsilon}
 
\newcommand{\g}{\gamma} 
\newcommand{\h}{\eta}    
  
\renewcommand{\l}{\lambda}

\renewcommand{\o}{\omega}

\renewcommand{\r}{\rho}  

\renewcommand{\t}{\tau}

\newcommand{\D}{\Delta} 
  
\newcommand{\G}{\Gamma}

\newcommand{\Si}{\Sigma}

\newcommand{\si}{\sigma}

\renewcommand{\leq}{\leqslant}             
\renewcommand{\geq}{\geqslant}             
 
\newcommand{\wt}{\widetilde}

\newcommand{\IND}{\mathbf{1}}

\newcommand{\bd}{\mathbf d}
\newcommand{\bn}{\mathbf n}

\newcommand{\supp}{\mathrm{supp}}
\newcommand{\dist}{\mathrm{dist}}

\newcommand{\ibf}{\mathbf i}
\newcommand{\jbf}{\mathbf j}

\renewcommand{\deg}{\mathrm{deg}}

\newcommand{\POI}{\mathrm{Poi}}
\newcommand{\BIN}{\mathrm{Bin}}
\newcommand{\UGW}{\mathrm{UGW}}

\newcommand{\wP}{\widehat{P}}
\newcommand{\wQ}{\widehat{Q}}

\newcommand{\wdP}{\widehat{\dP}}
\newcommand{\BINOM}[2]{\mathrm{Bin}{(#1,#2)}}

\newcommand{\weak}{\rightsquigarrow}

\newcommand{\CM}{\mathrm{CM}}
\newcommand{\wcG}{\widehat{\mathcal{G}}}

\newcommand{\exc}{\mathrm{exc}}
\def\lp{(\hspace{-2pt}(}
\def\rp{)\hspace{-2pt})}

\begin{document}

\title[Large deviations of sparse random graphs]{Large deviations of empirical neighborhood distribution in sparse random graphs}

\author[C. Bordenave]{Charles Bordenave}
\address[Ch.~Bordenave]{IMT UMR 5219 CNRS and Universit\'e Paul-Sabatier Toulouse III, France}\email{charles.bordenave(at)math.univ-toulouse.fr}
\urladdr{http://www.math.univ-toulouse.fr/~bordenave}
 \author[P. Caputo]{Pietro Caputo}
 \address[P.~Caputo]{Dipartimento di Matematica,
  Universit\`a Roma Tre, Italy}
\email{caputo(at)mat.uniroma3.it}
\urladdr{http://www.mat.uniroma3.it/users/caputo}

\begin{abstract} 
Consider the Erd\H{o}s-Renyi random graph on $n$ vertices where each edge is present independently with probability $\l/n$, with $\l>0$ fixed. For large $n$, a typical random graph locally behaves like a Galton-Watson tree with Poisson offspring distribution with mean $\l$.  Here, we study  large deviations from this typical behavior within the framework of the local weak convergence of finite graph sequences. The associated rate function is expressed in terms of an entropy functional on unimodular measures and takes finite values only at measures supported on trees. 
We also establish large deviations for other commonly studied random graph ensembles such as the uniform random graph with given number of edges growing linearly with the number of vertices, or the uniform random graph with given degree sequence. To prove our results, we 
introduce a new configuration model which allows one to sample uniform random graphs with a given neighborhood distribution, provided the latter is supported on trees. We also introduce a new class of unimodular random trees, which generalizes the usual Galton Watson tree with given degree distribution to the case of neighborhoods of arbitrary finite depth. These generalized Galton Watson trees turn out to be useful in the analysis of unimodular random trees and may be considered to be of interest in their own right. 
\end{abstract}


\maketitle
\thispagestyle{empty}
\section{Introduction and main results}
Consider the Erd\H{o}s-Renyi ensemble $\cG(n,p)$, where a random graph is obtained from the vertex set $[n]=\{1,\dots,n\}$ by adding each edge independently with probability $p$. In the sparse regime with $p=\l/n$, for a fixed $\l>0$, it is well known that, for large $n$, a typical graph from $\cG(n,p)$ locally looks like a Galton-Watson tree with Poisson offspring distribution with mean $\l$.  In this work we study  large deviations from this typical behavior. The problem is intimately related to the question: conditioned on having a certain neighborhood distribution, what does a typical element of $\cG(n,p)$ locally look like ? The same questions can be asked for other commonly studied random graph ensembles such as the uniform random graphs with fixed number of edges growing linearly with the number of vertices, or with given degree sequence. 
  We formulate the problem within the theory of {\em local weak convergence} of graph sequences that was recently  introduced by Benjamini and Schramm \cite{bensch} and Aldous and Steele \cite{aldoussteele}. The associated local weak topology has now become a common tool for studying sparse graphs, see Aldous and Lyons \cite{aldouslyons} and Bollob\`as and Riordan \cite{BR2011}. A surprising large variety of graph functionals are continuous for this topology. 
In Section  \ref{sec:LWC} below, we will give more details on local weak convergence. In order to present our result, here we first introduce the main terminology.

 \subsection{Local weak convergence}
\label{subsec:LWC}

A graph $G = (V,E)$, with $V$ a countable set of vertices, is said to be locally finite if, for all $v \in V$, the degree of $v$ in $G$ is finite. 
A {\em rooted graph} $(G,o)$ is a locally finite and connected graph $G = (V,E)$ 
with a distinguished vertex $o \in V$, called the root. For $t \geq 0$, we denote by $(G,o)_t$ the induced rooted graph with vertex set $\{u\in V: \,D(o,u)\leq t\}$, with $D(\cdot,\cdot)$ the natural graph distance. Two rooted graphs $(G_i,o_i) =  ( V_i , E_i , o_i )$, 
$i \in \{1,2\}$, are \emph{isomorphic} if there exists a bijection $\si: V_{1} \to V_{2}$ such that $\si ( o_1) = o_2$ and $\si ( G_1) = G_2$, where $\si$ acts on $E_1$ through $\si (  \{ u , v \}  ) = \{ \si ( u) , \si (v) \}$. We will denote this equivalence relation by $(G_1,o_1) \simeq (G_2, o_2)$. An equivalence class of rooted graphs is often simply referred to as {\em unlabeled rooted graph}. We denote by $\cG^*$ the set of (locally finite, connected) unlabeled rooted graphs. $\cT^*$ will be the set of unlabeled rooted trees. 
To each unlabeled rooted graph $g\in\cG^*$, we may associate  a labeled rooted graph $(G,o)$ with vertex set  $V\subset \dZ_+$, rooted at $0$, in a canonical way; see e.g.\ \cite{aldouslyons}. 
For ease of notation, one sometimes identifies $g \in \cG^*$  with its canonical rooted graph $(G,o)$.

\begin{figure}
\begin{tikzpicture}[-,>=stealth',shorten >=1pt,auto,node distance=1.5cm,
                    thick,main node/.style={circle,draw,font=\sffamily
                    \bfseries}]

  \node[main node] at (0,1.3) (1) {1};
  \node[main node] (2) [right of=1] {2};
  \node[main node] (3) [below of=1] {3};
  \node[main node] (4) [below of=2] {4};
   \node[main node] (6) [below of=4] {6};
   \node[main node] (5) [below of=3] {5};



\draw [fill] (5,1.5) circle (0.17cm);
\draw   (4.4,0.5) circle (0.17cm);
\draw  (5.6,0.5) circle  (0.17cm);
\draw  (5,-0.5) circle (0.17cm);
\draw  (5,-1.5) circle (0.17cm);

\draw (4.44,0.65) -- (5,1.5);
\draw (5.56,0.65) -- (5,1.5);
\draw (4.44,0.34) -- (4.94,-0.36);
\draw (5.56,0.34) -- (5.06,-0.36);
\draw (5,1.5) -- (5,-0.36);
\draw (5,-0.66) -- (5,-1.36);

\draw [fill] (7,1.5) circle (0.17cm);
\draw   (6.4,0.5) circle (0.17cm);
\draw  (7.6,0.5) circle  (0.17cm);
\draw  (7,-0.5) circle (0.17cm);
\draw  (6.4,-1.5) circle (0.17cm);

\draw (6.44,0.65) -- (7,1.5);
\draw (7.56,0.65) -- (7,1.5);
\draw (6.44,0.34) -- (6.94,-0.36);
\draw (7.56,0.34) -- (7.06,-0.36);
\draw (6.56,0.5) -- (7.47,0.5);
\draw (6.4,0.34) -- (6.4,-1.36);
\draw [fill] (9,1.5) circle (0.17cm);
\draw   (8.4,0.5) circle (0.17cm);
\draw  (9.6,0.5) circle  (0.17cm);
\draw  (9,-0.5) circle (0.17cm);
\draw  (8,1.5) circle (0.17cm);

\draw (8.44,0.65) -- (9,1.5);
\draw (9.56,0.65) -- (9,1.5);
\draw (8.44,0.34) -- (8.94,-0.36);
\draw (9.56,0.34) -- (9.06,-0.36);
\draw (9,1.5) -- (9,-0.36);
\draw (8.15,1.5) -- (8.85,1.5);

%
%

\draw  (11,0.5) circle (0.17cm);
\draw   (10.4,-0.5) circle (0.17cm);
\draw  (11.6,-0.5) circle  (0.17cm);
\draw  (11,-1.5) circle (0.17cm);

\draw [fill] (11,1.5) circle (0.17cm);

\draw (10.44,-0.35) -- (10.94,0.36);
\draw (11.56,-0.35) -- (11.06,0.36);
\draw (10.44,-0.66) -- (10.94,-1.36);
\draw (11.56,-0.66) -- (11.06,-1.36);
\draw (11,0.34) -- (11,-1.36);
\draw (11,0.66) -- (11,1.36);

%
%
%
%
%
%
%
%

\draw [fill] (13,1.5) circle (0.17cm);

\node at (5,-2.3) {$\a$};
\node at (7,-2.3) {$\b$};
\node at (9,-2.3) {$\c$};
\node at (11,-2.3) {$\g$};
\node at (13,-2.3) {$\veps$};

%
  \path[every node/.style={font=\sffamily\small}]
    (1) 
        edge node {} (2)
     edge node {} (3)
        edge node {} (4)
    (2) 
    edge node {} (4)
    (3)            
    edge node {} (4)
    (4) 
           edge node {} (5)
     ;
\end{tikzpicture}
%
\caption{Example of a graph $G$ and its empirical neighborhood distribution. Here $U(G)=\frac16(\d_{\a}+2\d_{\b}+\d_\c+\d_\g + \d_{\veps})$, where $\a,\b,\c,\g, \veps \in\cG^*$ are the unlabeled rooted graphs depicted above (the black vertex is the root), with $[G,1]=\a$, $[G,2]=[G,3]=\b$, $[G,4]=\c$, $[G,5]=\g$, $[G,6] = \veps$.}
\label{ugsample}
\end{figure}

For $\g\in\cG^*$ and $h \in \dN$, we write $\g_h$ for the truncation at $h$ of the graph $\g$, namely the unlabeled rooted graph obtained by removing all vertices (together with the edges incident to them) that are at distance larger than $h$ from the root.  
The {\em local topology} is the smallest topology such that for any $\g \in \cG_*$ and $h \in \dN$, the $\cG^* \to \{0,1\}$ function $
f(g) = \IND ( g_h = \g_h ) $
is continuous. Equivalently, a sequence $g_n\in\cG^*$ converges locally to $g\in\cG^*$ iff for all $h\in\dN$ there exists $n_0(h)$ such that $(g_n)_h=g_h$ whenever  $n\geq n_0(h)$. This topology is metrizable and the space $\cG^*$ is separable and complete \cite{aldouslyons}. The space   of probability measures on $\cG^*$, denoted $\cP(\cG^*)$, is equipped with the topology of weak convergence. We often write $\r_n\weak \r$ to indicate that a sequence $\r_n\in\cP(\cG^*)$ converges weakly to $\r\in\cP(\cG^*)$. 

For a finite graph $G = (V,E)$ and $v\in V$, one writes $G(v)$ for the connected component of $G$ at $v$. The {\em empirical neighborhood distribution} $U(G)$ of $G$ is the law of the equivalence class of the rooted graph 
$(G(o),o)$ where the root $o$ is sampled uniformly at random from $V$, i.e.\ $U(G)\in\cP(\cG^*)$ is defined by
\begin{equation}\label{eq:defUG}
U(G)=\frac 1{|V|}\sum_{v\in V} \d_{[G,v]},
\end{equation}
where $[G,v]\in\cG^*$ stands for the equivalence class of $(G(v),v)$ and $\d_g$ is the Dirac mass at $g\in\cG^*$; see Figure \ref{ugsample} for an example. If $\{G_n\}$ is a sequence of finite graphs, we shall say that $G_n$ has {\em local weak limit}  $\rho\in\cP(\cG^*)$ if $U(G_n)$ converges to  $\rho$ in $\cP(\cG^*)$ as $n\to\infty$. A measure $\r\in\cP(\cG^*)$ is called {\em sofic} if  there exists a sequence of finite graphs $\{G_n\}$ whose local weak limit is $\rho$. In other words, the set of sofic measures is the closure of the set $\{ U (G_n) : G_n \hbox{ finite graph}\}$.  
An example is the Dirac mass at the infinite regular tree with degree $d\in\dN$, which is almost surely the local weak limit of a sequence of uniformly sampled random $d$-regular graphs on $n$ vertices \cite{wormald}. Another example is the law of the Galton-Watson tree with Poisson offspring distribution with mean $\l>0$, which is almost surely the local weak limit of a sequence of random graphs sampled from $\cG(n,p)$ when $p=\l/n$. 
Sofic measures form a closed subset 
of $\cP(\cG^*)$. 
 
 Sofic measures share a 
 stationarity property called {\em unimodularity} \cite{aldouslyons}.
To define the latter, consider the set $\cG^{**}$ of unlabeled graphs with two distinguished roots, obtained as the set of equivalence classes of 
locally finite connected graphs with two distinguished vertices $(G,u,v)$. 
The notion of local topology extends naturally to $\cG^{**}$. 
A function $f$ on $\cG^{**}$ can be extended to a function on connected graphs with two distinguished roots $(G,u,v)$ through the isomorphism classes. Then, a measure $\r\in\cP(\cG^*)$ is called {\em unimodular} if for any Borel measurable function $f: \cG^{**} \to \dR_+$, we have 
\begin{equation}\label{eq:defunimod}
\dE_\r \sum_{ v \in V} f ( G , o , v) = \dE_\r \sum_{ v \in V} f ( G , v , o),
\end{equation}
where  $(G,o)$ is the canonical rooted graph whose equivalence class $g\in\cG^*$ has law $\rho$. It is not hard to check that if $G$ is a finite graph then its neighborhood distribution $U(G)$ is unimodular. 
In particular, all sofic measures are unimodular. The converse is open; see \cite{aldouslyons}. We denote by $\cP_{u} (\cG^*)$ the set of unimodular probability measures. Similarly, we write $\cP_{u} (\cT^*)$ for unimodular probability measures supported by trees.

\subsection{Unimodular Galton-Watson trees with given neighborhood}
We now introduce a family of unimodular measures that will play a key role in what follows.  As we will see, this is the natural generalization of the usual Galton-Watson trees with given degree distribution to the case of neighborhoods of arbitrary depth $h\in \dN$. These measures will be shown to be sofic, and this fact can be used to give an alternative proof of the Bowen-Elek theorem \cite{Bowen,elek2010,BLS} asserting  that all $\r\in\cP_{u} (\cT^*)$ are sofic, see Corollary \ref{cor:entropy1} below. 

Fix $h \in \dN$, and recall that $g_h$ denotes the truncation at depth $h$ of $g\in\cG^*$. Call $\cG^*_h$ the set of unlabeled rooted graphs with depth $h$, i.e.\ the set of $g\in\cG^*$ such that $g_h=g$. Similarly, call $\cT^ *_h$ the set of unlabeled rooted trees $t\in\cT^*$ such that $t_h=t$.  
Given $\r\in\cP(\cG^*)$, we write $\r_h\in\cP(\cG^*_h)$ for the $h$-neighborhood marginal of $\r$, i.e.\ the law of $g_h$ when $g$ has law $\r$. Notice that if $t\in\cT^*$ and $h=1$, then $t_h$ is simply the number of children of the root. In particular, $\cT^*_1$ can be identified with $\dZ_+$. When $h=0$, it is understood that  
$\cG^*_0$ contains only the trivial graph consisting of a single isolated vertex (the root), so that $|\cG^*_0|=1$.


If $G = (V,E)$ is a graph and  $\{ u , v\} \in E$ then define $G(u,v)$ as the rooted graph $(G'(v),v)$,  where $G'=(V, E\backslash \{u,v\})$, i.e.\ $G(u,v)$  is the rooted graph obtained from $G$ by removing the edge $\{u,v\}$ and taking the connected component at the root $v$. 
Next, given a rooted graph $(G,o)$, and $g,g'\in \cG_{h-1}^*$, define 
\begin{equation}\label{eq:defE}
E_h ( g , g') =\big| \big\{ v \stackrel{G}{\sim} o :\, G(o,v)_{h-1} \simeq g , \,G(v,o)_{h-1} \simeq g' \big\}\big|. 
\end{equation}
The notation $v \stackrel{G}{\sim} u$ indicates that the vertex $v$ is a neighbor of $u$ in $G$.  
Thus, $E_h ( g , g')$ is the number of neighbors of the root in $(G,o)$ which have the given patterns $G(o,v)_{h-1}\simeq g$ and $G(v,o)_{h-1} \simeq g'$. Notice that if $h=1$,  then necessarily $g,g'=o$ and $E_1(o,o)=\deg_G(o)$ is simply the degree of the root.   

As an example, consider the the rooted graph $\a$ from Figure \ref{ugsample}. Fix $h=2$, and call $g_1,g_2$ the elements of $\cG^*_{h-1}$ consisting respectively of  a rooted single edge and a rooted triangle. Then one has $E_h(g_1,g_2)=2$ and $E_h(g_2,g_1)=0$. Similarly, if the reference graph is $\b$  from Figure \ref{ugsample}, then $E_h(g_1,g_2)=0$ while $E_h(g_2,g_1)=1$.

We call a measure $P \in \cP( \cG^*_h)$  {\em admissible} if $\dE_P \deg_G (o) < \infty$ and for all 
$g, g' \in \cG^ *_{h-1}$, 
\begin{equation*}
e_{P} (g,g') = e_P (g',g),
\end{equation*}
where \begin{equation}\label{eq:defe}
e_P (g,g'):= \dE_P E_h(g,g').
\end{equation}
Here it is understood that $(G,o)$ represents the canonical rooted graph whose equivalence class
in $\cG_h^*$ has law $P$. By applying the definition of unimodularity \eqref{eq:defunimod} to the function 
$$f (G, u, v) = \IND \big( v \stackrel{G}{\sim} u \big) \IND \big( G(u,v)_{h-1} \simeq g  ; G(v,u)_{h-1} \simeq g'  \big),$$ 
it is not hard to check that if $\r$ is unimodular and $\dE_\rho \deg_G (o) < \infty$ then $\rho_h\in\cP(\cG_h^*)$ is admissible. In particular, for any finite graph $G$, the neighborhood distribution $U(G)_h$ truncated at depth $h$ is admissible.
Remark that, when $h=1$, since $|\cG^*_0|=1$, all $P \in \cP( \cT_1^*) = \cP( \dZ_+)$ with finite mean are admissible.

We now define the measures $\UGW_h(P) \in \cP (\cT^*)$; see also Section \ref{sec:UGW} below for more details. Fix $P\in\cP (\cT_h^*)$ admissible.
The probability $\UGW_h(P) \in \cP (\cT^*)$ is the law of the equivalence class of the random rooted tree $(T,o)$ 
defined below. 
%
%

For $t, t' \in \cT^*_{h-1}$ such that $e_P(t,t') \ne 0$ define,  for all $\tau \in \cT^*_{h}$, 
\begin{equation}\label{hatP}
\wP_{t,t'} ( \tau ) =   P( \tau \cup t'_+) \PAR{ 1 +   \big| \big\{  v \stackrel{\tau}{\sim} o : \tau(o,v) \simeq t' \big\}\big| }  \frac{\IND( \tau_{h-1} = t ) }{ e_P (t,t')} , 
\end{equation}
where $\tau\cup t'_+$ denotes  the tree obtained from $\tau$ by adding a new neighbor of the root whose rooted subtree is $t'$; see Figure \ref{sfig2} for an example. The subtree $\t(o,v)$ is defined before Eq.\ \eqref{eq:defE} with the graph $G$ replaced by $\t$.

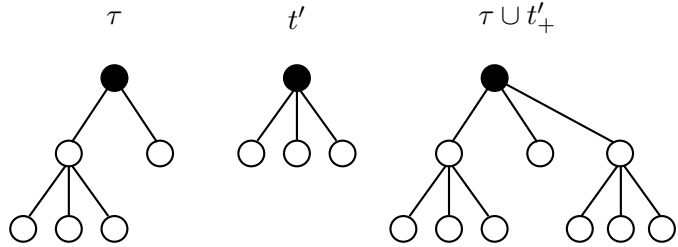
\begin{figure}[h]

\begin{tikzpicture}[-,>=stealth',shorten >=1pt,auto,node distance=1.5cm,
                    thick,main node/.style={circle,draw,font=\sffamily
                    \bfseries}]

\draw [fill] (7,1.5) circle (0.17cm);
\draw   (6.4,0.5) circle (0.17cm);
\draw  (7.6,0.5) circle  (0.17cm);
\draw  (7,-0.5) circle (0.17cm);
\draw  (6.4,-.5) circle (0.17cm);
\draw  (5.8,-.5) circle (0.17cm);

\draw (6.44,0.65) -- (7,1.5);
\draw (7.56,0.65) -- (7,1.5);
\draw (6.44,0.34) -- (6.94,-0.36);
\draw (6.36,0.34) -- (5.86,-0.36);
\draw (6.4,0.34) -- (6.4,-0.36);

\draw [fill] (9.4,1.5) circle (0.17cm);
\draw  (10,0.5) circle (0.17cm);
\draw  (9.4,.5) circle (0.17cm);
\draw  (8.8,.5) circle (0.17cm);

\draw (9.44,1.34) -- (9.94,0.64);
\draw (9.36,1.34) -- (8.86,0.64);
\draw (9.4,1.34) -- (9.4,0.64);

\draw [fill] (12,1.5) circle (0.17cm);
\draw   (11.4,0.5) circle (0.17cm);
\draw  (12.6,0.5) circle  (0.17cm);
\draw  (13.65,0.5) circle  (0.17cm);
\draw  (12,-0.5) circle (0.17cm);
\draw  (11.4,-.5) circle (0.17cm);
\draw  (10.8,-.5) circle (0.17cm);

\draw (11.44,0.65) -- (12,1.5);
\draw (12.56,0.65) -- (12,1.5);
\draw (11.44,0.34) -- (11.94,-0.36);
\draw (11.36,0.34) -- (10.86,-0.36);
\draw (11.4,0.34) -- (11.4,-0.36);
\draw (13.56,0.65) -- (12,1.5);

\draw (13.7,0.34) -- (14.2,-0.36);
\draw (13.62,0.34) -- (13.12,-0.36);
\draw (13.66,0.34) -- (13.66,-0.36);

\draw  (13.66,-0.5) circle (0.17cm);
\draw  (14.2,-.5) circle (0.17cm);
\draw  (13.12,-.5) circle (0.17cm);

\node at (7,2.3) {$\t$};
\node at (9.4,2.3) {$t'$};
\node at (12.3,2.3) {$\t\cup t'_+$};

\end{tikzpicture}

\caption{In this example one has $h=2$ and $1+\big| \big\{  v \stackrel{\tau}{\sim} o : \tau(o,v) \simeq t' \big\}\big|= 2$. Notice that if $t=\t_1$ denotes the rooted tree with two leaves, then in the rooted tree $\t\cup t'_+$ one has  $E_h(t',t)=2$ according to the definition \eqref{eq:defE}.}
\label{sfig2}
\end{figure}

It can be checked that $\wP_{t,t'}$ is a probability, i.e.\ $\wP_{t,t'}\in \cP(\cT^*_h)$; see Section \ref{sec:UGW}. 
We may now define the random rooted tree $(T,o)$. First, $(T,o)_{h}$ is sampled according to $P$. 
Next, for each vertex $v$ in the first generation of $(T,o)_{h}$, consider the subtree $t=T(o,v)_{h-1}$ with depth $h-1$ rooted at $v$ obtained by removing the edge $\{o,v\}$ and retaining the connected component up to distance $h-1$ from $v$. We add a layer to $t$ by replacing $t$ with a new tree $\t$ with depth $h$  that  coincides with $t$ in the first $h-1$ generations. The new tree $\t$ is sampled according to $\wP_{t, t'}$ where $t$ is as above while $t'$ denotes the subtree $T ( v,o)_{h-1}$ rooted at $o$ obtained from $(T,o)_{h}$ by removing the edge $\{o,v\}$ and retaining the connected component up to distance $h-1$ from $o$.  This operation is repeated for each $v$ in the first generation independently. After this step,  we have overall added one layer to $(T,o)_{h}$, and thus we have sampled $(T,o)_{h+1}$. 

We now proceed recursively, layer by layer, to obtain a sample of the full tree $(T,o)$. Formally, this construction can be stated as follows.   
If $u$ is the parent of $v$, we say that $v$ has {\em type} $(t,t')$, where $t,t'\in\cT^*_{h-1}$, if $T(u,v)_{h-1}\simeq t$ and $T ( v,u)_{h-1}\simeq t'$. The subtrees $T(u,v)$, and $T(v,u)$ are defined before Eq.\ \eqref{eq:defE} with $G$ replaced by $T$.
Denote by $1, \cdots, d$, with $d = \deg_T(o)$ the neighbors of the root in the canonical representation of the random variable with law $P$. Given $(T,o)_{h}$, the subtrees $T(o,v), 1 \leq v \leq d,$ are independent random variables and, given that $v$ has type $(t,t')$, then $T(o,v)_{h}$ has distribution $\wP_{t, t'}$. Once $T(o,v)_{h}$ is sampled, the type of a child $v'$ of $v$ 
is determined using only  $T(o,v)_{h}$ and $T(v,o)_{h-2}$. For each child $v'$ of $v$ we sample the subtree $T(v,v')_h$ independently according to $\wP_{t, t'}$ where $(t,t')$ is the type of $v'$ and so on, 
recursively. This defines our random rooted tree $(T,o)$. 

If $h =1 $, then there is only one type possible and $\UGW_1( P)$ is the  unimodular 
Galton-Watson tree with degree distribution $P \in \cP(\dZ_+)$, where 
the number $d$ of children of the root is sampled according to $P$, and conditionally on $d$, the subtrees of the children of the root are independent Galton-Watson trees with offspring distribution 
given by the size-biased law $\wP$:
\begin{equation}\label{eq:defwP}
\wP ( k ) =   \frac{ (k+1) P(k+1) }{ \sum_{\ell=1}^\infty \ell P (\ell)}. 
\end{equation}
If $P=\POI(\l)$ is the Poisson distribution with mean $\l$, then  $\wP=P$ and $\UGW_1( P)$ is the standard Galton-Watson tree with mean degree $\l$. 

The following proposition summarizes the main properties of the measures $\UGW_h(P)$ for generic $h\in\dN$ and $P\in\cP (\cT_h^*)$ admissible.
\begin{proposition}\label{prop:UGW}             
Fix $h\in\dN$ and $P\in\cP (\cT_h^*)$ admissible. The measure $\UGW_h(P)$ is unimodular. 
Moreover, the following consistency relation is satisfied:
 for any $k \geq h$, $(\UGW_h(P))_k\in\cP (\cT_k^*)$ is admissible and 
$$
\UGW_h (P) = \UGW_k( (\UGW_h(P))_k ). 
$$
\end{proposition}

 \subsection{Entropy of a measure $\r\in\cP(\cG^*)$}
 It is convenient to work with uniformly distributed random graphs with a given number of edges.
 For any $n,m\in\dN$, let $\cG_{n,m}$ be the set of graphs on $V=[n]$ with $|E|=m$ edges. 
Fix $d >0$, and a sequence $m = m(n)$ such that $m / n \to d/2$, as $n\to\infty$. 
Since
$$
\ABS{\cG_{n,m} } = \binom{
{ n (n -1)/2}}{ m },
$$
an application of Stirling's formula shows that 
\begin{equation}\label{eq:Gnmlog}
\log \ABS{  \cG_{n,m} } =   m \log n + s(d)\,n + o(n),\qquad 
s(d):=\frac{d}{2}  - \frac{ d } 2 \log d.
\end{equation}
 If $\r\in\cP(\cG^*)$, define
$$
\cG_{n,m} ( \rho , \veps) = \BRA{ G \in \cG_{n,m} : \;U(G) \in B ( \rho, \veps) }, 
$$
where $B ( \rho, \veps)$ denotes  the open ball with radius $\veps$ around $\r$ with respect to the L\'evy metric on $\cP(\cG^*)$.
 For $\veps >0$, define
$$
\overline \Sigma(\rho,\veps)  =  
 \limsup_{n \to \infty} \frac{ \log  \ABS{ \cG_{n,m} ( \rho , \veps)  }  -  m  \log n  }{ n }.
$$
Since $\veps \mapsto \overline\Sigma(\rho,\veps)$ is non-decreasing, one defines
$$
\overline \Sigma (\rho ) = \lim_{\veps \to 0} \downarrow \overline\Sigma(\rho,\veps).
$$
The extended real numbers $\underline \Sigma( \rho,\veps)$ and  $\underline \Sigma( \rho)$ are defined as above, with $\limsup$ replaced by $\liminf$. If $\rho$ is such that $\underline \Sigma(\rho) = \overline \Sigma(\rho)$, we set $\Sigma(\rho) := \overline \Sigma (\rho ) = \underline \Sigma (\rho )$. The number $\Sigma(\rho)$ can be interpreted, up to an overall constant, as a microcanonical entropy associated to the state $\rho$. 
From \eqref{eq:Gnmlog}, one has that $\Si(\r)\in[-\infty,s(d)]$, whenever it is well defined. 
\begin{theorem}\label{th:entropy1}
Fix $d > 0$ and  choose a sequence $ m = m(n)$ such that $m / n \to d / 2$. For any $\rho \in \cP( \cG^*)$, the entropy $\Sigma(\rho)  \in \; \left[-\infty,  s(d)\right]$ is well defined, it is upper semi-continuous, and it does not depend on the choice of the sequence $m(n)$. Moreover, $\Sigma(\rho) = - \infty$ if at least one of the following is satisfied:
\begin{enumerate}[i)]
\item $\r$ is not unimodular 
\item $\r$ is not supported on rooted trees.
\item $\dE_{\rho} \deg_G (o) \ne d$. 
\end{enumerate}
\end{theorem}
Notice that the definition of $\Sigma(\rho)$ depends on 
the parameter $d$. For simplicity, we do not write explicitly this dependence. 
In view of Theorem \ref{th:entropy1}(iii), to avoid trivialities, unless otherwise stated, $\Sigma(\rho)$ will refer to the value at $d = \dE_\rho \deg_G ( o )$ (provided that the latter is finite). 
The next theorem computes the actual value of $\Sigma(\rho)$ for unimodular Galton-Watson trees and gives an expression for $\Sigma(\rho)$ for all 
$\r\in \cP_u (\cT^*)$. Moreover, it shows that unimodular Galton-Watson trees maximize entropy under a $h$-neighborhood marginal constraint. 

Let us introduce some additional notation. For any $P \in \cP (\cT^*_h)$, define the Shannon entropy 
$$H(P)=-\sum_{t\in\cT^*_h}P(t)\log P(t).$$
For $h \in \dN$, call $\cP_h$ the set of all $P \in \cP (\cT^*_h)$, with $P$ admissible such that $H(P)<\infty$ and $\dE_P \SBRA{  \deg_T(o) \log \PAR { \deg_T (o)} } < \infty$\footnote{We shall actually see with Lemma \ref{le:phgap2} below that $P \in \cP_h$ is equivalent to $P$ admissible and $\dE_P \SBRA{  \deg_T(o) \log \PAR { \deg_T (o)} } < \infty$.}. 
For $P\in\cP_h$, let $\pi_P$ denote the probability on $\cT^*_{h-1}\times \cT^*_{h-1}$ defined by 
$$
\pi_P(s,s') = \frac1{d}\,e_P(s,s')\,,\;\qquad (s,s') \in\cT^*_{h-1}\times \cT^*_{h-1}, 
$$
where $ d = \dE_{P} \deg_G (o)$, $e_P (s,s')= \dE_P E_h(s,s')$, and $E_h(s,s')$ is defined in \eqref{eq:defE}.
We write $H(\pi_P)$ for the Shannon entropy of $\pi_P$:
$$H(\pi_P)=-\!\!\!\!\!\!\!\sum_{(t,t')\in\cT^*_{h-1}\times\cT^*_{h-1}}\pi_P(t,t')\log \pi_P(t,t').$$
\begin{theorem}\label{th:entropy2}
Fix $h \in \dN$. 
The expression
\begin{equation}\label{enh0}
J_h(P) =  - s(d)+
H (P) -\frac{d}2\,H(\pi_P) \;-\!\!\!\!\!\! \sum_{(s,s')\in\cT^*_{h-1}\times \cT^*_{h-1}}   \dE_P \log (E_h(s,s')!),
\end{equation}
defines a function 
$J_h:\cP_h\mapsto [-\infty,s(d)]$, satisfying 
$$\Sigma(\UGW_h(P)) = J_h(P),$$ for all $P\in\cP_h$.
Define $\overline J_h:   \cP (\cT^*_h)\mapsto [-\infty,s(d)]$ by $\overline J_h (P) = J_h(P)$ if $P\in\cP_h$,
and $\overline J_h(P)=-\infty$ if $P\notin\cP_h$.
If $\rho \in \cP_{u}(\cT^*)$, then for all $h\in\dN$,
\begin{equation}\label{upperenh0}
\Sigma(\rho) \leq \overline J_h(\r_h),
\end{equation} 
and, if $\rho_1$ has finite support, the inequality is strict unless
$\r= \UGW_h(\r_h)$.
Finally, for any $\rho \in \cP_{u}(\cT^*)$, $\overline J_h(\r_h)$ is non-increasing in $h\in\dN$, and 
\begin{equation}\label{sigmalim}
\Sigma( \rho) =  \lim_{h\to\infty}   \downarrow \overline J_h( \rho_h). 
\end{equation}
\end{theorem}
In Remark \ref{altrem} below we provide an alternative expression for $J_h(P)$ in terms of relative entropies.
Specializing  to the case $h=1$,  we obtain the following corollary of Theorem \ref{th:entropy2}. 

\begin{corollary}\label{cor:entropy}
If $P\in \cP( \dZ_+)$ has mean $d$, then
$$
\Sigma(\UGW_1(P)) =  s(d)
- H ( P \,|\, \POI(d) ),
$$
where $\POI(d)$ stands for Poisson distribution with mean $d$, and $H(\cdot\,|\,\cdot)$ is the relative entropy.
\end{corollary}
In particular,  the standard Galton-Watson tree $\rho = \UGW_1( \POI(d))$ 
maximizes the entropy $\Sigma(\r)$ among all measures $\r$ with mean degree $d$.   

As a byproduct of our analysis, 
we will also obtain an alternative proof of the Bowen-Elek Theorem \cite{Bowen,elek2010,BLS}. 

\begin{corollary}\label{cor:entropy1}
If $\rho \in \cP_{u} (\cT^*)$, then $\rho$ is sofic. 
\end{corollary}

We observe finally that, from its definition,  the map $\Sigma: \rho \mapsto \Sigma(\rho)$ is easily seen to be upper semi-continuous for the local weak topology (see Lemma \ref{le:lscSi}). In 
 Proposition \ref{prop:discSi} below, we will however prove that $\Si$ fails to be continuous at any $\rho = \UGW_1 (P)$ whenever $P \in \cP(\dZ_+)$ has finite support and satisfies $P(0) = P(1) = 0$, $P(2) < 1$.

\subsection{Large deviations 
of uniform graphs with given degrees}
Given a 
 vector $\bd  \in \dZ_+^n$, let $\cG(\bd)$ denote the set of 
 graphs $G = ([n],E)$ such that $\bd$ is the degree sequence of $G$, i.e.\ if $\bd=(d(1),\dots,d(n))$, then for all $v \in [n]$, $\deg_G(v) = d(v)$. 
  Consider a sequence $\bd^{(n)}$, $n\in\dN$, of degree vectors $(d^{(n)}(1),\dots,d^{(n)}(n))$ such that, for some fixed $\theta\in\dN$, and $P\in\cP(\dZ_+)$: 
   \begin{enumerate}[(C1)]
\item $\sum_{v=1}^nd^{(n)}(v)$ is even,
\item $\max_{1\leq v \leq n} d^{(n)}(v) \leq \theta$,  
\item $\frac 1 n \sum_{v\in [n]} \delta_{d^{(n)}(v)} \weak P$,  
\end{enumerate}
where $ \weak$ denotes weak convergence in $\cP(\dZ_+)$. A consequence of Erd\H{o}s and Gallai \cite{Erdos1960} is that if (C1)-(C3) above are satisfied, then $\cG(\bd^{(n)})$ is not empty for all $n$ large enough.
We  shall consider a random graph $G_n$ sampled uniformly from $\cG(\bd^{(n)})$. Models of this type are well known in the random graph literature; see e.g.\  Molloy and Reed \cite{MR1370952}. In particular, it is a folklore fact that 
almost surely the neighborhood distribution $U(G_n)$ defined in \eqref{eq:defUG} is weakly convergent to $\UGW_1(P)$; see also Theorem \ref{th:convlocCM} below for a more general statement. One of our main results concerns the large deviations of $U(G_n)$.
Here and below whenever we say that $U(G_n)$ satisfies the large deviation principle (LDP) in $\cP(\cG^*)$ with speed $n$ and good rate function $I$, we mean that 
 the function $I:\cP(\cG^*)\mapsto[0,\infty]$ is lower semi-continuous with compact level sets, and for every Borel set $B\subset\cP(\cG^*)$
 \begin{equation}\label{eq:defLDP}
-\inf_{\r \in B^\circ }I(\r )\leq \liminf _{n\rightarrow
\infty}\frac{1}{n}
\log \dP\left(  U(G_n) \in   B \right) \leq\limsup
_{n\rightarrow\infty}\frac{1}{n}\log \dP\left(  U(G_n) \in B  \right) \leq -\inf _{ \r  \in \overline{B}}I(\r ),
\end{equation}
where $B^\circ$ denotes the interior of $B$ and
$\overline{B}$ denotes the closure of $B$.

\begin{theorem}\label{th:LDPCM}
Let $\bd^{(n)}$ be a sequence satisfying conditions $(C1)-(C3)$ above. 
Let $G_n$ be uniformly distributed on $\cG(\bd^{(n)})$. Then $U(G_n)$ satisfies the LDP in $\cP(\cG^*)$ with speed $n$ and good rate function 
$$
I ( \rho ) = \left\{ \begin{array}{ll}
\Sigma(\UGW_1(P)) - \Sigma(\rho) & \hbox{ if $\rho_1 = P$}, \\
\infty & \hbox{ otherwise.}
\end{array}\right.
$$
\end{theorem}
It follows from Theorem \ref{th:entropy2} that for any integer $h \geq 1$, and $Q \in \cP_h $ with $Q_1 = P$, then
$$
\min \{ I ( \rho) : \;\rho_h = Q \} =  J_1(P)-J_h ( Q) , 
$$
and the minimum is uniquely attained for $\rho = \UGW_h(Q)$. 
This allows one to compute large deviations of neighborhood measures $U(G_n)_h$ explicitly in terms of the function $J_h$. 

On the other hand, consider the special case of $d$-regular graphs, where $\bd^{(n)}$ is the constant vector $(d,\dots,d)$, and $P=\d_d$, for some fixed $d\in\dN$. To have $\Si(\r)>-\infty$, $\r$ must be supported on trees by Theorem \ref{th:entropy1}, and because of the constant degree constraint we find that the only $\r\in \cP(\cG^*)$ such that $I(\r)<\infty$ is the 
Dirac mass at the infinite rooted $d$-regular tree, which coincides with $\UGW_1(P)$, where the rate function is zero. 
Thus, for the $d$-regular random graph 
$I(\r)$ is either zero or infinite, and one should look at faster speed than $n$ here for non trivial large deviations.

We note finally Theorem \ref{th:LDPCM} establishes a large deviations principle with speed $n$. Other interesting large deviation events occur at higher speed. For example, for the proportion of vertices in a triangle in $G_n$, the speed would be $n \log n$. 

\subsection{Large deviations of Erd\H{o}s-R\'enyi graphs}
Next, we describe our main results for sparse Erd\H{o}s-R\'enyi  graphs such as the uniform random graph
from $\cG_{n,m}$, with $m\sim nd/2$ and the $\cG(n,p)$ where each edge is independently present with probability $p=d/n$. 
It is well known that, in both cases, with probability one, $U(G_n)$ converges weakly to the standard Galton-Watson tree with mean degree $d$, i.e.\ $\r=\UGW_1( \POI(d))$, which by Corollary \ref{cor:entropy} satisfies $\Si(\r)=s(d)$. 
 \begin{theorem}\label{th:LDPER1}
Fix $d > 0$ and a sequence $ m = m(n)$ such that $m / n \to d / 2$, as $n\to\infty$. Let $G_n$ be uniformly distributed in $\cG_{n,m}$. 
Then $U(G_n)$  satisfies the LDP in $\cP(\cG^*)$ with speed $n$ and good rate function
\begin{equation}\label{LDPER11}
I(\r)= \left\{ \begin{array}{ll}
s(d)- \Sigma(\rho) & \hbox{ if $\dE_\rho \deg_G(o) = d$}, \\
 \infty & \hbox{ otherwise.}
\end{array}\right.
\end{equation}
\end{theorem}
 
\begin{theorem}\label{th:LDPER2}
Fix $\l > 0$ and  take $G_n$ with law $\cG(n,\l/n)$. Then $U(G_n)$ satisfies the LDP in $\cP(\cG^*)$ with speed $n$ and good rate function
\begin{equation}\label{LDPER21}
I(\r)= \frac \l2 - \frac d 2 \log \l - \Sigma (\rho),
\end{equation}
where $d:=\dE_\r\deg_G(o)$, with the convention that if $d=0$ then $\Sigma (\r)=s(0)=0$.
\end{theorem}
In the special case of $1$-neighborhoods, Theorem \ref{th:LDPER1}, 
Theorem \ref{th:LDPER2}
and Corollary \ref{cor:entropy} allow us to prove the following results.
Let $u(G_n)\in \cP(\dZ_+)$ denote the empirical distribution of the degree: $u(G_n)  = \frac1n\sum_{i=1}^n \d_{\deg_{G_n}(i)}$.  

\begin{corollary}\label{cor:LDPER1}
Fix $d > 0$, a sequence $ m = m(n)$ such that $m / n \to d / 2$, and let $G_n$ be uniformly distributed in $\cG_{n,m}$. 
Then $u(G_n)$ 
satisfies the  LDP in $\cP(\dZ_+)$ with good rate function 
$$
K(P) =  \left\{ \begin{array}{ll}
 H ( P \,|\, \POI(d) )  & \hbox{ if \;$\sum_k k P(k) = d$}, \\
 \infty & \hbox{ otherwise.}
\end{array}\right.
$$
\end{corollary}
\begin{corollary}\label{cor:LDPER2}
Fix $\l > 0$ and take $G_n$ with law $\cG(n,\l/n)$.  
Then $u(G_n)$ 
satisfies the LDP in $\cP(\dZ_+)$ with speed $n$ and good rate function 
$$
K(P) =  \left\{ \begin{array}{ll}
\frac {\l - d} 2 - \frac{ d } 2 \log \frac \l d  + H ( P \,|\, \POI(d) )  & \hbox{ if \;$d:=\sum_k k P(k)  < \infty$} \\
 \infty & \hbox{ otherwise.}
\end{array}\right.
$$
\end{corollary}
%

\subsection{Plan and methods}
The proof of the main results discussed above is organized as follows. 
In Section \ref{sec:LWC} we review some basic facts about local weak convergence in the context of multi-graphs. We also establish a compactness criterion which parallels recent results of Benjamini, Lyons and Schramm \cite{BLS}. In Section \ref{sec:UGW} we introduce the unimodular Galton Watson trees with given $h$-neighborhood distribution and prove the properties stated in Proposition \ref{prop:UGW} . 
 In Section \ref{EN} we prove our main results concerning the entropy $\Si(\r)$, cf.\ Theorem \ref{th:entropy1} and   Theorem \ref{th:entropy2}. These are  
 crucially
based on the possibility of counting asymptotically the number of graphs in 
$\cG_{n,m}$ which have a certain $h$-neighborhood distribution. 
To compute such things, we introduce what we call a {\em generalized configuration model}. 
   The standard configuration model, introduced in Bollobas \cite{Bollobas1980}, allows one to compute asymptotically the number of
graphs with a given degree sequence. Since here we want to uncover the $h$-neighborhood of a vertex and not only its degree,
we need to generalize the usual construction. To keep track of the $h$-neighborhood structure, we introduce directed multigraphs with colored edges and analyze the associated configuration model;  see Section \ref{CM}. This will allow us to sample a random graph with a given sequence of $h$-neighborhoods, as long as these neighborhoods are rooted trees.  
As an application, we prove Corollary \ref{cor:entropy1} at the end of Section \ref{CM}. It seems to us that this new configuration model may turn out to be a natural tool in other applications as well.
Finally, Section \ref{LD} is devoted to the proof of  large deviation principles in the classical random graphs ensembles.
We stress that our methods allow in principle a much greater generality, since one could establish
large deviation estimates for random graphs that are uniformly sampled from the class of all graphs with 
a given $h$-neighborhood distribution and not only with given degree sequences; see Remark \ref{re:LDPCM}.

\subsection{Related work}
Large deviations in random graphs is a rapidly growing topic. 
For dense graphs, e.g.\ $\cG(n,p)$ with fixed $p\in(0,1)$, a thorough treatment has been given recently by Chatterjee and Varadhan \cite{MR2825532}, in the framework of the cut topology introduced by Lov{\'a}sz and Szegedy \cite{LoSz}, see also Borgs, Chayes, Lov{\'a}sz, S{\'o}s and Vesztergombi \cite{MR2455626,MR2925382}. In the sparse regime, only a few partial results are known. 
O'Connell \cite{MR1616567}, Biskup, Chayes and Smith \cite{MR2352285} and Puhalskii \cite{MR2118868} have proven large deviation asymptotics for the connectivity and for the size of the connected components.
Large deviations for degree sequences of Erd\H{o}s-R\'enyi graphs has been studied in Doku-Amponsah and M{\"o}rters \cite{MR2759726} and Boucheron, Gamboa and L\'eonard \cite[Theorem 7.1]{BGL2002}.  Closer to our approach, large deviations in the local weak topology were obtained for critical multi-type Galton-Watson trees by Dembo, M{\"o}rters and Sheffield \cite{DMS2005}.
Finally, large deviations for other models of statistical physics on Erd\H{o}s-R\'enyi graphs have been considered in Rivoire \cite{MR2099724} and Engel, Monasson, and Hartmann \cite{MR2099723}. 

As far as we know, this is the first time that large deviations of the neighborhood distribution are addressed in a systematic way. While our approach does not cover results on connectivity and the size of connected components such as \cite{MR1616567}, it does yield a simplification of some of the existing arguments concerning the large deviations for degree sequences. We point out that our Corollary \ref{cor:LDPER2}
gives a corrected version of \cite[Corollary 2.2]{MR2759726}. 
Under a stronger sparsity assumption, large deviations of neighborhood distributions for random networks have been used in \cite{BordenaveCaputo2012} to study the large deviations of the spectral measure of certain random matrices.

\section{Local weak convergence}
\label{sec:LWC}

In this section, we first recall the basic notions of local weak convergence in the more general context of rooted multi-graphs; see  \cite{bensch}, \cite{aldoussteele}, and \cite{aldouslyons}. Then, we give a general tightness lemma.

\subsection{Local convergence of rooted multi-graphs}

Let $V$  be a countable set, a {\em multi-graph} $G = (V, \omega)$ is a vertex set $V$ together with a map  
$\omega$ from $V^2$ to $\dZ_+$ such that for all $(u,v) \in V^2$, $\omega ( u, u)$ is even and 
$
\omega (u,v) =  \omega (v,u).
$
For ease of notation, we sometimes set $\omega (v) = \omega (v,v)$ for the weight of the loop at $v$. If $e = \{u,v\}$ is an unordered pair ($u\ne v$), we may also write $\omega(e)$ in place of $\omega(u,v)$. 
The {\em edge set} $E$ of $G$ is the set of unordered pairs $e = \{ u,v\}$ such that $\omega( e) \geq 1$, $\omega(e)$ being the {\em multiplicity} of the edge $e \in E$. Similarly, $\omega(v) /2$ is the number of {\em loops} attached to $v$. A multi-graph with no loop, and with no edge with multiplicity greater than $1$ is a graph.

The {\em degree} of $v$ in $G$ is defined by
$$\deg (v)  =  \sum_{u \in V } \omega (v,u).$$
The multi-graph $G$ is  {\em locally finite}  if for any vertex $v$, $\deg(v)< \infty$.

We denote by $\wcG$ the set of all locally finite multi-graphs. For a multi-graph $G\in\wcG$, to avoid possible confusion, we will often denote by $V_G$, $\omega_G$, $\deg_G$  the corresponding vertex set, weight and degree functions.

Recall that a path $\pi$ from $u$ to $v$ of length $k$ is a sequence $\pi = (u_0, \cdots, u_k)$ with $u_0 = u$, $u_k = v$ and, for $0 \leq i \leq k-1$, $\{ u_i, u_{i+1}\} \in E$. If such $\pi:u\to v$ exists, the distance $D(u,v)$ in $G$ between $u$ and $v$ is defined as the minimal length of all paths from $u$ to $v$. If there is no path $\pi:u\to v$, then the distance $D(u,v)$ is set to be infinite. A multi-graph is {\em connected} if $D(u,v)<\infty$ for any $u\neq v \in V$.

Below, a {\em rooted multi-graph} $(G,o) = ( V, \omega, o)$ is a locally finite and connected multi-graph $(V, \omega)$ 
with a distinguished vertex $o \in V$, the root. For $t \geq 0$, we denote by $(G,o)_t$ the induced rooted multi-graph with vertex set $\{u\in V: \,D(o,u)\leq t\}$. 
Two rooted multi-graphs $(G_i,o_i) =  ( V_i , \omega_i , o_i )$, 
$i \in \{1,2\}$, are \emph{isomorphic} if there exists a bijection $\si : V_{1} \to V_{2}$ such that $\si ( o_1) = o_2$ and $\si ( G_1) = G_2$, where $\si$ acts on $G_1$ through $\si (  u , v   ) = (\si ( u) , \si (v) )$ and $\si ( \omega ) = \omega \circ \si$. We will denote this equivalence relation by $(G_1,o_1) \simeq (G_2, o_2)$. The associated equivalence classes can be seen as unlabeled rooted multi-graphs. 
We call $\wcG^*$ the set of all such equivalence classes.

We define the semi-distance $d$ between two rooted multi-graphs $(G_1,o_1)$ and $(G_2,o_2)$ as 
$$d ((G_1,o_1),( G_2,o_2)) = \frac1{1 + T}\,,$$ where $T$ is the supremum of those $t > 0$ such that $(G_1,o_1)_t$ and $(G_2,o_2)_t$ are isomorphic. 
%
On the space $\wcG^*$, $d$ is a  distance. The associated
topology will be referred to as the \emph{local topology}. The space $(\wcG^*,d)$ is Polish (i.e.\ separable and complete) \cite{aldouslyons}. 

Explicit compact subsets of  $\wcG^*$ can be constructed as follows. If $g \in \wcG^*$, we define 
$$
| g | = \sum_{v \in V} \deg ( v ),
$$
i.e.\ twice the total number of edges in $g$.
For $g \in \wcG^*$, $t \in\dN$, the truncation at distance $t$, $g_t$, is defined as the equivalence class of $(G,o)_t$ where the equivalence class of $(G,o)$ is $g$. 
\begin{lemma}\label{le:compactloc}
Let $t_0 \geq 0$ and $\varphi : \dN \to \dR_+ $ be a non-negative function. Then 
$$K = \big\{ g \in \wcG^* : \forall t \geq t_0 ,  \, | g_ t | \leq \varphi (t) \big\},$$ is a compact subset of 
$\,\wcG^*$ for the local topology. 
\end{lemma} 
\begin{proof}
For each $t \geq t_0$, there is a finite number of elements in $\wcG^*$, say $f_{t,1}, \cdots, f_{t,n_t}$, such that $| g | \leq \varphi (t)$ and for any vertex the distance to the root is at most $t$. Therefore, the collection $A_{t,1}, \cdots, A_{t,n_t}$ where $A_{t,k} = \{g \in \wcG^* :  g _t = f_{t,k} \}$ is a finite covering of $K$ of radius $1/(1+t)$. 
\end{proof}

The notions of local weak convergence introduced in \S \ref{subsec:LWC} are immediately extended to the present setting of multi-graphs. 
%
The definitions of $U(G)$ in \eqref{eq:defUG} and unimodularity \eqref{eq:defunimod} easily carry over to $\cP ( \wcG^*)$. The next simple lemma is proved in \cite{bensch}.

\begin{lemma}\label{le:closeunimod}
The set $\cP_{u} (\wcG^*)$ is closed in the local weak topology. 
\end{lemma}

\subsection{Compactness lemma for the local weak topology} 

Let $G_n$ be a sequence of finite multi-graphs. We now give a condition which guarantees that the sequence $U(G_n)$ is tight for the local weak topology. If $G = (V, \o)$ is a multi-graph, we define the degree of a subset $S \subset V$ as 
\begin{equation}\label{eq:defdegS}
\deg_G(S) = \sum_{ v \in S} \deg_G (v). 
\end{equation}

The next lemma is a sufficient condition for tightness in $\cP_{u} (\wcG^*)$. A similar  result appears in  Benjamini, Lyons and Schramm \cite[Theorem 3.1]{BLS}. We give an independent proof. 

\begin{lemma}\label{le:tightLWC}
Let $\delta : [0,1] \to \dR_+$ be a continuous increasing function such that $\delta(0) = 0$. There exists a compact set $\Pi = \Pi (\delta) \subset \cP_{u} (\wcG^*)$ such that if a finite multi-graph $G = (V, \o)$
satisfies 
\begin{equation}\label{eq:tightLWC}
\deg_G(S) \leq |V| \delta \PAR{\frac{  |S| } { |V | } }
\end{equation}
for all $S \subset V$, 
then  $U(G) \in \Pi$.
\end{lemma}

Considering a sequence $U(G_n), n \geq 1$, condition \eqref{eq:tightLWC} amounts to a uniform integrability of the degree sequences of the multi-graphs $(G_n), n \geq 1$. It may seem quite paradoxical that a sole condition on the degrees implies the tightness of the whole graph sequence. However, 
the unimodularity of $U(G)$ yields  enough uniformity for  this result to hold.  

\begin{proof}[Proof of Lemma \ref{le:tightLWC}]
Since $\wcG^*$ is a Polish space, from Prohorov's theorem, a set $\Pi \subset \cP(\wcG^*)$ is relatively compact if and only if for any $\veps > 0$, there exists a compact $K\subset\wcG^*$ such that for all $\mu \in \Pi$, $\mu ( K^c ) \leq \veps$. 

Set $c = \delta (1)$. Without loss of generality, we may assume $c > 1$. We consider the increasing function $[0 , c] \mapsto [0,1]$  
$$f = \delta^{-1}.$$
Now, for each $\veps>0$, 
and integer $t \geq 1$, we set 
$$
h_\veps ( t) = (f \circ \cdots \circ f) ( \veps 2^{-t} ) \quad  \hbox{ and } \quad \varphi_\veps ( t) =  \frac{  ( c  / h_\veps (t)  ) ^ { t} - 1 } { 1 -  h_\veps(t) / c },
$$
where the composition holds $t$ times. We now define $\Pi$ as being the closure of the set of measures $\mu$ in $\cP_{u} (\wcG^*)$ such that for any $\veps > 0$, $\mu ( K_\veps^c ) \leq \veps$ where 
$$
K_\veps = \BRA{ g \in \wcG^* : \forall t \geq 1, \, | g_ t | \leq \varphi_\veps (t) }.
$$ 
By Lemma \ref{le:compactloc}, $K_\veps$ is a compact set of $\wcG^*$. Hence,  Prohorov's theorem asserts that $\Pi$ is a compact set of $\cP_{u} (\wcG^*)$.

We now check that $\rho = U(G) \in \Pi$. This will conclude the proof of our lemma. It is sufficient to prove that $\rho (  K_\veps ) \geq 1 -  \veps$ for all $\veps>0$. 
Let $t \geq 0$ be an integer, for $S \subset V$, $B(S,t)$ denote the set of vertices at distance at most $t$ from a vertex in $S$. In particular, if $v\in V$ and $g$ is the equivalence class of $(G(v),v)$ we have $$\deg ( B( v, t ) ) = | g_t |.$$
Notice  also that $| B (S, 1) | \leq \deg (S)$. Set $|V| = n$. By iteration on \eqref{eq:tightLWC}, it follows that if $S \subset V$ is such that $|S| \leq h_\veps ( t) n $ then 
$$
 |B ( S, t) |  \leq \deg ( B (S,t-1) ) \leq( f\circ\cdots \circ f) ( |S|/n)\leq 2^{-t} \veps n. 
$$ 
Moreover, from \eqref{eq:tightLWC}, we have
$$
\deg( V )  = \sum_{v \in V} \deg (v) \leq c  n
$$
Hence, using Markov inequality, we deduce that the set 
$$
S_t = \BRA{ v \in V : \deg(v) \geq c  / h_\veps (t) } 
$$
has cardinality at most $  h_\veps ( t) n $. From what precedes, the set 
$$
U_t = \BRA{ v \in V : \exists u \in B (v,t), \,  \deg_{G_n}(u) \geq c / h_\veps (t) } 
$$
has cardinality at most $2^{-t} \veps n$. Note that, if $ v \notin U_t$, then $\deg( B (v , t) )$ is bounded by 
$$
 \frac{  ( c  / h_\veps (t)  ) ^ { t} - 1 } {c / h_\veps(t) - 1 }=\varphi_\veps ( t) . 
$$
This implies that the set 
$$
V_t = \BRA{ v \in V : \deg ( B ( v, t ) ) \geq \varphi_\veps ( t) } 
$$
has cardinality at most $2^{-t} \veps n$. So finally, from the union bound, the set 
$$
W = \BRA{ v \in V : \forall t \geq 1, \, \deg ( B ( v, t ) ) \leq \varphi_\veps ( t) } 
$$
has cardinality at least $(1 -  \veps)n$. We have thus checked that $\rho (  K_\veps ) \geq 1 - \veps$. 
\end{proof}

\section{Unimodular Galton-Watson trees with given neighborhood}
\label{sec:UGW}

The aim of this section is to prove Proposition \ref{prop:UGW}. We thus fix $h \in \dN$ and $P \in \cP(\cT^*_h)$  admissible. We start with some simple observations which ensure that $\UGW_h(P)$ is indeed well defined. 

First observe that if $\tau \in \cT^* _{h}$, $t' \in \cT^* _{h-1}$ and $S =  \tau \cup t'_+$, then 
(recall the definition of $ \tau \cup t'_+$ and Figure \ref{sfig2})
\begin{equation}\label{eq:qqqqq}
 1 +   \big| \big\{  v \stackrel{\tau}{\sim} o : \tau(o,v) \simeq t' \big\}\big|  =  \big| \big\{  v \stackrel{S}{\sim} o : S(v,o) \simeq \tau , S(o,v) \simeq t' \big\}\big|.  
\end{equation}
Therefore, for any $t,t' \in \cT^* _{h-1}$,
\begin{align*}
& \sum_{\tau \in \cT^* _h} P( \tau\cup t'_+) \PAR{ 1 +   \big| \big\{  v \stackrel{\tau}{\sim} o : \tau(o,v) \simeq t' \big\}\big| }  \IND( \tau_{h-1} = t ) \\
& \quad =  \sum_{\tau \in \cT^* _h}\sum_{S \in \cT^*_ h}  P( S ) \big| \big\{  v \stackrel{S}{\sim} o : S(v,o) \simeq \tau , S(o,v) \simeq t' \big\}\big|  \IND ( S = \tau \cup t'_+)   \IND( \tau_{h-1} = t ) \\
& \quad =   \sum_{S \in \cT^* _h} P( S ) \big| \big\{  v \stackrel{S}{\sim} o : S(v,o)_{h-1} \simeq t , S(o,v) \simeq t' \big\}\big| 
 =  e_P (t,t'),
\end{align*}
where $e_P$ was defined by 
\eqref{eq:defe}. We thus have checked that $\wP_{t,t'}$ defined by  \eqref{hatP} is indeed a probability measure on $\cT^*_h$.  Consequently, the probability measure $\UGW_h(P)$ is well defined.

\subsection{Unimodularity} 
The next lemma is a direct argument for the unimodularity of $\UGW_h(P)$, which establishes the first part of Proposition \ref{prop:UGW}. We remark however that this fact could be derived indirectly from Theorem \ref{th:convlocCM} and Lemma \ref{le:treelike} below, which ensure in particular that  $\UGW_h(P)$ is sofic (and hence unimodular). 
\begin{lemma} \label{le:UGWunimod}Fix $h\in\dN$ and $P\in\cP (\cT_h^*)$ admissible. The measure $\UGW_h(P)$ is unimodular. 
\end{lemma}
\begin{proof}
It is sufficient to check the so-called involution invariance, i.e.\ that \eqref{eq:defunimod} holds with $f$ restricted to functions $f : \cG^{**} \to \dR_+$ such that $f ( G, u, v) = 0$ unless $\{ u , v\} \in E_G$; see \cite{aldouslyons}. Recall that we may extend $f : \cG^{**} \to \dR_+$ to all connected graphs with two distinguished roots  $(G,u,v)$ through the isomorphism class. 

 Let $(T ,o)$ be the random rooted tree defined in the introduction whose equivalence class has law $\UGW_h (P)$. Recall that the neighbors of the root $o$ are indexed by $1, \cdots, \deg_T(o)$ and that the vector of subtrees $(T(o,1),\cdots,T(o,\deg_T(o)))$ is exchangeable. We write
\begin{align*}
\dE \sum_{ v \stackrel{T}{\sim} o } f ( T , o  , v ) & =   \sum_{g\in\cT^*_h}  P(g) \sum_{ v \stackrel{g}{\sim} o } \dE [ f ( T , o  ,v  ) \,|\, (T,o)_h \simeq g  ] \\
& =   \sum_{ \tau\in\cT^*_h, \,t' \in\cT^*_{h-1}}  P(S ) \big| \big\{   v \stackrel{S}{\sim} o : S(o,v) \simeq t' ,S(v,o) \simeq \tau \big\}\big| \\
& \quad \quad \quad\quad\quad  \times  \; \dE [ f ( T , o  ,1  )\,|\, T(o,1)_{h-1}  \simeq t' , T(1,o)_{h}  \simeq  \tau  ],
\end{align*}
where, in the summand, $ S = \tau \cup t'_+$. 
Now,  \eqref{hatP} and \eqref{eq:qqqqq} imply 
\begin{align*}
\dE \sum_{ v \stackrel{T}{\sim} o } f ( T , o  , v )  =   \sum_{t , t'} e_P(t,t') \sum_{\tau : \,\,\tau_{h-1} = t}   \wP_{t,t'} (\tau ) \dE [ f ( T , o  ,1  ) \,|\,T(o,1)_{h-1}  \simeq t' , T(1,o)_{h}  \simeq \tau ].
\end{align*}

For $(t,t') \in \cT^* _{h-1}$, we introduce a new random tree $H=H_{t,t'}$ defined as follows. Start with two vertices $o$ and $o'$ which are connected by an edge. Attach the tree $t$ to $o$ and the tree $t'$ to $o'$, so that the type of $o$ is $(t,t')$ and the type of $o'$ is $(t',t)$.
Sample independently $H(o',o)_{h}$
according to $\wP_{t,t'}$ and $H(o,o')_{h}$ according to $\wP_{t',t}$. 
The subtrees  $H(o',o)_{h}$ and $H(o,o')_{h}$ define the types of the children of $o$ and $o'$. Next, sample independently their rooted subtrees, according to their types, i.e.\ $H(o,v)_h$ (resp. $H(o',v)_h$) is sampled according to $\wP_{a,b}$ if $v\sim o$ (resp.\ $v\sim o'$) has type $(a,b)$.  Repeating  recursively for all children defines the random  tree $H$. From the definition of $\UGW_h (P)$, one has 
\begin{align*}
 \sum_{\tau : \,\,\tau_{h-1} = t}   \wP_{t,t'} (\tau ) \dE [ f ( T , o  ,1  ) |  T(o,1)_{h-1}  \simeq t' , T(1,o)_{h}  \simeq \tau  ]  =  \dE_{t,t'} [ f ( H , o  ,o'  ) ],
\end{align*}
where we use $\dE_{t,t'}$ for expectation over the random $H=H_{t,t'}$ defined above.
It follows that
$$
\dE \sum_{ v \stackrel{T}{\sim} o } f ( T , o  , v ) = \sum_{t , t'} e_P(t,t')  \dE_{t,t'} [ f ( H , o  ,o'  ) ].
$$
Similarly,
$$
\dE \sum_{ v \stackrel{T}{\sim} o } f ( T , v , o ) =  \sum_{t , t'} e_P(t,t') \dE_{t,t'} [ f ( H , o'  ,o  ) ]  = \sum_{t , t'} e_P(t,t') \dE_{t',t} [ f ( H , o  ,o'  ) ],
$$
where the second identity follows from the symmetry in $o,o'$ in the definition in $H$, which implies that $\dE_{t,t'} [ f ( H , o'  ,o  ) ]=\dE_{t',t} [ f ( H , o  ,o'  ) ]$.

Finally, the assumption $e_P(t,t') = e_P(t',t)$ yields 
$$
\dE \sum_{ v \stackrel{T}{\sim} o } f ( T , v  , o )  =  \sum_{t , t'} e_P(t',t)  \dE_{t',t} [ f ( H , o  ,o'  ) ] 
 =  \dE \sum_{ v \stackrel{T}{\sim} o } f ( T , o  , v ). 
$$
\end{proof}

\subsection{Consistency lemma}
We turn to the second part of Proposition \ref{prop:UGW}. The following lemma computes the law of the $(h+1)$-neighborhood of a Galton-Watson tree with a given $h$-neighborhood.

\begin{lemma}\label{le:consistency0}
Fix $h\in\dN$, $P \in \cP( \cT^*_h)$ admissible and 
set $\rho = \UGW_h ( P)$. For any $\tau \in \cT^{*}_{h+1}$ with $\deg_\tau(o) = d$, we have 
$$
 \dP_{\rho} \PAR{ (T,o)_{h+1} = \tau   }  =  P\PAR{\tau_h } \prod_{a \in \cA}  { n_a \choose  (k_{a,b})_{b \in \cB_a}  } \prod_{b \in \cB_a} \wP_{s^a,s^{-a}} (t^{a,b})^{k_{a,b}}, 
$$
where 
\begin{enumerate}[\SMALL{$\bullet$}]
\item  $(t^i \in \cT^*_{h}, 1 \leq i \leq d)$ are the subtrees of $\tau$ attached to the offspring of the root, and for $1 \leq i \leq d$, $s^i = (t^i)_{h-1}$ ; 
\item $\{s^a\}_{a \in \cA}$ is set of distinct elements of $( s^i , 1\leq i \leq d )$, and, for each $a \in \cA$,  $\{t^{a,b}\}_{b \in \cB_a}$ is the set the distinct elements of $(t^i, 1\leq i \leq d)$, such that $( t^{a,b} )_{h-1} = s^a $ ;  
\item $n_a$ is the cardinality of $s^i$'s equal to $s^a$ and $k_{a,b}$ is the cardinality of $t^i$'s equal to $t^{a,b}$ ; 
\item  $s^{-a}  = (t^{-a})_{h-1}$ and $t^{-a} \in \cT^*_{h}$ is the tree obtained from $\tau_{h}$ by removing one offsping with subtree equal to $s^{a}$. 
\end{enumerate} 
\end{lemma}
\begin{proof}
Using $\rho_h = P$, for a fixed $\t\in \cT^{*}_{h+1}$, the above definitions allow us to write
\begin{align*}
 \dP_{\rho} &\PAR{ (T,o)_{h+1} = \tau   }=    P\PAR{ \t_h } \dP_{\rho} \PAR{  (T,o)_{h+1} = \tau  | (T,o)_{h} = \t_h  }  \\
&\quad =   \; P\PAR{  \t_h  } \dP_{\rho} \PAR{  \forall a \in \cA, b \in \cB_a  : \big| \big\{ v \stackrel{T}{\sim } o : T(o,v)_{h} = t^{a,b} \big\}\big| = k_{a,b} \,  \bigm| \,  (T,o)_{h} = \t_h   }.
\end{align*}
Observe that  $T(o,v)_{h} = t^{a,b}$ implies that $T(o,v)_{h-1} = s^a$. Moreover, given $(T,o)_{h} = t$,  $T(o,v)_{h-1} = s^a$ implies that $T(v,o)_{h-1} = s^{-a}$, i.e.\ the type of vertex $v$ is $(s^a, s^{-a})$. The lemma is then a consequence of the conditional independence of the subtrees attached to the offspring of the root given $(T,o)_h$. 
\end{proof}


\begin{lemma}\label{le:consistency}
Fix integers $ k>h\geq 1$, $P \in \cP( \cT^*_h)$ admissible and 
set $\rho = \UGW_h ( P)$. Then 
$$
\rho = \UGW_k ( \r_k). 
$$
\end{lemma}

\begin{proof}
By recursion, it suffices to prove the statement for $k = h+1$. For $s \in \cT^*_{h-1}$ such that $e_P(s,s') >0$ for some $s' \in \cT^* _{h-1}$, we may define the probability measure 
\begin{equation*}
\wdP_s ( \cdot ) = \frac { \dE_\rho \big| \big\{ v \stackrel{T}{\sim} o : T(o,v) \in \cdot \; , \;T(v,o)_{h-1} = s \big\}\big| } { \dE_\rho \big| \big\{ v \stackrel{T}{\sim} o : T(v,o)_{h-1} =  s\big\}\big| }. 
\end{equation*}
In words, $\wdP_s\in \cP(\cT^*)$ is the law of the whole subtree $T(o,v)$  of a neighbor $v$ of the root given that $T(v,o)_{h-1} = s$, where $(T,o)$ has law $\rho$. Next, we show that, for $s, s' \in \cT^*_{h-1}$, $t \in \cT^*_h$ such that $t_{h-1} = s$ and $e_P(s,s') >0$, one has
\begin{equation}\label{eq:QwhP}
\wP_{s,s'} ( t) =  \frac { \wdP_{s'} (  (T,o)_{h} = t )} {\wdP_{s'}(  (T,o)_{h-1} = s ) },
\end{equation}
where $(T,o)$ is now the random variable with law $\wdP_{s'}$.
Since $P=\r_h$ one has $e_P(s,s') = \dE_{\r} \big| \big\{ v \stackrel{T}{\sim} o : T(o,v)_{h-1}  =   s \; , T(v,o)_{h-1} = s' \big\}\big|$, and therefore
$$
 \frac { \wdP_{s'} (  (T,o)_{h} = t )} {\wdP_{s'}(  (T,o)_{h-1} = s ) } = \frac{ \dE_{\rho} \big| \big\{ v \stackrel{T}{\sim} o : T(o,v)_{h} =  t \; , T(v,o)_{h-1} =  s' \big\}\big|}{ e_P(s,s') }. 
$$
However, with $
n =  \big| \big\{  v \stackrel{t}{\sim} o : t(o,v) = s' \big\} \big| ,
$
we deduce from the unimodularity of $\rho$ and \eqref{eq:qqqqq} that 
\begin{align}
\dE_{\rho} \big| \big\{ v \stackrel{T}{\sim} o : T(o,v)_{h} =   t \; , T(v,o)_{h-1} =  s'\big\}\big| &  =  \dE_{\rho} \big| \big\{ v \stackrel{T}{\sim} o : T(v,o)_{h} =   t \; , T(o,v)_{h-1} =  s' \big\}\big| \nonumber \\
& =  (n+1) P ( t \cup s'_+ ). \label{eq:qq2}
\end{align}
This proves \eqref{eq:QwhP}. 

Now we set  $Q=\r_{h+1}$ and $\rho' =  \UGW_{h+1} ( Q)$. Our aim is to prove that $\rho'= \rho$.  
It is sufficient to prove that for any $t, t' \in \cT^*_{h}$ and $\tau \in \cT^*_{h+1}$ such that $\tau_{h} = t$ and $e_Q (t,t') >0$, 
\begin{equation}\label{eq:PwhP}
\wQ_{t,t'} (\tau) = \frac { \wdP_{s'} (  (T,o)_{h+1} = \tau )} {\wdP_{s'}(  (T,o)_{h} = t ) },
\end{equation}
where 
$s'=t'_{h-1}$. Indeed, since $\rho'$ and $\rho$ have the same $h+1$ neighborhood, this would prove that they have in fact the same $h+2$ neighborhood and, by conditional independence, we would deduce that $\rho = \rho'$. 

Let us prove \eqref{eq:PwhP}. Set $$  k = \big|\big\{ v \stackrel{\tau}{\sim} o : \tau (o,v) = t' \big\} \big| \leq n=\big| \big\{  v \stackrel{t}{\sim} o : t(o,v) = s' \big\} \big|,$$ where, as above, $t=\t_h$ and $s'=t'_{h-1}$. Since $(\tau \cup t'_+)_h = t \cup s'_+$,  and $\rho'_{h+1}=\r_{h+1}$, we have 
\begin{align*}
Q ( \tau \cup t'_+ ) & =  P (  t \cup s'_+ ) \dP_{\rho'} ( (T,o)_{h+1} = \tau \cup t'_+   | (T,o)_h = t \cup s'_+ ) \\
& =  P ( t \cup s'_+ ) ) \dP_{\rho} ( (T,o)_{h+1} = \tau \cup t'_+   | (T,o)_h = t \cup s'_+).
\end{align*}
As in Lemma \ref{le:consistency0}, let $t^i \in \cT^*_{h}, 1 \leq i \leq d$ be the subtrees of $\tau$ attached to the offspring of the root and call $s^i$ their restriction to $\cT^*_{h-1}$. By construction, $k$ elements of the $t^i$'s are equal to $t'$ and $n$ elements of the $s^i$'s  are equal to $s'$. Let $(s^a)_a$ be the set of distinct elements of the set $\{s'\}\cup\{s^i , 1\leq i \leq d \}$, and, for each $a$, let $(t^{a,b})_b$ denote the distinct elements of $\{t'\}\cup\{t^i, 1\leq i \leq d \}$, such that $t^{a,b}$ restricted to $\cT^*_{h-1}$ is $s^a$. We denote by $n_a$ the cardinality of $s^i$'s equal to $s^a$ and $k_{a,b}$ the cardinality of $t^i$'s equal to $t^{a,b}$. We set $n'_a = n_a + \IND ( s_a = s')$ and $k'_{a,b} = k_{a,b} + \IND ( t_{a,b} = t')$. Then, Lemma \ref{le:consistency0} yields 
\begin{align}
 \dP_{\rho} &\left( (T,o)_{h+1} = \tau \cup t'_+  | (T,o)_h = t \cup s'_+ \right)  = \prod_a  { n'_a \choose  (k'_{a,b})_{b}  } \prod_{b} \wP_{s^a,s^{-a}} (t^{a,b})^{k'_{a,b}} \nonumber\\
 & \qquad\qquad \qquad\qquad= \frac{n +1}{k+1} \wP_{s',s} (t')  \prod_a  { n_a \choose  (k_{a,b})_{b}  } \prod_{b} \wP_{s^a,s^{-a}} (t^{a,b})^{k_{a,b}} , 
\label{eq:nk1}
\end{align}
where,  $s^{-a} = [(t\cup s'_+)^{-a}]_{h-1}$ and $(t\cup s'_+)^{-a} \in \cT^*_{h}$ is the tree obtained from $t \cup s'_+$ by removing one of the offspring with subtree equal to $s^{a}$. 
Thus, we find
\begin{equation}
\label{eqQQ}
Q ( \tau \cup t'_+ )  =  P (  t \cup s'_+ ) \frac{n +1}{k+1} \wP_{s',s} (t')  \prod_a  { n_a \choose  (k_{a,b})_{b}  } \prod_{b} \wP_{s^a,s^{-a}} (t^{a,b})^{k_{a,b}} .
\end{equation}
Since $\r_{h+1}=\r'_{h+1}$, one has 
$$
e_Q(t,t')  =  e_Q( t',t)  =   \dE_{\rho} \big| \big\{ v \stackrel{T}{\sim} o : T(o,v)_{h} =  t' \; , T(v,o)_{h} =  t \big\}\big|.
$$
By sampling the $h$-neighborhood $(T,o)_h$ first, and using the number $n$ as above, one has
\begin{equation}
\label{eqQQ1}
e_Q(t,t')  =  (n+1) P ( t \cup s'_+ )  \wP_{s',s} (t').
\end{equation}
From \eqref{eqQQ} and \eqref{eqQQ1} we find 
\begin{equation}
\label{eqQQ10}
\wQ_{t,t'} (\tau)  =\frac{(k+1)Q ( \tau \cup t'_+ ) }{e_Q(t,t')  } =\prod_a  { n_a \choose  (k_{a,b})_{b}  } \prod_{b} \wP_{s^a,s^{-a}} (t^{a,b})^{k_{a,b}}. 
\end{equation}
Next, we show that the right hand side in 
\eqref{eq:PwhP} equals the above expression. 
We have
\begin{align*}
\frac { \wdP_{s'} (  (T,o)_{h+1} = \tau )} {\wdP_{s'}(  (T,o)_{h} = t ) } & =  \frac{ \dE_{\rho} \big| \big\{ v \stackrel{T}{\sim} o : T(o,v)_{h+1} =  \tau \; , T(v,o)_{h-1} =  s' \big\}\big|}{  \dE_{\rho} \big| \big\{ v \stackrel{T}{\sim} o : T(o,v)_{h} =  t\; , T(v,o)_{h-1} =  s' \big\}\big|} \\
& =  \frac{\dE_{\rho} \big| \big\{ v \stackrel{T}{\sim} o : T(v,o)_{h+1} =  \tau \; , T(o,v)_{h-1} =  s' \big\}\big|} {  (n+1) P( t \cup s'_+ )},
\end{align*}
where we have used unimodularity and \eqref{eq:qq2}. Now, by sampling first the $h$-neighborhood $(T,o)_h$, one finds that 
\begin{align*}
\dE_{\rho} &\big| \big\{ v \stackrel{T}{\sim} o : T(v,o)_{h+1} =  \tau \; , T(o,v)_{h-1} =  s' \big\}\big|\\&
\qquad =P (  t \cup s'_+  ) 
\sum_{t' : t'_{h-1} = s'}\dE_{\rho} \Big[\sum_{ v \stackrel{T}{\sim} o} \IND( T(v,o)_{h+1} =  \tau,\, T(o,v)_{h} =  t') \,\big| \,(T,o)_h = t \cup s'_+\Big]
\\&\qquad =
P (  t \cup s'_+  )\sum_{t' : t'_{h-1} = s'}  (k+1)\dP_{\rho} \left( (T,o)_{h+1} =  \tau \cup t'_+ \, |\, (T,o)_h = t \cup s'_+ \right),
\end{align*}
where, as before, $k=k(t')$ stands for the number of $v \stackrel{\tau}{\sim} o$ such that  $\tau (o,v) = t'$. 
Using 
Lemma \ref{le:consistency0} in the form \eqref{eq:nk1}, and the fact that $\sum_{t' : t'_{h-1} = s'}\wP_{s',s} (t')=1$, we find
\begin{align*}
\dE_{\rho} &\big| \big\{ v \stackrel{T}{\sim} o : T(v,o)_{h+1} =  \tau \; , T(o,v)_{h-1} =  s' \big\}\big|\\&
\qquad =(n+1) P (  t \cup s'_+  ) \prod_a  { n_a \choose  (k_{a,b})_{b}  } \prod_{b} \wP_{s^a,s^{-a}} (t^{a,b})^{k_{a,b}}.
\end{align*}
Hence,
\begin{equation}
\label{eqQQ11}
\frac { \wdP_{s'} (  (T,o)_{h+1} = \tau )} {\wdP_{s'}(  (T,o)_{h} = t ) } = \prod_a  { n_a \choose  (k_{a,b})_{b}  } \prod_{b} \wP_{s^a,s^{-a}} (t^{a,b})^{k_{a,b}}.
\end{equation}
The identity \eqref{eq:PwhP} follows from \eqref{eqQQ11} and \eqref{eqQQ10}.
\end{proof}

\begin{remark}\label{boh}
From 
\eqref{eqQQ1}
one deduces the identity 
$$
e_Q(t,t')=e_P(s,s')\wP_{s,s'}(t)\wP_{s',s}(t'),
$$
for any $t,t'
\in\cT^*_{h}$, with $s=t_{h-1},s'=t'_{h-1}$, for any $P\in\cP(\cT^*_h)$ admissible, with $Q=[\UGW_{h}(P)]_{h+1}$. 

\end{remark}

\section{Configuration model for directed graphs with colored edges} \label{CM}
This section introduces a generalized configuration model, to be used later on to count the number of graphs with a given tree-like neighborhood distribution. 

\subsection{Directed multi-graphs with colors}
\label{subsec:defCM}
We are now going to define a family of directed multi-graphs with colored edges. Let $L$ be a fixed integer. Each pair $(i,j)$ with $ 1 \leq i , j \leq L$ is interpreted as a color.  
Define 
the sets of  colors 
$$\cC = \BRA{ (i,j) : 1 \leq i , j \leq L  }$$
$$\cC_{<} = \BRA{ (i,j) : 1 \leq i  < j \leq L  } \,,\qquad
\cC_{=} = \BRA{ (i,i) : 1 \leq i  \leq L  }.$$
Also, define  $\cC_\leq= \cC_<\cup\cC_=$, $\cC_>=\cC\setminus \cC_\leq$ and $\cC_{\neq}= \cC\setminus \cC_=$.
If $c = (i,j) \in \cC$, then set $\bar c = (j,i)$ for the conjugate color. 

We consider the class $\wcG(\cC)$ of  directed  
multi-graphs with $\cC$-colored edges defined as follows. We say that a directed  
multi-graph $G$ is an element of $\wcG(\cC)$ if $G = (V, \omega)$ where $V=[n]$ for some $n\in\bbN$, $\omega=\{\omega_c\}_{c\in\cC}$ and for each $c\in\cC$, $\omega_c$ is a map $\omega_c : V^2 \to \dZ_+$ with the following properties:
if $c\in \cC_=$, then $\omega_c (u,u)$ is even for all $u\in V$, and $\omega_c (u,v) = \omega_{ c} (v,u)$ for all $u,v\in V$; if $c\in \cC_{\neq}$, then  $\omega_c (u,v) = \omega_{ \bar c} (v,u)$ for all $u,v\in V$.
The interpretation is that, for any $c\in\cC$, if $u\neq v$ then $\omega_c (u,v)$ is the number of directed edges of color $c$ from $u$ to $v$; if $u=v$ and $c\in\cC_=$, then $\frac12 \omega_c (u,u)$ is the number of loops of color $c$ at $u$, while if  $u=v$ and $c\in\cC_{<}$ then $\omega_c (u,u)=\omega_{\bar c} (u,u)$ is  the number of loops of color $c$ at $u$; we adopt the convention that there are no loops of color $c\in\cC_>$ at any vertex. We call $\cG(\cC)$ the subset of $ \wcG(\cC)$ consisting of  graphs, i.e.\ $G=(V,\omega)$ such that $\omega_c(u,v)\in\{0,1\}$ for all $c\in\cC$ and $u,v\in V$ (no multiple edges) and $\omega_c(u,u)=0$ for all $c\in\cC$ and $u\in V$ (no loop).
See Figure \ref{ocmg} for an example of an element of $\wcG(\cC)$.

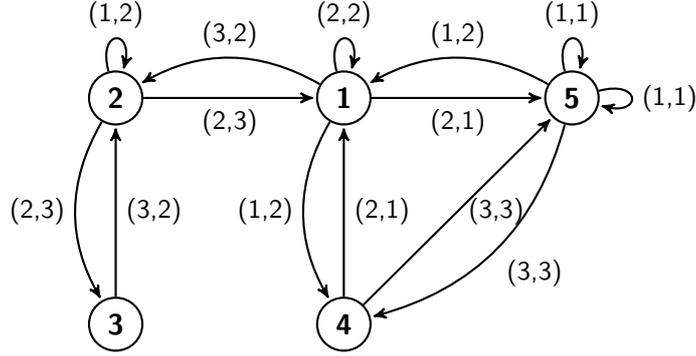
\begin{figure}
\begin{center}
\begin{tikzpicture}[->,>=stealth',shorten >=1pt,auto,node distance=3cm,
                    thick,main node/.style={circle,draw,font=\sffamily
                    \bfseries}]

  \node[main node] (1) {1};
  \node[main node] (2) [left of=1] {2};
  \node[main node] (3) [below  of=2] {3};
  \node[main node] (4) [right of=3] {4};
   \node[main node] (5) [right of=1] {5};
   
  \path[every node/.style={font=\sffamily\small}]
    (1) 
        edge [bend right] 
        node[above] {(3,2)} (2)
        edge [loop above] node {(2,2)} (1)
        edge [bend right] node[left] {(1,2)} (4)
        edge node [below] {(2,1)} (5)
    (2) edge node [below] {(2,3)} (1)
        edge [loop above] node {(1,2)} (2)
        edge [bend right] node[left] {(2,3)} (3)
        
    (3) 
        edge node[right] {(3,2)} (2)
    (4) edge node [right] {(2,1)} (1)
        edge node[ right] {(3,3)} (5)
        (5) 
        edge  [bend left] node [below right]{(3,3)} (4)
        edge [bend right] 
        node[above] {(1,2)} (1)
        edge [loop above] node {(1,1)} (5)
        edge [loop right] node {(1,1)} (5)
    ;
\end{tikzpicture}
\end{center}
\caption{An example of directed colored multi-graph in $\wcG(\cC)$. Here $n=5$, $L=3$, with: 
$\omega_{(2,1)}(4,1)=
\omega_{(1,2)}(1,4)=\omega_{(2,1)}(1,5)=\omega_{(1,2)}(5,1)=1$; $\omega_{(2,3)}(2,3)=
\omega_{(3,2)}(3,2)=\omega_{(2,3)}(2,1)=\omega_{(3,2)}(1,2)=1$; $\omega_{(3,3)}(4,5)=
\omega_{(3,3)}(5,4)=1$; $\omega_{(2,2)}(1,1)=2$; $\omega_{(1,1)}(5,5)=4$; $\omega_{(2,2)}(1,1)=2$;
$\omega_{(1,2)}(2,2)=\omega_{(2,1)}(2,2)=1$; all other entries of $\omega$ are zero.}
\label{ocmg}
\end{figure}

 If $G\in\wcG(\cC)$, one can define the colorblind 
 multi-graph 
 $\bar G=(V,\bar\omega)$, by setting
 \begin{equation}\label{colorblind}
 \bar\omega(u,v)=\sum_{c\in\cC}\omega_c(u,v).
\end{equation}  
The multi-graph 
 $\bar G=(V,\bar\omega)$ can be identified with an undirected multi-graph, in that by construction $\bar\omega(u,v)=\bar\omega(v,u)$ for all $u,v\in V$. We say that $G$ is a simple graph if $\bar G$ has no loops  and no multiple edges. 
 Clearly, if $L=1$ then there is only one color, so that any  multi-graph 
 $G\in\wcG(\cC)$ coincides with its own $\bar G$.
 
 If $G\in\wcG(\cC)$, $c\in\cC$ and $u\in V$, set
\begin{equation}\label{dcu}
D_c (u) = \sum_v \omega_c(u,v),
\end{equation}  
and write $D(u)=\{D_c(u),c\in\cC\}$.  Note that  $D(u)$ is an element of $\cM_L $, defined as the set of $L\times L$ matrices  with nonnegative integer valued entries. The vector $\bD = \{D(u), u\in V\}$ of such matrices will be called the {\em degree sequence} of $G$.

 \subsection{
 Directed colored multi-graphs with given degree sequence}\label{configmod}
   Fix $n\in\bbN$, and let $\cD_n$ denote
 the set of all vectors $(D(1),\dots,D(n))$ such that 
 $D(i)=\{D_c(i),\, c\in\cC\}\in \cM_L $ for all $i\in[n]$, and such that 
\begin{equation}\label{dS}
S = \sum_{i = 1}^n D(i) 
\end{equation}
is a symmetric matrix with even coefficients on the diagonal, i.e.\ 
$S=\{S_c,\, c\in\cC\}$, $S_c=S_{\bar c}$ for all  $c\in\cC$, and $S_c\in 2\dZ_+$ for all $c\in\cC_=$. Clearly, if $G\in\wcG(\cC)$ then the vector $\bD$ defined by \eqref{dcu}
yields an element of $\cD_n$ for some $n$. 
Next, for a given $\bD\in\cD_n$ we consider the set of all elements of $\wcG(\cC)$
 which have $\bD$ as their degree sequence.

\begin{definition}\label{defgd}
Fix $n\in\bbN$ and $\bD\in\cD_n$. \begin{itemize}
\item
$\wcG(\bD)$ is the set of multi-graphs $G\in\wcG(\cC)$
with $V=[n]$ such that the degree sequence of $G$ defined by \eqref{dcu} coincides with $\bD$.


\item 
$\cG(\bD,h)$ is the set of  $G\in\wcG(\bD)$ such that the colorblind graph $\bar G$ defined in \eqref{colorblind} contains no cycle of length $\ell\leq h$.
\end{itemize}
\end{definition} 

We also use the notation $\cG(\bD)$ for the set  of simple graphs in $ \wcG(\bD)$. 
This set coincides with $\cG(\bD,2)$, since loops are cycles with length $1$ and multiple edges are cycles with length $2$. 
The main goal of this section is to provide asymptotic formulas for the cardinality of $\cG(\bD)$ and more generally of $\cG(\bD,h)$, for any $h\in\dN$. To this end we introduce a natural extension of the 
usual configuration model from \cite{Bollobas}, see also \cite{MR1782847}.

Fix a multi-graph $G\in\wcG(\bD)$.
For a fixed $c\in\cC_=$, let $G_c$ denote the subgraph of $G$ obtained by removing all edges but the ones with color $c$. 
 If $c\in\cC_<$ instead, then define $G_c$ as the subgraph of $G$ obtained by removing all edges but the ones with color $c$ or $\bar c$.  
 Thus, every $G\in\wcG(\bD)$ is the result of the superposition of the multi-graphs $G_c$, $c\in\cC_\leq$. We may then analyze each color separately. 
 
\subsubsection{Configuration model for $c\in\cC_=$} When $c\in\cC_=$, every pair $u,v$ satisfies $\o_c(u,v)=\o_c(v,u)$, so $G_c$ is actually a multi-graph
 with undirected edges, and we may use the usual construction \cite[Section 2.4]{Bollobas}.
We provide the details for completeness.
  The degrees of $G_c$ are fixed by the sequence $D_c(1),\dots,D_c(n)$. 
 Let $W_c=\cup_{i=1}^nW_c(i)$, be a fixed set of $S_c=\sum_{i=1}^nD_c(i)$ points, with the subsets $W_c(i)$ 
satisfying $|W_c(i)|=D_c(i)$. 
Recall that $S_c$ is even by assumption. Let $\Sigma_c$ be the set of all perfect matchings of
  the complete graph over the points of $W_c$, i.e.\ the set of all partitions of $W_c$ into disjoint edges. Then,
 $$|\Sigma_c|=(S_c-1)!!=(S_c-1)(S_c-3)\cdots 1 = \frac{S_c !}{ (S_c/2) ! 2^{S_c/2}}.$$
Elements of $\Sigma_c$ are  called configurations. For any configuration $\si_c\in \Sigma_c$, call $\G(\si_c)$ the multi-graph on $[n]$ with undirected edges obtained
 by including an edge $\{i,j\}$ iff $\si_c$ has a pair with one element in $W_c(i)$ and the other in $W_c(j)$.
 Notice that $\G(\si_c)$ has the same degree sequence $D_c(1),\dots,D_c(n)$ of $G_c$. 
 Moreover, any multi-graph with that degree sequence equals $\G(\si_c)$ for some $\si_c\in\Sigma_c$. 
 
 \begin{lemma}\label{le:match}
 Fix $c\in\cC_=$.
Let $H$ be a multi-graph on $[n]$ with undirected edges and with degree sequence $D_c(1),\dots,D_c(n)$.
The number of $\si_c\in\Sigma_c$ such that $\G(\si_c)=H$ is given by
\begin{equation}\label{match1}
n_c(H)=\frac{\prod_{i=1}^n D_c(i)!}{\prod_{i=1}^n\left(\o_c(i,i)/2\right)!2^{\left(\o_c(i,i)/2\right)}\prod_{i<j}\o_c(i,j)!}
\end{equation}
where $\o_c(i,j)$ is the number of edges between nodes $\{i,j\}$ in $H$, while $\o_c(i,i)/2$ is the number of loops at node $i$ in $H$. 
\end{lemma} 
\begin{proof}
We need to count the number of matchings $\si_c\in\Si_c$ such that for every $i<j$ one has 
$\o_c(i,j)$
edges between $W_c(i)$ and $W_c(j)$, and such that for all $i$ one has $\frac12\o_c(i,i)$ edges 
within $W_c(i)$. 
Fix $i<j$. Once we choose the $\o_c(i,j)$ elements of $W_c(i)$ and the  $\o_c(i,j)$ elements of $W_c(j)$ to be matched
together to produce the $\o_c(i,j)$ edges, then there are
$\o_c(i,j)!$ distinct matchings that produce the same graph. Similarly,  once we fix the $\o_c(i,i)$ elements of $W_c(i)$ to be matched
together to produce the $\frac12\o_c(i,i)$ loops at $i$, then there are
$(\o_c(i,i)-1)!!$ distinct matchings that produce the same graph. On the other hand, for every node $i$ there are
$$
\binom{D_c(i)}{\o_c(i,1),\dots,\o_c(i,n)}=\frac{D_c(i)!}{\o_c(i,1)!\cdots\o_c(i,n)!}
$$
distinct ways of choosing the elements of $W_c(i)$ to be matched with $W_c(1),\dots,W_c(n)$ respectively. 
Putting all together we arrive at the following expression for the total number of configurations producing the graph $H$:
$$
\prod_{i=1}^n \binom{D_c(i)}{\o_c(i,1),\dots,\o_c(i,n)} \prod_{i< j}\o_c(i,j)!\prod_{i=1}^n(\o_c(i,i)-1)!!,
$$
which can be rewritten as \eqref{match1}.
\end{proof}

\subsubsection{Configuration model for $c\in\cC_<$} 
When $c\in\cC_<$, every pair $u,v$ satisfies $\o_c(u,v)=\o_{\bar c}(v,u)$, so 
for the multi-graph $G_c$, $D_c(i)$ represents the number of outgoing edges at node $i$, which equals the number of incoming edges at that node. Here we use a bipartite version of the previous construction. 
Let $W_c=\cup_{i=1}^nW_c(i)$, be a fixed set of $S_c=\sum_{i=1}^nD_c(i)$ points, with the subsets $W_c(i)$ 
satisfying $|W_c(i)|=D_c(i)$. Similarly, set $\bar W_c=\cup_{i=1}^n\bar W_c(i)$, with $|\bar W_c(i)|=D_{\bar c}(i)$.
Consider the set $\Sigma_c$ of all perfect matchings of the complete bipartite graph over the sets $(W_c,\bar W_c)$, i.e.\ the set of perfect matchings containing only  edges connecting an elements of $W_c$ with an element of $\bar W_c$. Since $S_c=S_{\bar c}$, one has $|W_c|=|\bar W_c|$, and $\Sigma_c$ can be identified with the set of permutations of $S_c$ objects, or the set of bijective maps $W_c\mapsto\bar W_c$, and $|\Sigma_c|=S_c !$. 
A configuration is an element $\si_c\in\Sigma_c$. For any configuration $\si_c$, 
let $\G(\si_c)$ denote the directed multi-graph on $[n]$  obtained
 by including the directed edge $(i,j)$ with color $c$ and the edge $(j,i)$ with color $\bar c$ iff $\si_c$ has a pair with one element in $W_c(i)$ and the other in $\bar W_c(j)$.
 Notice that $\G(\si_c)$ has the same degree sequence $D_c(1),\dots,D_c(n)$ of $G_c$, and 
 any multi-graph with directed edges with colors with the same degree sequence equals $\G(\si_c)$ for some $\si_c\in\Sigma_c$. 
 \begin{lemma}\label{le:bipmatch}
 Fix $c\in\cC_<$.
Let $H$ be a multi-graph on $[n]$ with directed edges with colors $(c,\bar c)$ only 
and with degree sequence 
$D_c(1),\dots,D_c(n)$.
The number of $\si_c\in\Sigma_c$ such that $\G(\si_c)=H$ is given by
\begin{equation}\label{match3}
n_c(H)=\frac{ \prod_{i=1}^n D_c(i)!
D_{\bar c}(i)!}{\prod_{i,j}\o_c(i,j)!}
\end{equation}
where $\o_c(i,j)=\o_{\bar c}(j,i)$ is the number of edges from $i$ to $j$ with color $c$ in $H$. 
\end{lemma} 
\begin{proof}
We have to count the number of bijective maps $W_c\mapsto\bar W_c$ such that 
for every $i,j\in[n]$ (including the case $i=j$), $\o_c(i,j)$ elements of $W_c(i)$ are mapped to $\bar W_c(j)$.
We begin by choosing, for every fixed node $i$, the subsets of $W_c(i)$ that are mapped into $\bar W_c(k)$, $k=1,\dots,n$, and
the subsets of $\bar W_c(i)$ that are mapped into $W_c(k)$, $k=1,\dots,n$. This  can be done in 
$$
\prod_{i=1}^n\binom{D_c(i)}{\o_c(i,1),\dots,\o_c(i,n)}\binom{D_{\bar c}(i)}{\o_{\bar c}(i,1),\dots,\o_{\bar c}(i,n)}
$$
 distinct ways. Once these subsets are chosen there remain, for every $i,j$,  $\o_c(i,j)!$
distinct bijections producing the same graph. Therefore, the total number of bijections from $W_c$ to $\bar W_c$ which preserve the numbers $\o_c(i,j)=\o_{\bar c}(j,i)$ is given by 
$$
\prod_{i=1}^n
\binom{D_c(i)}{\o_c(i,1),\dots,\o_c(i,n)}\binom{D_{\bar c}(i)}{\o_{\bar c}(i,1),\dots,\o_{\bar c}(i,n)}
\prod_{i,j}\o_c(i,j)!
$$
The latter expression can be rewritten as \eqref{match3}.
\end{proof}

 \subsubsection{Generalized configuration model}
 We now define the configuration model for a generic degree sequence $\bD\in\cD_n$ by putting together the configuration models for all the colors. 
 Let $\Sigma$ denote the cartesian product of $\Sigma_c,\,c\in\cC_\leq$, where, as defined above, $\Sigma_c$ are the sets of configurations associated to the degree sequence $D_c(1),\dots,D_c(n)$, that is $\Sigma_c$ is the set of matchings of $W_c$ if $c\in\cC_=$ and $\Sigma_c$ is the set of bijections $W_c\mapsto\bar W_c$ if $c\in\cC_<$.   
 A configuration is an element $\si=(\si_c)_{c\in\cC_\leq}$ of $\Sigma$.  
The map $\G(\cdot):\Sigma\mapsto\wcG(\bD)$ is defined by calling $\G(\si)$ the multi-graph obtained by superposition of the multi-graphs $\G(\si_c)$ defined above.  
The {\em configuration model}, denoted $\CM(\bD)$,  is the law of $\G(\si)$ when $\si\in\Sigma$ is chosen uniformly at random. 
\begin{lemma}\label{le:unifCM}
Let $\bD \in \cD_n$, $G$ with distribution $\CM (\bD)$ and $H \in\wcG(\bD)$. We have
\begin{equation}\label{confmod}
\dP \PAR{ G = H } = \frac{\prod_{c\in\cC}\prod_{i=1}^n D_c(i) ! }{b(H) \prod_{c \in \cC_<} S_c ! \prod_{c \in \cC_=} (S_c-1) !!},
\end{equation}
where $S_c = \sum_{i=1}^nD_c(i)$, and $b(H)$ is defined by
\begin{equation}\label{b(H)}
b (H)= 
\prod_{c \in \cC_{<}} \prod_{i,j}\omega_c ( i,j) ! 
\prod_{c \in \cC_=} \prod_{i=1}^n (\omega_c (i,i) /2)!  2^{ (\omega_c (i,i) / 2)}  \prod_{i < j} \omega_c (i,j) !
\end{equation}
In particular, for any $h\geq 2$,  if $\cG(\bD,h)$ is not empty, the law of $G$ conditioned on $\cG(\bD,h)$ is the uniform distribution on $\cG(\bD,h)$.
\end{lemma}
\begin{proof}
 The cardinality of $\Sigma$ is given by $\prod_{c \in \cC_<} S_c ! \prod_{c \in \cC_=} (S_c-1) !!$. Thus, it suffices to check that $\G^ {-1} (H)$ has cardinality 
$b(H)^{-1}\prod_{ c \in \cC}\prod_i D_c(u) !$.
This follows from Lemma \ref{le:match} and Lemma \ref{le:bipmatch} by observing that 
$|\G^ {-1} (H)|=\prod_{c\in\cC_\leq} n_c(H_c)$, where $H_c$ denotes the multi-graph $H$ after all edges with color $c'\notin\{c,\bar c\}$ are removed. This proves \eqref{confmod}. If $H\in \cG(\bD,h)$, $h\geq 2$, then $\o_c(i,i)=0$ and $\o_c(i,j)\in\{0,1\}$
for all $i,j\in[n]$ and $c\in\cC$, so that $b(H)=1$. This proves the last assertion.
\end{proof}

\subsection{Probability of having no cycles of length $\ell\leq h$} \label{H1H2}
Fix $\theta\in\dN$ and call $\cM_L^{(\theta)}$ 
the set of $L\times L$ matrices with nonnegative integer entries bounded by $\theta$. Fix 
$P \in \cP (\cM_L^{(\theta)})$, a probability on $\cM_L^{(\theta)}$. 
We consider a sequence $\bD^{(n)} = ( D^{(n)}(u))_{u\in[n]}\in \cD_n$, $n \geq 1$ such that 

\begin{enumerate}[(H1)]

\item
for all $u \in [n]$,  $D^{(n)}(u)\in\cM_L^{(\theta)}$ ;

\item 
as $n \to \infty$, 
$
\frac 1 n \sum_{u = 1} ^ n \delta_{D^{(n)}(u)} \weak P.
$
\end{enumerate}
%
The main result  of this section is the following  
\begin{theorem}\label{th:Poilim}
Fix $\theta\in\dN$, $P \in \cP (\cM_L^{(\theta)})$, and a sequence $\bD^{(n)}$ satisfying (H1)-(H2). 
Take $G_n \in \wcG(\bD^{(n)} )$ with distribution $\CM ( \bD^{(n)})$. 
For every $h\in\dN$, there exists $\a_h>0$ such that 
\begin{equation}\label{nocycleprob}
\lim_{n\to\infty} \dP \PAR{ G_n \in \cG(\bD^{(n)},h ) } = \alpha_h.
\end{equation}

\end{theorem}
The actual value of $\alpha_h$ could be in principle computed in terms of $P$ (see proof of Theorem \ref{th:Poilim}). We will however not need that. 

\begin{corollary}\label{cor:Poilim}
In the setting of Theorem \ref{th:Poilim}, writing $S_c^{(n)}=\sum_{u\in[n]}D_c^{(n)}(u)$, for all $h\geq 2$:
$$
| \cG(\bD^{(n)},h ) | \sim \alpha_h  \,\frac{ \prod_{c \in \cC_<} S^{(n)}_c ! \prod_{c \in \cC_=} (S^{(n)}_c-1) !!}{ \prod_{c \in \cC}\prod_u D^{(n)}_c(u)! }.
$$
\end{corollary}

\begin{proof}
By definition of $\CM(\bD^{(n)})$, one has 
$$ \dP \PAR {G_n \in \cG(\bD^{(n)},h)} = \frac1{ |\Sigma | }
\sum_{\si \in \Sigma} \IND \left( \G (\si) \in \cG(\bD^{(n)},h  ) \right).$$ 
As in Lemma \ref{le:unifCM}, for each $H \in \cG(\bD^{(n)},h) $, $ |\G^ {-1} (H)| = \prod_{c \in \cC}\prod_u D^{(n)}_c(u) !$. Hence the sum in the right hand side  above equals $ | \cG(\bD^{(n)},h ) | \prod_{c \in \cC}\prod_u D^{(n)}_c(u) !$. The conclusion follows from Theorem \ref{th:Poilim} and $|\Sigma|=\prod_{c \in \cC_<} S^{(n)}_c ! \prod_{c \in \cC_=} (S^{(n)}_c-1) !!$. 
\end{proof}

The proof of Theorem \ref{th:Poilim} will follow 
a well known strategy; see e.g. Bollob\'as \cite[proof of Theorem 2.16]{Bollobas} for a similar result. 
Our first lemma computes the number of copies of a subgraph in a graph sampled from $\CM ( {\bD}^{(n)})$. To formulate it, we need to introduce some more notation. Let $\wcG_n$ denote the set of 
$G\in\wcG(\cC)$ with vertex set $[n]$. If $G\in\wcG_n$, and $H\in\wcG(\cC)$ has vertex set $V\subset [n]$, we let $Y(H , G)$ be the number of times that $H \subset G$. When $G$ is not a simple graph, then $Y(H,G)$ may be larger than $1$.
Indeed, one has
\begin{align}
\label{yhg}
Y(H , G) = \IND(H\subset G)
\prod_{c\in\cC_<} 
\prod_{u, v}B_c^{H,G}(u,v)
\prod_{c\in\cC_=} 
\prod_{u\leq v}B_c^{H,G}(u,v)
\end{align}
where we use the notation $B_c^{H,G}(u,v)$ for the binomial coefficient
$\binom{\o_c^G(u,v)}{\o_c^H(u,v)}$, with the convention that if $u=v$ and $c\in\cC_=$, then $B_c^{H,G}(u,u)$ equals 
$\binom{(\o_c^G(u,u)/2)}{(\o_c^H(u,u)/2)}$.


Next, for $G\in\wcG_n$ and $H\in \wcG_k$, $1 \leq k \leq n$, define $X(H,G)$ as the number of distinct subgraphs of $G$ that are isomorphic to $H$. 
If $a(H)$ denotes the cardinality of the automorphism group of $H$,  
i.e.\ the number of permutations of the vertex labels which leave $H$ invariant, then
\begin{equation}
\label{eq:Xauto}
X ( H , G) =  \frac1{a(H)} \sum_\tau  Y (  \tau(H) , G ),
\end{equation}
where the sum is over all injective maps $\t$ from $[k]$ to $[n]$, and $\t(H)$ represents the multi-graph obtained by embedding $H$ in $[n]$ through $\t$.

For $H\in \wcG_k$, the $c$-degree at vertex $u$ is denoted  
$$
d_c^H(u)=\sum_v\o^H_c(u,v)
\,.
$$
The {\em excess} of $H$ is defined by
$$
\exc (H) = 
\PAR{\frac 1 2  \sum_{c \in \cC} \sum_{i =1 } ^k d^H_c (i) } - k ,
$$
Notice that $\exc(H)= |E(H)| - k$, where $E(H)$ is the total number of edges of $\bar H$ (counting $1$ for each loop) where $\bar H$ is the colorblind undirected multi-graph obtained from $H$ by \eqref{colorblind}. 
Notice that for $H$ connected, then $\exc (H) \geq -1$, and $\exc (H) = -1$ iff $\bar H$ is a tree. 
If $n\geq k$ are positive integers, we use the notation $(n)_k = n!/(n-k)!$ for the number of injective maps $[k]\mapsto[n]$, with $(n)_0 = 1$. 
\begin{lemma}\label{le:formulacount}
Let $G_n \in \wcG(\bD^{(n)} )$ with distribution $\CM ( {\bD}^{(n)})$, where $\bD^{(n)}$ satisfies assumptions (H1)-(H2). 
For any fixed $k\in\dN$, $H\in\wcG_k$, as $n \to \infty$:
$$
\dE X ( H, G_n) \sim \frac{ \prod_{i = 1} ^ k \dE\prod_{c\in\cC} (D_c)_{d_c^H(i)} }{a(H) b(H) \prod_{c\in\cC}  \PAR{ \dE  D_c}^{s_c^H/2}} \,\,n^{-\exc (H)},
$$
where $D\in\cM^{(\theta)}_L$ has distribution $P$ and $s^H_c:=\sum_{i=1}^kd_c^H(i)$. 
\end{lemma}

\begin{proof}
From \eqref{eq:Xauto}, 
$\dE X ( H , G_n) =  a(H)^{-1} \sum_\tau  \dE Y (  \tau(H) ,  G_n )$. 
Below, we fix a map $\t$ and write $H$ instead of $\t(H)$ for simplicity.
We start by showing that 
\begin{equation}
\label{eq:EYHG}
\dE Y(  H ,  G_n ) = \frac{ \prod_{c \in \cC} \prod_{i= 1} ^k (D^{(n)}_c(i))_{d_c^H(i)}  }{ b(H)\, \prod_{c \in \cC_<} (S^{(n)}_c)_{s^H_c} \prod_{c \in \cC_{=}} \lp S^{(n)}_c  \rp_{s^H_c}},
\end{equation}
where we use the notation  $\lp n  \rp_{k} = (n-1)!!/(n-k-1)!!$.
Since $\CM(\bD^{(n)})$ is a product measure over $c\in \cC_\leq$, we may analyze one color at a time. 

Consider first the case $c\in\cC_<$. 
Set
$$Y_c(H,G)= \IND(H_c\subset G)\prod_{u, v}\binom{\o_c^G(u,v)}{\o_c^H(u,v)},$$
where $H_c$ is the graph $H$ with all edges removed except for edges of color $c$ or $\bar c$, and the condition $G\supset H_c$ indicates that $\o_c^G(u,v)\geq \o_c^H(u,v)$ for all $u,v\in[n]$.  
Then, as in Lemma \ref{le:unifCM}
$$
\dE Y_c(H,G_n) = \sum_{G:\;G\supset H_c}
\frac{\prod_{i=1}^n D_c(i) !D_{\bar c}(i)! }{S_c!\prod_{i,j}\o_c^G(i,j)!}\prod_{u, v}\binom{\o_c^G(u,v)}{\o_c^H(u,v)}
$$
where we drop the superscript $(n)$ from $D_c(i)$ and $S_c$, and the sum runs over all $G\in\wcG_n$ with $(c,\bar c)$ colors only,
with degree sequence given by $(D_c(i),D_{\bar c}(i))_{i\in[n]}$.
Therefore,
$$
\dE Y_c(H,G_n) = \frac{\prod_{i=1}^n D_c(i) !D_{\bar c}(i)! }{S_c!\prod_{i, j}\o_c^H(u,v)!}
\sum_{G:\;G\supset H_c}
\prod_{u, v}\frac1{(\o_c^G(u,v)-\o_c^H(u,v))!}
$$
On the other hand, applying \eqref{match3} to the multi-graph $G\!\smallsetminus\! H$ defined by $(\o_c^G(u,v)-\o_c^H(u,v))$,  one has 
$$
\sum_{G:\;G\supset H_c}
\prod_{u, v}\frac1{(\o_c^G(u,v)-\o_c^H(u,v))!} = \frac{(S_c - s^H_c)!}{\prod_{i}(D_c(i)-d_c^H(i))!(D_{\bar c}(i)-d_{\bar c}^H(i))!}  
$$
Thus, for $c\in\cC_<$ one has
\begin{equation}
\label{eq:EYHGc1}
\dE Y_c(  H ,  G_n ) = \frac{\prod_{i}  (D^{(n)}_c(i))_{d_c^H(i)}  (D^{(n)}_{\bar c}(i))_{d_{\bar c}^H(i)}}{ (S_c)_{s^H_c} \prod_{i, j}\o_c^H(u,v)!}. 
\end{equation}
Next, consider the case $c\in\cC_=$. Here
$$Y_c(H,G)= \IND(H_c\subset G)\prod_{u< v}\binom{\o_c^G(u,v)}{\o_c^H(u,v)}\prod_{u}\binom{(\o_c^G(u,u)/2)}{(\o_c^H(u,u)/2)},$$
where $H_c$ is the graph $H$ with all edges removed except for edges of color $c$.  
Then, 
$$
\dE Y_c(H,G_n) = \sum_{G:\;G\supset H_c}
\frac{\prod_{i} D_c(i) ! \prod_{u< v}\binom{\o_c^G(u,v)}{\o_c^H(u,v)}\prod_{u}\binom{(\o_c^G(u,u)/2)}{(\o_c^H(u,u)/2)}}{(S_c-1)!!\prod_{i<j}\o_c^G(i,j)!\prod_{i}\left(\o_c^G(i,i)/2\right)!2^{\left(\o^G_c(i,i)/2\right)}}
$$
Applying \eqref{match1} to the multi-graph $G\!\smallsetminus\! H$ and simplifying, one arrives at
\begin{equation}
\label{eq:EYHGc2}
\dE Y_c(  H ,  G_n ) = \frac{\prod_{i} (D^{(n)}_c(i))_{d_c^H(i)} }{ ((S_c))_{s^H_c} \prod_{i< j}\o_c^H(i,j)!\prod_{i}\left(\o_c^H(i,i)/2\right)!2^{\left(\o^H_c(i,i)/2\right)}}. 
\end{equation}
Finally, taking products over $c\in\cC_<$ of \eqref{eq:EYHGc1} together with products over $c\in\cC_=$ of \eqref{eq:EYHGc2}, we arrive at \eqref{eq:EYHG}.  

Summing over the injective maps $\t:[k]\mapsto[n]$, we deduce that 
\begin{equation}\label{eq:formulaCMXH} 
\dE X ( H , G_n)  =
\frac{(n)_k\,\dE  \prod_{c\in\cC}\prod_{i= 1} ^k (M_c(i))_{ d_c^H(i) }} {a(H) b(H) \, \prod_{c \in \cC_<} (S^{(n)}_c)_{s^H_c} \prod_{c \in \cC_{=}} \lp S^{(n)}_c  \rp_{s^H_c} }  , 
\end{equation}
where $(M(1), \cdots , M(k) )$ is uniformly sampled without replacement on $(\bD^{(n)}(1),\dots,\bD^{(n)}(n))$. 
From assumptions (H1)-(H2), for every fixed $k$ and $H\in\wcG_k$, as $n\to\infty$:
$$\dE  \prod_{c\in\cC}\prod_{i= 1} ^k (M_c(i))_{ d_c^H(i) }\to \prod_{i = 1} ^ k \dE \prod_{c\in\cC}(D_c)_{d^H_c(i)}
,$$
where $D\in\cM_L^{(\theta)}$ has law $P$.
Moreover, for $c \in \cC_<$ and $c \in \cC_{=}$ respectively, 
$$
 (S^{(n)}_c)_{s^H_c} \sim n^{s^H_c} (\dE D_c )^{s^H_c}  
 \quad \hbox{ and }  \quad \lp S^{(n)}_c  \rp_{s^H_c} \sim n^{s^H_c/2}  (\dE D_c )^{s^H_c/2}. 
$$
The desired conclusion now follows by  using these asymptotics in \eqref  {eq:formulaCMXH} together with $(n)_k\sim n^k$ and 
$$\sum_{c \in \cC_<} s^H_c + \frac 1 2 \sum_{c \in \cC_=} s^H_c= \frac 1 2  \sum_{c \in \cC} s^H_c  = \exc(H) + k.$$
\end{proof}

\begin{proof}[Proof of Theorem \ref{th:Poilim}]
For every $\ell\in\dN$, call $\cL_\ell$ the set of all $H\in\wcG_\ell$ such that 
the undirected graph $\bar H$ defined by \eqref{colorblind} is a cycle of length $\ell$. 
If $\ell=1$, then $\cL_\ell$ is the union over $c\in\cC$ of the 
single loop graph at vertex $\{1\}$ with color $c$, 
if $\ell=2$, then $\cL_\ell$ is the union over $c,c'\in\cC$
of the double edge graph at vertices $\{1,2\}$ with $\o_c(1,2)=\o_{\bar c}(2,1)=1,\o_{c'}(1,2)=\o_{\bar c'}(2,1)=1$, and so on.   
Let $\cL_{\leq h}=\cup_{\ell=1}^h \cL_\ell$. We define the random variable 
\begin{equation}\label{eq:Z1} 
Z = 
\sum_{H\in\cL_{\leq h}} X(H,G_n),
\end{equation}
where $G_n \in \wcG(\bD^{(n)} )$ has distribution $\CM ( {\bD}^{(n)})$.
With this notation, we need to show that under the assumptions of the theorem one has 
\begin{equation}\label{eq:Z2}
 \lim_{n\to\infty}\dP(Z=0)= \a_h,
 \end{equation} 
 for some $\a_h>0$.
 
If $H\in\cL_{\leq h}$, then $\exc(H)=0$. By Lemma \ref{le:formulacount}, for some $\lambda_H \geq 0$, as $n \to \infty$, one has
\begin{equation}\label{eq:Z3}
\dE X ( H , G_n)  \to \lambda_H,
\end{equation}
and, setting $\lambda(h)  = \sum_{H \in \cL_{\leq h}}\lambda_H$, one finds  
\begin{equation}\label{eq:Z4}
 \lim_{n\to\infty}\dE Z = \lambda(h). 
\end{equation}
We are going to prove that $Z$ converges weakly to a Poisson random variable with mean $\lambda(h)$. This will prove \eqref{eq:Z2} with $\alpha_h = e^{-\lambda(h)}$.
To this end, by the well known moment method, it is sufficient to prove that for any integer $p\geq 1$:\begin{equation}\label{eq:momentZp}
 \lim_{n\to\infty}\dE \left[(Z)_p\right]   =\lambda(h)^ p,
\end{equation}
where $(Z)_p=Z!/(Z-p)!$. 
The case $p=1$ is \eqref{eq:Z4}. Below, we establish \eqref{eq:momentZp} for all $p\geq 2$. 

For any $H\in\cL_{\leq h}$, let $\cH_H$ denote the set 
of multi-graphs $F\in\wcG(\cC)$ with vertex set $V_F \subset [n]$ which are isomorphic to $H$.
If $\cH = \cup_{H\in\cL_{\leq h}} \cH_H$, then one has 
$$
Z = \sum_{ F \in \cH} Y_F,
$$
where $Y_F:=Y(F,G_n)$ is defined by \eqref{yhg}.
The proof of \eqref{eq:momentZp} uses two elementary topological facts:
\begin{enumerate}[(i)]
\item  if $F \ne F' \in \cH$ and $F \cap F' \ne \emptyset$, i.e.\ $V_F \cap V_{F'} \ne \emptyset$, then $\exc ( F \cup F') \geq 1$,
\item  if $H \in \wcG_k$ and $H' \in \wcG_{k'}$, then $\exc (H \cup H') \geq \exc (H) + \exc (H')$ and $\exc (H \oplus H') = \exc (H) + \exc (H')$,  
\end{enumerate}
where $H \oplus H' \in \wcG_{k+k'}$ is the multigraph obtained from the disjoint union of $H$ and an isomorphic copy of $H'$ with vertex set $\{k+1, \cdots, k + k'\}$. We also use two consequences of Lemma \ref{le:formulacount}:
\begin{enumerate}[(i)]
\item[(iii)] if $H \in \wcG_k$ and  $\exc ( H) \geq 1$ then $\dE X(H,G_n) = o(1)$ ; 
\item[(iv)] if $H \in \wcG_k$ and $H' \in \wcG_{k'}$, then $\dE X(H \oplus H',G_n) \sim \dE X(H,G_n)\,\dE X(H',G_n)$.
\end{enumerate}

We start by showing that for all $q\geq 1$, there exists $c=c(q)>0$ such that
\begin{equation}\label{eq:momentZ2}
\dE \left[ Z^{q} \right]  \leq c. 
\end{equation}
Write
$$
Z^{q} = \sum_{(F_1,\cdots, F_q) \in \cH^q}  \prod_{i=1}^q Y_{F_{i}}.
$$
By assumption (H1), $Y_F \leq c_0$ for some $c_0=c_0 (\theta,h)$, and hence, for some $c_1 = c_1 ( \theta,h, q)$,  one has the crude bound
$$
Z^{q} \leq c_1 \sum_{k=1} ^q \sum_{*}  \prod_{i=1}^k Y_{F_{i}},
$$
where the sum $\sum_{*}$ is over all choices of pairwise distinct $F_1,\cdots, F_k$ in $\cH$. 
We now decompose $\sum_{*}$ into the sum $\sum_{**}$ over all choices of $k$ pairwise disjoint sets $F_i$ in $\cH$, and the sum $\sum_{***}$ over all choices of $k$ pairwise distinct $F_i$ in $\cH$ such there exists $i \ne j$ with $F_i \cap F_j \ne \emptyset$. Notice that this last summation satisfies
$$
 \sum_{***}   \prod_{i=1}^k Y_{F_{i}}\leq c_0\sum_K X ( K , G_n)
$$
 for some $c_0=c_0(\theta,h)$,  where $K$ ranges over a finite collection (with cardinality independent of $n$) of multi-graphs which by facts (i-ii) satisfy $\exc(K) \geq 1$. In particular, fact (iii) implies that $\dE \sum_{***}  \prod_{i=1}^k Y_{F_{i}}=o(1)$ as $n\to\infty$. On the other hand, 
$$
\sum_{**} \prod_{i=1} ^ k  Y_{F_i} = \sum_{(H_1,\dots,H_k)\in (\cL_{\leq h})^k} X ( H_1 \oplus  \cdots \oplus  H_k, G_n). 
$$
Fact (iv)  and \eqref{eq:Z3} then imply that 
\begin{equation}\label{sum*yf}
\dE \sum_*  \prod_{i=1} ^ k  Y_{F_i} =  \sum_{(H_1,\dots,H_k)\in (\cL_{\leq h})^k} \prod_{i=1} ^k \lambda_{H_i} + o(1)=\l(h)^k + o(1). 
\end{equation}
This ends the proof of \eqref{eq:momentZ2}.

Next, 
define $\tilde Y_F = \IND ( Y_F = 1)$ and  $\tilde Z = \sum_{ F \in \cH} \tilde Y_F$. Let $E$ be the event that for all $F \in \cH$, $Y_F = \tilde Y_F$. Note that $\tilde Z = Z$ if $E$ holds and $\IND_{E^ c}\leq \sum_K X(K,G_n)$,
where $K$ ranges over a finite collection of multi-graphs with $\exc(K) \geq 1$. From fact (iii), it follows that $\dP (E^c) = o(1)$. 

Clearly, $\tilde Z \leq Z$ and $(Z)_p \leq Z^p$. Cauchy-Schwarz' inequality yields 
$$
\ABS{\dE (Z)_p - \dE (\tilde Z)_p} \leq \dE [(Z)_p \IND_{E^ c}] \leq \sqrt{ \dE ( Z^{2p} ) \dP ( E^ c) }.   
$$
Therefore,  using \eqref{eq:momentZ2} and $\dP (E^c) = o(1)$, we see that it suffices to prove that $ \dE (\tilde Z)_p$ converges to $\lambda(h)^ p$. 
Since $\tilde Y_{F} \in \{0,1\}$, we write
$$
( \tilde Z)_p =  \sum_*  \prod_{i=1} ^ p \tilde Y_{F_i}, 
$$
where the sum $\sum_*$ is over all choices of $p$ pairwise distinct $F_i$ in $\cH$. By assumption (H1), $Y_F$ is uniformly bounded, and therefore  
$$
\sum_*  \ABS{\prod_{i=1} ^ p \tilde Y_{F_i} - \prod_{i=1} ^ p  Y_{F_i}}
$$
can be bounded by $c_1\sum_K X( K , G)$ where $K$ ranges over 
a finite collection of multi-graphs with $\exc (K) \geq 1$ and $c_1=c_1(\theta,h,p)$. Therefore from fact (iii) we get 
$$
\dE ( \tilde Z)_p = \dE \sum_*  \prod_{i=1} ^ p  Y_{F_i} + o(1).  
$$
The conclusion $\dE ( \tilde Z)_p\to \l(h)^p$, $n\to\infty$, then follows from \eqref{sum*yf}.
\end{proof}

\subsection{Unimodular Galton-Watson trees with colors}
Let $\wcG^*(\cC)$ denote the set of equivalence classes of rooted directed locally finite colored multi-graphs, i.e.\ the set of connected multi-graphs $G\in\wcG(\cC)$ with a distinguished vertex $o$ (the root) where two rooted multi-graphs are identified if they only differ by a relabeling of the vertices. 
An element of $\wcG^*(\cC)$ is called a rooted directed colored tree if the corresponding colorblind multi-graph defined via \eqref{colorblind} has no cycles. 
We now introduce a probability measure on $\wcG^*(\cC)$ supported on rooted colored directed trees. 
Let $P \in \cP ( \cM_L)$ be a probability measure on $\cM_L$, $|\cC|=L^2$, 
such that for all $c\in \cC$, 
\begin{equation}\label{edc}
\dE D_c = \dE D_{\bar c},
\end{equation}
where $D\in \cM_L$ has distribution $P$. For each $c \in \cC$ such that $\dE D_c > 0$,
define the probability measure $\wP^{c} \in \cP ( \cM_L)$ such that, for $M \in \cM_L$, 
$$
\wP^{c} (M) = \frac{ ( M_{\bar c} +1 )\,P ( M + E^ {\bar c} )  }{ \dE D_c },
$$
where $D$ has distribution $P$, and for any $c\in\cC$, $E^c$ denotes the matrix with all entries equal to $0$ except for the entry at $c$, which equals $1$. Notice that 
$\wP^c$ is indeed a probability since
$$
\sum_{M \in \cM_L}( M_{\bar c} +1 )\,P ( M + E^ {\bar c} ) = \sum_{M \in \cM_L} M_{\bar c} \,P ( M ) = \dE D_{\bar c} = \dE D_c.
$$
If $\dE D_c = 0$ then we set  $\wP^{c} (M) = \IND ( M = 0)$.

In a rooted directed colored tree $(T,o)$, for all $v \ne o$, call $a(v)$ the parent of 
$v$ in $T$. The {\em type} of a vertex $v \ne o$ in $(T,o)$ is defined as the color of the edge $(a(v),v)$. 
The probability measure $\UGW(P)\in \cP(\wcG^*(\cC))$ is the law of the multi-type Galton-Watson tree defined as follows. The root $o$ produces offspring according to the distribution $P$, i.e.\ the root has $D_c$ children  of type $c$, for all $c\in\cC$, where $D\in\cM_L$ has law $P$. Recursively, and independently, any $v\neq o$ of type $c$, produces offspring according to the distribution $\wP^ c$, i.e.\ 
$v$ has $D_{c'}$ children  of type $c'$, for all $c'\in\cC$, where $D\in\cM_L$ has law $\wP^ c$.
Notice that in the case of a single color ($L=1$ and $\cC = \{(1,1)\}$), then $P$ is a probability measure on $\dZ_+$ and $\UGW(P)$ coincides with the 
Galton-Watson tree $\UGW_1(P)$ with degree distribution $P$, cf.\ \eqref{eq:defwP}.

Following the argument of Lemma \ref{le:UGWunimod}, it could be proved that the measure $\UGW(P)$ is unimodular. However, in the next paragraph, Theorem \ref{th:convlocCM} implies that $\UGW(P)$ is sofic (and hence unimodular).

\subsection{Local weak convergence}
It is straightforward to extend the local topology introduced in Section \ref{sec:LWC} to the case of rooted directed multi-graphs with colored edges $\wcG^*(\cC)$. The only difference is that the weight function $\omega$ is now matrix-valued. 
\begin{theorem}\label{th:convlocCM}
If $G_n \in \wcG(\bD^{(n)} )$ has distribution $\CM ( {\bD}^{(n)})$, with $\bD^{(n)}$ such that  assumptions (H1)-(H2) hold, then with probability one $U(G_n) \weak \UGW(P)$. Moreover the same result holds if $G_n$ is uniformly sampled on $\cG(\bD^{(n)},h )$, for any fixed $h\geq 2$. 
\end{theorem}
 
In the case of a single color $L=1$, Theorem \ref{th:convlocCM} is folklore; see e.g.\ the monographs \cite{MR1782847,MR2656427}. 
 The proof of Theorem \ref{th:convlocCM} in the general case is given in the appendix.

 \subsection{Graphs with given tree-like neighborhood}
Here we show how the configuration model can be used to count the number of graphs 
with a given tree-like neighborhood structure.

 Fix $n$ and a graph $G=(V,E)$ with $V=[n]$. Call $\cG_n$ the set of all such graphs. For $h\in\dN$, define the $h$-neighborhood vector
 \begin{equation}\label{psiG}
 \psi_h(G)=\left([G,1]_h,\dots,[G,n]_h\right),
\end{equation}  
where $[G,u]_h$ stands for the equivalence class of the $h$-neighborhood of $G$ at vertex $u$. 
We say that $G$ is $h$-tree-like if $[G,u]_h$ is a tree for all $u\in[n]$.

  We describe now a procedure which turns the given graph $G$ into a directed colored graph $\wt G$ in $\cG(\cC)$.
    The color set $\cC$ is defined as follows. Let $\cF\subset \cG^*_{h-1}$ denote the collection of all 
 equivalence classes of the subgraphs $G(u,v)_{h-1}$, where we recall that $G(u,v)$ is the rooted graph obtained from $G$ by removing the edge $\{u,v\}$ and taking the root at $v$. For simplicity, below we will identify $G(u,v)_{h-1}$ with its equivalence class. If $L=|\cF|$ denotes the cardinality of $\cF$,  we call $\cC$ the set of $L^2$ pairs 
 $(g,g')$, with $g,g'\in\cF$; see Figure \ref{treelike1} for an example. To construct the directed colored graph, for every pair $u,v$ such that $\{u,v\}$ is an edge of $G$, we include a directed edge $(u,v)$ with color
 \begin{equation}\label{dcmg}
 (g,g')=(G(u,v)_{h-1},G(v,u)_{h-1}),
 \end{equation} 
 together with the directed edge $(v,u)$ with color
 $(g',g)=(G(v,u)_{h-1},G(u,v)_{h-1})$. This defines an element $\wt G$ of $\cG(\cC)$; see Figure \ref{gtilde}.
 As such, we can define its degree sequence $\bD=\bD(\wt G)$ as in \eqref{dcu} above.
 Notice that if $G$ is $h$-tree-like, then  the above construction yields an element of $\cG(\bD,2h+1)$ since 
 being $h$-tree-like is equivalent to having no cycles with length $1\leq \ell\leq 2h+1$; see Figure \ref{gtilde2}. 
A crucial property to be used below is that, for this particular choice of $\bD$,  {\em all} elements of $\cG(\bD,2h+1)$ have the same $h$-neighborhoods.

\begin{figure}
\begin{tikzpicture}[-,>=stealth',shorten >=1pt,auto,node distance=1.5cm,
                    thick,main node/.style={circle,draw,font=\sffamily
                    \bfseries}]

  \node[main node] (1) {1};
  \node[main node] (2) [left of=1] {2};
  \node[main node] (3) [right of=1] {3};
  \node[main node] (4) [above of=1] {4};
   \node[main node] (5) [below of=2] {5};
   \node[main node] (6) [below of=3] {6};
     \node[main node] (7) [below of=1] {7};
   \node[main node] (8) [above of=3] {8};
   \node[main node] (9) [left of=4] {9};
%
\draw [fill] (3.6,1.5) circle (0.17cm);

\draw [fill] (5,1.5) circle (0.17cm);
\draw   (4.7,0.5) circle (0.17cm);
\draw  (5.3,0.5) circle  (0.17cm);
\draw  (4.7,-0.5) circle (0.17cm);
\draw  (5.3,-0.5) circle (0.17cm);

\draw (4.74,0.65) -- (5,1.5);
\draw (5.26,0.65) -- (5,1.5);
\draw (4.7,0.34) -- (4.7,-0.36);
\draw (5.3,0.34) -- (5.3,-0.36);

\draw [fill] (6.4,1.5) circle (0.17cm);
\draw  (6.4,0.5) circle (0.17cm);
\draw  (6.4,-0.5) circle (0.17cm);

\draw (6.4,0.65) -- (6.4,1.5);
\draw (6.4,0.34) -- (6.4,-0.36);

\draw [fill] (7.8,1.5) circle (0.17cm);
\draw  (7.8,0.5) circle (0.17cm);
\draw  (8.1,-0.5) circle (0.17cm);
\draw  (7.5,-0.5) circle (0.17cm);

\draw (7.8,0.65) -- (7.8,1.5);
\draw (7.75,0.34) -- (7.54,-0.35);
\draw (7.83,0.34) -- (8.06,-0.35);

\draw [fill] (9.1,1.5) circle (0.17cm);
\draw   (8.8,0.5) circle (0.17cm);
\draw  (9.4,0.5) circle  (0.17cm);
\draw  (9.4,-0.5) circle (0.17cm);

\draw (8.84,0.65) -- (9.1,1.5);
\draw (9.36,0.65) -- (9.1,1.5);
\draw (9.4,0.34) -- (9.4,-0.36);

\node at (3.6,0.5) {$\a$};
\node at (5,-1.5) {$\b$};
\node at (6.4,-1.5) {$\c$};
\node at (7.8,-1.5) {$\d$};
\node at (9.1,-1.5) {$\h$};

  \path[every node/.style={font=\sffamily\small}]
    (1) 
        edge node {} (2)
        edge node {} (4)
    (2) 
    edge node {} (5)
edge node {} (9)
    (3) 
           edge node {} (6)
           edge node {} (8)
    (4) 
       edge node {} (8)
        (5) edge node {} (7)
        
    (6) edge node {} (7) 
    ;
\end{tikzpicture}
%
\caption{A $3$-tree-like graph $G\in\cG_{9}$ (left). When $h=3$, the associated set $\cF$ of equivalence classes is given by the $L=5$ rooted unlabeled graphs depicted on the right (the black vertex is the root). Here $G(2,9)_2=(\b,\a)$, $G(9,2)_2=(\a,\b)$; 
$G(1,2)_2=(\c,\h)$, $G(2,1)_2=(\h,\c)$;  $G(5,2)_2=(\c,\h)$, $G(2,5)_2=(\h,\c)$; $G(7,5)_2=(\c,\d)$, $G(5,7)_2=(\d,\c)$; $G(6,7)_2=G(7,6)_2=G(3,6)_2=G(6,3)_2=G(3,8)_2=G(8,3)_2=  G(4,8)_2=G(8,4)_2=(\c,\c)$;  $G(4,1)_2=(\c,\d)$, $G(1,4)_2=(\d,\c)$.
}
\label{treelike1}
\end{figure}
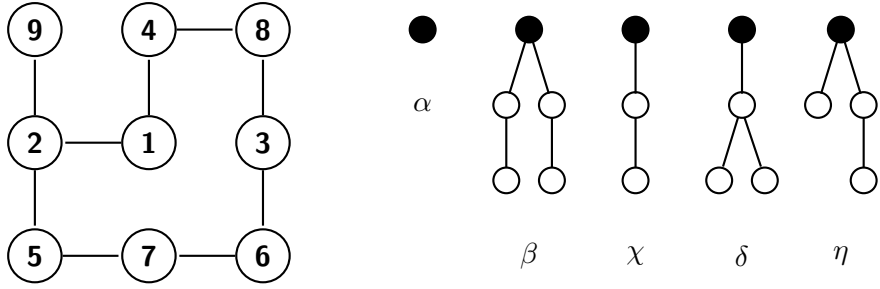
  
  \begin{figure}
\begin{tikzpicture}[->,>=stealth',shorten >=1pt,auto,node distance=2cm,
                    thick,main node/.style={circle,draw,font=\sffamily
                    \bfseries}]

  \node[main node] (1) {1};
  \node[main node] (2) [left of=1] {2};
  \node[main node] (3) [right of=1] {3};
  \node[main node] (4) [above of=1] {4};
   \node[main node] (5) [below of=2] {5};
   \node[main node] (6) [below of=3] {6};
     \node[main node] (7) [below of=1] {7};
   \node[main node] (8) [above of=3] {8};
   \node[main node] (9) [left of=4] {9};

  \path[every node/.style={font=\sffamily\small}]
    (1) 
        edge [bend right=20] 
        node {} (2)
        edge node [left] {$(\d,\c)$} (4)
    (2) 
    edge  node [below] {$(\h,\c)$} (1)
    edge node [left] {$(\h,\c)$} (5)
edge node [left] {$(\b,\a)$} (9)
    (3) 
           edge node {} (6)
           edge node [right] {$(\c,\c)$}(8)
    (4) edge [bend left=20] node {} (1)
       edge node [below] {$(\c,\c)$}(8)
        (5) edge node {} (7)
        edge [bend right=20] node  {} (2)
    (6) edge node {} (7) 
    edge [bend right=20] node [right] {$(\c,\c)$} (3) 
 (7) edge [bend left=20] node [above] {$(\c,\c)$}(6) 
 edge [bend right=20] node [above] {$(\c,\d)$} (5) 
(8) edge [bend right=20] node {} (4) 
 edge [bend right=20] node {} (3) 
(9) edge [bend left=20] node {}  (2) 
    ;
\end{tikzpicture}
%
\caption{The graph $\wt G\in\wcG(\cC)$ defined by \eqref{dcmg} when $G$ is the graph from Figure \ref{treelike1}.
It is understood that if the directed edge $(u,v)$ has color $(g,g')\in\cC$, then the opposite edge $(v,u)$ has color $(g',g)$.
}
\label{gtilde}
\end{figure}
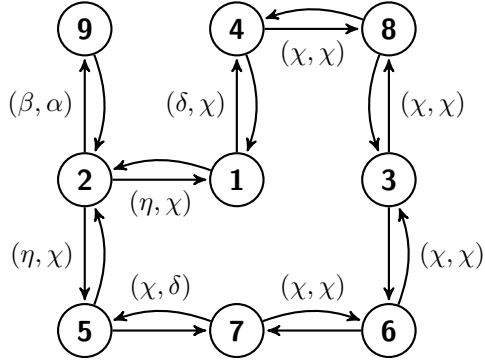

 \begin{figure}
\begin{tikzpicture}[->,>=stealth',shorten >=1pt,auto,node distance=2cm,
                    thick,main node/.style={circle,draw,font=\sffamily
                    \bfseries}]

  \node[main node] (1) {1};
  \node[main node] (2) [left of=1] {2};
  \node[main node] (3) [right of=1] {3};
  \node[main node] (4) [above of=1] {4};
   \node[main node] (5) [below of=2] {5};
   \node[main node] (6) [below of=3] {6};
     \node[main node] (7) [below of=1] {7};
   \node[main node] (8) [above of=3] {8};
   \node[main node] (9) [left of=4] {9};

  \node[main node] (10) [right of=8]{9};
  \node[main node] (11) [right of=3] {2};
   \node[main node] (12) [right of=11]{1};
  \node[main node] (13) [right of=12] {3};
  \node[main node] (14) [above of=12] {4};
   \node[main node] (15) [below of=11] {5};
   \node[main node] (16) [below of=13] {6};
     \node[main node] (17) [below of=12] {7};
   \node[main node] (18) [above of=13] {8};

   \path[every node/.style={font=\sffamily\small}]
    (1) 
        edge [bend right=20] 
        node {} (2)
        edge node [left] {$(\d,\c)$} (4)
    (2) 
    edge  node [below] {$(\h,\c)$} (1)
    edge node [left] {$(\h,\c)$} (5)
edge node [left] {$(\b,\a)$} (9)
    (3) 
           edge [bend left=15] node {} (6)
           edge [bend right=40] node {}(6)
    (4) edge [bend left=20] node {} (1)
       edge node [below] {$(\c,\c)$}(8)
        (5) edge node {} (7)
        edge [bend right=20] node  {} (2)
    (6) 
    edge [bend right=40] node {} (3) 
    edge [bend left=15] node {} (3) 
 (7) 
 edge [bend right=20] node [above] {$(\c,\d)$} (5) 
 edge   node {}(8)
(8) edge [bend right=20] node {} (4) 
 edge [bend right=12] node {} (7) 
(9) edge [bend left=20] node {}  (2) 
    ;
    
    \node at (2.1,0.91) {$(\c,\c)$};
     \node at (1,-1.3) {$(\c,\c)$};

      \path[every node/.style={font=\sffamily\small}]
    (12) 
        edge [bend right=20] 
        node {} (11)
        edge [bend right=20] node {} (17)
    (11) 
    edge  node {} (12)
    edge node [left] {$(\h,\c)$} (15)
edge node [left] {$(\b,\a)$} (10)
    (13) 
           edge node {} (16)
           edge node [right] {$(\c,\c)$}(18)
    (14) 
    edge [bend left=10] node {} (15)

       edge node [below] {$(\c,\c)$}(18)
        (15) edge node {} (14)
        edge [bend right=20] node  {} (11)
    (16) edge node {} (17) 
    edge [bend right=20] node [right] {$(\c,\c)$} (13) 
 (17) edge [bend left=20] node [above] {$(\c,\c)$}(16) 
 edge  node {} (12) 
(18) edge [bend right=20] node {} (14) 
 edge [bend right=20] node {} (13) 
(10) edge [bend left=20] node {}  (11) 
    ;
     \node at (6.6,-0.8) {$(\c,\d)$};
\node at (4.8,0.6) {$(\c,\h)$};
 \node at (5.05,-1.4) {$(\c,\d)$};
\end{tikzpicture}
%
\caption{Two examples of multigraphs $\G_1,\G_2\in \wcG(\bD)$, where $\bD=\bD(\wt G)$ is the degree sequence of $\wt G$ from Figure \ref{gtilde}. Notice that $\G_1$ (left) yields a colorblind multigraph $\bar \G_1$ with a double edge at $\{3,6\}$, while $\G_2$ (right) yields a $3$-tree-like graph $\bar\G_2$, i.e.\ $\G_2\in \cG(\bD,7)$.  In particular, as guaranteed by Lemma \ref{le:treelike}, $\bar\G_2$ has the same $3$-neighborhoods of $G$. 
}
\label{gtilde2}
\end{figure}

 \begin{lemma}\label{le:treelike}
Let $h\in\dN$, let $G\in\cG_{n}$ be a fixed $h$-tree-like graph and let $\bD=\bD(\wt G)$ be the associated degree sequence as above. For any $\G\in\cG(\bD,2h+1)$, the colorblind graph $\bar\G\in\cG_{n}$ defined via \eqref{colorblind} satisfies $\psi_h(\bar\G)=\psi_h(G)$.
\end{lemma} 
\begin{proof}
Consider first the case $h=1$. 
If $\G\in\cG(\bD,3)$, then for any node $i\in[n]$,
the $1$-neighborhood $(\bar\G,i)_1$ at $i$
is uniquely determined by the number of edges exiting node $i$. By \eqref{colorblind}, this number equals $\sum_{c\in\cC}D_c(i)$, which is independent of $\G$. Thus, all $\G\in\cG(\bD,3)$ satisfy 
necessarily  $\psi_1(\bar\G)=\psi_1(G)$.

Next, 
we assume that any $\G\in\cG(\bD,2h+1)$ satisfies $\psi_{h-1}(\bar\G)=\psi_{h-1}(G)$, and show that 
 $\psi_{h}(\bar\G)=\psi_{h}(G)$. Since $\cG(\bD,2h+1)\subset \cG(\bD,2(h-1)+1)$, by induction over $h$ this will prove the desired result.
 
 Since there are no cycles of length $\ell\leq 2h+1$ in $G$,
 $\cF$ consists of unlabeled rooted trees of depth $h-1$. 
 For any $t\in\cF$ we write $t_k$ for the $k$-neighborhood of the root in $t$ (truncation of $t$ at depth $k$).  
 Moreover, if $t$ is a rooted tree, we write $t_{k,+}$ for the unlabeled rooted tree of depth $k+1$ obtained from $t_k$ by adding a new edge to the root and taking the other endpoint of that edge as the new root.  If $t,t'$ are finite rooted trees, we write $t\cup t'$ for the rooted tree obtained by joining $t,t'$ at the common root. 
Since there are no cycles of length $\ell\leq 2h+1$ in $\bar\G$, 
to prove $\psi_{h}(\bar\G)=\psi_{h}(G)$ it is sufficient to show that for any
edge $(u,v)$ with color $(t,t')$ in $\G$, with $t,t'\in\cF$, one has $\bar\G(u,v)_{h-1}=t'$ and $\bar\G(v,u)_{h-1}=t$. 

Let $(u,v)$ be an edge in $\G$ with color $(t,t')$.
Notice that in $\wt G$, $u$ must have an edge $(u,\wt v)$ with color $(t,t')$ going out of $u$, and $v$ must have an edge $(v,\wt u)$ with color $(t',t)$ going out of $v$.
Therefore, $[G,u]_{h-1}=t\cup t'_{h-2,+}$ and $[G,v]_{h-1}=t'\cup t_{h-2,+}$. By assumption, $(\bar\G,u)_{h-1}=[G,u]_{h-1}$ and $(\bar\G,v)_{h-1}=[G,v]_{h-1}$. Therefore, the rooted trees $T:=\bar\G(v,u)_{h-1}$ and $T':=\bar\G(u,v)_{h-1}$ must satisfy
\begin{equation}\label{tt'}
T\cup T'_{h-2,+}=t\cup t'_{h-2,+}\,,\qquad \;T'\cup T_{h-2,+}=t'\cup t_{h-2,+}.
\end{equation}
We need to show that $t=T$ and $t'=T'$. From \eqref{tt'}, one has that it is sufficient to show that $T'_{h-2}=t'_{h-2}$ and $T_{h-2}=t_{h-2}$. Truncating \eqref{tt'} at depth $h-2$ one has
$$
T_{h-2}\cup T'_{h-3,+}=t_{h-2}\cup t'_{h-3,+}\,,\qquad \;T'_{h-2}\cup T_{h-3,+}=t'_{h-2}\cup t_{h-3,+}.
$$
Thus, it is sufficient to show that $T'_{h-3}=t'_{h-3}$ and $T_{h-3}=t_{h-3}$. Iterating this reasoning, one finds that it suffices to show that $T'_{1}=t'_{1}$ and $T_{1}=t_{1}$. However, this is guaranteed by the fact that the degree of $u$ in $G$ and $\bar \G$ is the same, for any $u\in[n]$. 
\end{proof}


We turn to the problem of counting the number of graphs $G'\in\cG_{n}$ whose $h$-neighborhood distribution coincides with that of a given $h$-tree-like graph $G$.
The following is an important corollary of Lemma \ref{le:treelike}.
 \begin{corollary}\label{cor:treelike}
Fix an arbitrary $h$-tree-like graph $G\in\cG_{n}$, and define
\begin{equation}\label{nhg}
N_h(G)=\ABS{\big\{G'\in\cG_{n}:\; U(G')_h = U(G)_h\big\}}.
\end{equation}
One has 
\begin{equation}\label{nhgo}
N_h(G)=
n(\bD)|\cG(\bD,2h+1)|,
\end{equation}
where $\bD=(D(1),\dots,D(n))$ is the degree sequence associated to $G$ via \eqref{dcmg}, and $n(\bD)$ denotes the number of distinct vectors $(D(\pi_1),\dots,D(\pi_n))\in\cD_n$ as $\pi:[n]\mapsto[n]$ ranges over permutations of the labels. 
\end{corollary} 
\begin{proof}
For a permutation $\pi:[n]\mapsto[n]$, let 
$\bD^\pi=(D(\pi_1),\dots,D(\pi_n))$. Since the cardinality of $\cG(\bD^\pi,2h+1)$ does not depend on $\pi$, 
$n(\bD)|\cG(\bD^\pi,2h+1)|$ coincides with the cardinality of $\cup_\pi  \cG(\bD^\pi,2h+1)$. By 
Lemma \ref{le:treelike}, any two distinct elements  $\G_1,\G_2\in \cup_\pi  \cG(\bD^\pi,2h+1)$ 
yield two distinct graphs $\bar\G_1,\bar\G_2$ such that $U(\bar\G_i)_h=U(G)_h$, $i=1,2$. This proves that $N_h(G)\geq 
n(\bD)|\cG(\bD,2h+1)|$. On the other hand, any two distinct elements  $G_1,G_2\in \cG_{n}$ with
$U(G_i)_h = U(G)_h$, $i=1,2$,  yield  
two distinct elements $\wt G_1,\wt G_2\in \cup_\pi\cG(\bD^\pi,2h+1)$ with the map $G\mapsto\wt G$ defined
by \eqref{dcmg}. This proves the other direction. 
\end{proof}



\begin{lemma}\label{le:existenceG}
Fix $h\in\dN$, and $P \in \cP(\cT^*_h)$ admissible, with finite support
and $\dE_P \deg (o) = d$. Let $m = m(n)$ be a sequence such that $m/n \to d /2$ as $n \to \infty$. Then, there exist a finite set $\D \subset \cT^*_h$ and a sequence of graphs $\G_n \in \cG_{n,m}$ such that the support of $U(\G_n)_h $ is contained in $\D$ for all $n$ and  $U(\G_n)_h \weak P$ as $n\to\infty$. 
\end{lemma}

\begin{proof}
Let $S:=\{t_1,\dots,t_r\}\subset \cT_h^*$ be the finite support of $P$. We define the vector $\bg^{(n) } = (g^{(n)} (1), \cdots, g^{(n)}(n))$ with  $g^{(n)} (i) \in S$ by setting $g^{(n)} (i) = t_{k}$ if $\sum_{\ell \leq k} P(t_\ell) > (i-1)/n$ and $\sum_{\ell \leq k - 1} P(t_\ell) \leq (i-1) /n$ with the convention that the sum over an empty set is $-\infty$. The empirical measure of $\bg^{(n)}$, say $P^{(n)}$, converges weakly to $P$.

Let $\cC$ denote the set of all pairs $c=(t,t')\in\cT_{h-1}^*\times  \cT_{h-1}^*$  associated to any element $g \in S$ as in \eqref{dcmg}. In this manner, we associate to any $g^{(n)} (i)$ an integer valued matrix $D^{(n)} (i) \in \cM_L$ where $\ABS{\cC} = L^2$. We denote by $S^{(n)}_c = \sum_{i=1} ^n  D^{(n)}_c (i)$. We finally set, for $c \in \cC_=$,  $\widetilde S^{(n)}_c = 2 \lfloor  S^{(n)}_c /2 \rfloor$ and, for $c \in \cC_{\ne}$, $\widetilde S^{(n)}_c =  S^{(n)}_c  \wedge S^{(n)}_{\bar c}$. We may fix a sequence of integer-valued matrices $\widetilde \bD^{(n)}  = ( \widetilde D^{(n)} (i))_{1 \leq i \leq n}$ such that component-wise $\widetilde D^{(n)}(i) \leq D^{(n)}(i)$ and, for all $c \in \cC$, $ \sum_{i=1}^n \widetilde D^{(n)}_c (i)  = \widetilde S^{(n)}_c $.  The properties $P^{(n)} \weak P$ and $\supp (P^{(n)}) \subset S$ imply that for all $c \in \cC$, $\widetilde S^{(n)}_c - S^{(n)}_c = o(n)$ and for all but $o(n)$ vertices $\widetilde D^{(n)}(i) =  D^{(n)}(i)$. Moreover, 
$$
\widetilde m = \frac 1 2  \sum_{c \in \cC} \widetilde S^{(n)}_c = m + o(n). 
$$ 

We consider the generalized configuration model on $\widetilde \bD^{(n)}$. Corollary \ref{cor:Poilim} implies the existence, for all $n$ large enough, of an directed colored graph $\widetilde \G_n$ with girth at least $2 h +1$ and whose colored degree sequence is precisely given by $\widetilde \bD^{(n)}$. Let $\bar \G_n$ be the associated color-blind graph. The proof of Lemma \ref{le:treelike} actually shows that if a vertex $v$ of $\widetilde \G_n$ is such that all vertices $u$ in $(\widetilde \G_n,v)_h$ satisfy $\widetilde D^{(n)}_c (u) = D^{(n)}_c (u)$ then the equivalence class of $(\bar\G_n,v)_h$ is precisely $g^{(n)} (v)$. Now, let $\theta$ be the maximal degree of vertices in $t \in S$ and set $\kappa = \sum_{\ell = 0}^{h} \theta^h$. Any vertex is in the $h$-neighborhood of at most $\kappa$ vertices. Since for all but $o(n)$ vertices $\widetilde D^{(n)}(v) =  D^{(n)}(v)$, we deduce that for all but $o(n)$ vertices, the equivalence class of $(\bar\G_n,v)_h$ is  $g^{(n)} (v)$. We thus have proved that $U(\bar \G_n)_h \weak P$. Also, by construction, the support of $U(\bar \G_n)_h$ is contained in the finite set $\D_{h,\theta}$ of unlabeled rooted trees $t \in \cT^*_h$ such that all degrees of vertices in $t$ are bounded by $\theta$.

A last modification is needed: we have $\bar \G_n\in\cG_{n,\widetilde m}$ and we need a graph $\G_n \in\cG_{n,m}$. However, since the number of vertices in $(\bar \G_n,v)_h$ is bounded by $\kappa$, adding or removing one edge in $\bar \G_n$ will change the value of  $(\bar \G_n,v)_h$ for at most $2 \kappa$ vertices. Let $\delta(n) = |\widetilde m - m| = o(n)$. Assume first that $\widetilde m < m$, then we need to add edges to $\bar \G_n$. We may add $\delta(n)$ new edges to $\bar \G_n$ such that any vertex has a most one new adjacent edge.  From what precedes, we obtain a graph $\G_n  \in\cG_{n,m}$ such that $U(\G_n)_h \weak P$. Moreover the support $U(\G_n)_h$ is contained in $\Delta_{h,\theta +1}$.   If $\widetilde m > m$, we need to remove edges. We remove an arbitrary subset of them of cardinality $\delta(n)$. We get a graph $\G_n  \in\cG_{n,m}$ such that $U(\G_n)_h \weak P$ and the support of $U(\G_n)_h$ is  contained in $\Delta_{h,\theta}$.  \end{proof}

\subsection{Proof of Corollary \ref{cor:entropy1}}
We note that the set, say $\cS$, of sofic measures supported on trees is a closed subset of $\cP_{u} ( \cT^*)$.  Let $B$ be the set of measures of the form $\rho = \UGW_h ( P)$ with $P \in \cP( \cT^*_h)$ admissible with finite support and $h \in \dN$. 
A consequence of Lemma \ref{le:existenceG} and Theorem \ref{th:convlocCM} is that $B$ is a subset of $\cS$.

Let us first check that for any $h \in \dN$ and $P \in \cP( \cT^*_h)$ admissible, $\rho = \UGW_h(P) \in \cS$. For each $n \in \dN$, consider the forest $F_n$ obtained from $(T,o)$, with law $\rho$, by removing all edges adjacent to a vertex with degree higher than $n$. We may define $\rho^{(n)}$ as the the law of $(F_n(o),o)$, the connected component of the root. It is easy to check that $\rho^{(n)}$ is a unimodular measure. We define $Q_n = \rho^{(n)}_h$, the law of its $h$-neighborhood. By construction, $\UGW_h (Q_n) \in B$ and $Q_n$ converges weakly to $P$. We deduce that $\UGW_h(Q_n) \weak \UGW_h (P)$ and $\UGW_h(P) \in \cS$. 

Moreover, if $\r\in\cP_{u} ( \cT^*)$, then $\UGW_h(\r_h)\weak \r$, as $h\to\infty$. From what precedes $\UGW_h(\r_h) \in \cS$. Therefore, $\rho \in \cS$ and $\cS = \cP_{u} ( \cT^*)$.

\section{Graph counting and Entropy}
\label{EN}

In this section we prove Theorem \ref{th:entropy1} and Theorem \ref{th:entropy2}. The strategy will be as follows. We
first establish the cases $\Sigma (\rho) = -\infty$ in Theorem \ref{th:entropy1}.  We then prove Theorem \ref{th:entropy2}, and later complete the proof of Theorem \ref{th:entropy1}.
In what follows, we fix $d>0$ and a sequence $m=m(n)$ such that $m/n\to d/2$ as $n\to\infty$.

\subsection{
Measures with $\Si(\r)=-\infty$}\label{sec:enpre}

Since unimodular measures form a closed subset of $\cP(\cG^*)$, if $\r\notin\cP_u(\cG^*)$,
then for some $\veps>0$ one has $B(\r,\veps)\subset \cP(\cG^*)\setminus \cP_u(\cG^*)$. Since $U(G_n)\in\cP_u(\cG^*)$, then  
$|\cG_{n,m}(\r,\veps)|=0$.  Therefore $\overline\Si(\r)=-\infty$ for all $\r\notin\cP_u(\cG^*)$. 

Next, we show that $\overline \Sigma (\rho) = -\infty$ whenever $\dE_{\rho} \deg(o) \ne d$.
We start with the case $\dE_\rho \deg( o ) > d$.
Let $\rho \in \cP_{u} ( \cG^*)$ and assume that $\overline \Sigma (\rho) > - \infty$. Then, by an extraction argument, there must exist a sequence of graphs $G_{n} \in \cG_{n,m}$ such that $U(G_{n}) \weak \rho$. Weak convergence then implies that $\dE_{U(G_{n})} [\deg(o)\wedge t] \to  \dE_{\rho} [\deg(o)\wedge t]$ for any $t>0$, and therefore, letting $n\to\infty$ and then $t\to\infty$:
$$
\LIMINF_{n\to\infty} \dE_{U(G_{n})} \deg(o) \geq  \dE_{\rho} \deg(o).
$$
On the other hand, by construction,
\begin{equation}\label{degd}
\dE_{U(G_n)} \deg (o)  = \frac 1 n \sum_{v= 1} ^ n \deg_{G_n} (v) = \frac {2m }{n} = d + o(1). 
\end{equation}
We thus have checked that if $\dE_\rho \deg( o ) > d$, then
$\overline \Sigma(\rho)= -\infty$. 

The case $\dE_\rho \deg ( o , G) < d$ requires a little more care. 
\begin{lemma}
If $\dE_\rho \deg( o) < d$, then $\overline \Sigma(\rho)= -\infty$.
\end{lemma}
\begin{proof}
From \eqref{eq:Gnmlog}, it is sufficient to prove that, for any sequence $\veps_n  \to 0$,
\begin{equation}\label{eq:domo}
\lim_{n\to\infty}
\frac1n \log  \dP(U(G_n)\in B( \rho , \veps_n)) 
=-\infty,
\end{equation}
where $G_n$ is a uniform random graph in $\cG_{n,m}$. Define $d' = \dE_\r \deg( o)$ and $\delta = d- d' > 0$. 
If  $U(G_n)\in B(\r,\veps_n)$ for all $n$, then for any $t>0$: 
$$\dE_{U(G_n)} [\deg (o)\IND(\deg(o)\leq t)]\to \dE_\r[ \deg( o)\IND(\deg(o)\leq t)].$$ 
Therefore, for some sequence $t_n \to \infty$, one has
$$
\frac 1 n \sum_{v \in [n]}   \deg_{G_n}(v)  \IND ( \deg_{G_n}(v)  \leq t_n ) \to d'.
$$
Define $A_n = \{ i \in [n] :  \deg_{G_n}(i) >  t_n \}$. Using \eqref{degd} one has 
$$
\frac 1 n \sum_{v \in A_n} \deg_{G_n}(v) \to \delta. 
$$
On the other hand, by Markov's inequality and \eqref{degd}, the cardinality of $A_n$ satisfies $|A_n|\leq\a_n n$, where $ \a_n= 2d/t_n$ for all $n$  large enough.
Thus $U(G_n)\in B( \rho , \veps_n)$ implies that there exists $S\subset [n]$ with  $|S|\leq \a_n n$
such that $\deg_{G_n}(S):=\sum_{v \in S} \deg_{G_n}(v)$ is larger than $\d n/2$ for all $n$ large enough.
By the union bound one has
$$
\dP(U(G_n)\in B( \rho , \veps_n))\leq {n \choose  \a_n n} \dP \PAR{
\deg_{G_n}\left([\a_n n]\right) \geq \delta n /2 },
$$
where $[\a_n n]=\{1,\dots,\a_n n\}$. 
Next, we check that 
\begin{equation}\label{deg1o}
\lim_{n\to\infty}
\frac1n\log \dP \PAR{
\deg_{G_n}\left([\a_n n]\right) \geq \delta n /2
}= - \infty.
\end{equation}
To this end, observe that $\deg_{G_n}([\a_n n])$ is stochastically dominated by $2N$, where $N$ denotes the binomial random variable $N=\BINOM{\a_n n^2}{2d/n}$. Indeed, the number of potential edges incident to the set $[\a_n n]$ is trivially bounded by $ \a_n n^2$ and each potential edge can be included 
in $G_n$ recursively, where at each step the probability of inclusion is bounded above by 
$$
\frac{m}{\binom{n}{2} - \a_n n^2}\leq \frac{2d}n,
$$
if $n$ is large enough, where we use $m/n\to d/2$ and $\a_n\to 0$. Therefore, from Chernov's bound,
for any $x> 0$, 
\begin{align*}
\dP &\PAR{
\deg_{G_n}\left([\a_n n]\right) \geq \delta n /2
}
\leq  
\dP \PAR{ 2 N \geq \delta n /2 } \leq e^{-\delta n x   }\dE[e^{4xN}] \\
&\;\; \quad= e^{-\delta n x } \PAR{ 1 +  (2d/n)  ( e^{4x} - 1)  }^{\a_n n^2} 
\leq e^{- \delta n x   +   2d \a_n n e^{4x} }.
\end{align*}
Taking e.g.\ $x=-\frac14\log {\a_n} $, one obtains \eqref{deg1o}.
Moreover, Stirling's formula implies 
$$
\frac 1 n \log {n \choose  n\a_n}  \sim - \a_n \log \a_n \to 0.
$$
This implies \eqref{eq:domo}. 
\end{proof}

We turn to the claim that $\Si(\r)=-\infty$ whenever $\r$ is not supported on trees. 
\begin{lemma}
Suppose $\r\in\cP_u(\cG^*)$ is such that $\r(\cT^*)<1$. Then there exists $\veps_0 > 0$ such that if $0 < \veps < \veps_0$, then
\begin{equation}\label{eq:kapparho}
\limsup_{n\to\infty} \; \frac {\log | \cG_{n,m} ( \rho, \veps) | }{ n \log n} < \frac d 2.  
\end{equation}
In particular, 
$\overline \Sigma ( \rho, \veps) = - \infty $, for any $0 < \veps < \veps_0$.
\end{lemma}
\begin{proof}
Once \eqref{eq:kapparho} is established, the last assertion follows from \eqref{eq:Gnmlog} and $m/n\to d/2$. 
Let us prove \eqref{eq:kapparho}. By assumption, there exist integers  $t$ and $\ell \geq 3$ such that 
$$
\dP_{\rho} \PAR{ (G,o)_t \hbox{ contains  a cycle of length $\ell$}} > 0.
$$
For integer $k \geq 2$, let us say that a cycle is a $(k, \ell)$-cycle if its length is $\ell$ and the degree of all vertices on the cycle is bounded by $k$. Since $(G,o)_t$ is $\rho$-a.s. locally finite, there exists an integer $k \geq 2$ such that 
$$
\dP_{\rho} \PAR{ (G,o)_t \hbox{ contains  a $(k,\ell)$-cycle}} > 0.
$$

Consider the function $f (G , o, v) =  \IND ( \dist_G ( o, v) \leq t\,  ; \,v \hbox{ is in a $(k,\ell)$-cycle})$. From what precedes 
$$
\dE_\rho \sum_{v \in V(G)} f ( G , o , v) > 0.  
$$
Since $\r$ is  unimodular, equation \eqref{eq:defunimod} applied to  $f$ implies that for some $\eta>0$,
$$
\dP_{\rho} \PAR{ o \hbox{ is in a $(k,\ell)$-cycle}} > 2\eta.
$$
Thus, if $G \in \cG_{n,m} ( \rho, \veps)$ and $\veps$ is small enough, 
$$
\dP_{U(G)} \PAR{ o \hbox{ is in a $(k,\ell)$-cycle}} > \eta.
$$
By definition of $U(G)$, this implies that the number of vertices in a $(k,\ell)$-cycle in $G$ is at least $\eta n$. 
Since degrees are bounded by $k$ in a $(k,\ell)$-cycle, we deduce that $G$ contains at least $\delta n$ mutually disjoint cycles of length $\ell$, for some $\delta  = \d(\ell, k )>0$.  Therefore, 
$$
\ABS{ \cG_{n,m} ( \rho, \veps) } \leq C_{n,\ell} \ABS{\cG_{n, m - \ell \lceil \delta n \rceil } },
$$
where $C_{n,\ell}$ is the number of ways to place $ \lceil \delta n\rceil $ disjoint cycles of length $\ell$ on $n$ vertices. One has
$$
C_{n,\ell} \leq \frac{ ( n)_{ \ell \lceil \delta n\rceil} }{ \lceil \delta n\rceil !} \leq \frac{ n^{ \ell \lceil \delta n\rceil}  }{ \lceil \delta n\rceil ! }.
$$
Indeed $ ( n)_{ \ell \lceil \delta n\rceil} $ counts the number of ordered choices of the $\ell$ vertices for each of $\lceil \delta n\rceil$ labeled cycles (the first $\ell$ vertices define the first cycle and so on), while division by $ \lceil \delta n\rceil !$ is used to remove cycle labels. 
By Stirling's formula,
$$
\log (C_{n,\ell}) \leq  \ell \delta n  \log n  - \delta n \log n + o ( n \log n ). 
$$
On the other hand, from \eqref{eq:Gnmlog}, we have
$$
\log  \ABS{ \cG_{n, m - \ell \lceil \delta n \rceil} }  = \Big( \frac d 2 - \ell \delta\Big)  n \log n + o ( n \log n ).
$$
So finally, 
$$
\log \ABS{ \cG_{n,m} ( \rho, \veps) } \leq  \frac d 2   n \log n  - \delta n \log n + o ( n \log n ).
$$
This proves \eqref{eq:kapparho}. 
\end{proof}

\subsection{Proof of Theorem \ref{th:entropy2} and Theorem \ref{th:entropy1}}\label{pfth:entropy2}
Notice that if $P\in\cP_h$, then $J_h(P)$ is a well defined extended real number in $[-\infty,\infty)$. The fact that $J_h(P)\leq s(d)$ follows from Proposition \ref{pro:enh} below and  from the upper bound $\overline\Sigma(\r)\leq s(d)$, cf.\ \eqref{eq:Gnmlog}. 

As before, we fix $d>0$ and an integer  sequence $m=m(n)$ such that $m/n\to d/2$ as $n\to\infty$. We start with three preliminary lemmas. 

\begin{lemma}\label{le:lscSi}
The function $\rho \mapsto \underline \Si(\rho)$ on $\cP_u(\cG_*)$ is upper semi-continuous.
\end{lemma}
\begin{proof}
Consider a sequence $(\rho_k)$ converging to $\rho$. We should check that $\underline \Si (\rho) \geq \limsup \underline \Si (\rho_k)$. Observe that for any $\veps >0$, for all $k$ large enough, $B(\rho , \veps) \supset B(\rho_k , \veps/2)$. We get for $k$ large enough,
$$
\underline \Si( \rho,\veps) \geq \underline \Si( \rho_k,\veps/2) \geq \underline \Si( \rho_k).  
$$ 
Letting $k$ tend to infinity and then $\veps$ to $0$, we obtain the claim. \end{proof}

We will also need two general lemmas.

\begin{lemma}\label{le:HwPbis}
Let $P=\{p_x,\,x\in\cX\}$ be a probability measure on a discrete space $\cX$ such that $H(P) < \infty$. Let $(\ell_x)_{x \in \cX}$ be  a sequence with $\ell_x \in\dZ_+$, $x\in\cX$, such that $ \sum_{x} p_x \ell_x \log \ell_x < \infty$. Then 
$ - \sum_{x} p_x \ell_x \log  p_x < \infty$.
\end{lemma}

\bp{}
We can assume without loss of generality that $p_x \ne 0$ for all $x \in \cX$. We look for the sequence $(\ell_x)$ which maximizes  the linear function $ - \sum_{x} p_x \ell_x \log  p_x < \infty$ under the constraints $\ell_x \geq 1$ and $\sum_{x} p_x \ell_x \log \ell_x  = c$.  If the constraint $\ell_x \geq 1$ is not saturated, taking derivative, we find $0 = - p_x \log p_x - \lambda p_x - \lambda p_x \log \ell_x$ where $\lambda$ is the Lagrange mutliplier associated to the constraint $\sum_{x} p_x \ell_x \log \ell_x  = c$. We get $\ell_x = e^{-1} p_x^{-1/\lambda}$. Let $\cX_1$ be the set of $x$ such that $\ell_x =1$. We thus find
\begin{align*}
- \sum_{x} p_x \ell_x \log  p_x & =   - \sum_{x \in \cX_1} p_x \log p_x - \sum_{x \notin \cX_1}  p_x \ell_x \log  p_x \\
& \leq  H(P) -  \sum_{x \notin \cX_1}  p_x \ell_x \log e^{-\lambda} \ell_x^{-\lambda}\\
& \leq  H(P)   + \lambda \sum_x p_x \ell_x  + \lambda \sum_{x}  p_x \ell_x \log \ell_x.
\end{align*}
The conclusion follows.  \ep

\begin{lemma}\label{le:contH} 
Let $p,\kappa$ be integers and $\cA_\kappa \subset \cP( \dZ^p)$ the set of probability measures $P$ on $\dZ^p$ such that $\dE \sum_{i=1}^p |X_i |  \leq \kappa$ where $ X = (X_1, \cdots, X_p)$ has law $P$. Then the map $P \mapsto H(P)$ is continuous on $\cA_\kappa$ for the weak topology. 
\end{lemma}

\bp{} A simple truncation argument shows that  
$\cA_\kappa$ is weakly closed. Let $Q_n$ (resp.\ $Q$) be  the law of $ \| X \|_1= \sum_{i=1}^p   |X_i | $ where $X$ has law $P_n$ (resp.\ $P$). If $P_n ^k$ (resp.\ $P^k$) is the conditional law of $P_n$ (resp.\ $P$) conditioned on  $\|X\|_1= k$, we have 
$$
H(P_n) = \sum_{k\geq 0} Q_n (  k ) H( P_n ^k ) + H(Q_n),
$$
and similarly for $P$. Since $P_n ^k$ is a probability measure on a finite set of size $c_k \leq    ( 2 k +1)^p$, we have for any $k$, $Q_n(k) \to Q(k)$, $H(P_n ^k) \to H(P^k)$ as $n \to \infty$. Also, $H(P^k_n) \leq \log (c_k) \leq  p \log (2  k +1) $. Since $\sum_k k Q_n(k) \leq \kappa $, 
using that $x/\log(2x+1)$ is increasing for $x\geq 1$, it follows that for $\theta \geq 1$, 
$$\sum_{k \geq \theta} Q_n(k) H( P_n ^k ) \leq \frac{  p \log ( 2  \theta+ 1)}{\theta} \sum_{k\geq \theta}   k Q_n(k)\leq  \frac{ p \kappa  \log (2  \theta+ 1)}{\theta}.$$ 
This proves the uniform integrability of $k \mapsto H(P_n^k)$ for the measures $Q_n$. 
Hence letting first $n$ and then $\theta$ tend to infinity, we get
$$
\lim_{n \to \infty} \sum_{k\geq 0} Q_n (  k ) H( P_n ^k ) = \sum_{k\geq 0} Q (  k ) H( P^k ). 
$$
It thus remains to prove that $\lim_{n \to \infty} H(Q_n) = H(Q)$. The proof is similar. First, for any $\theta$, $$\lim_{n\to \infty} - \sum_{k < \theta} Q_n (k) \log Q_n(k) = - \sum_{k < \theta} Q(k) \log Q(k).$$ 
Then, we  need to upper bound   $ - \sum_{k \geq \theta} Q_n (k) \log Q_n(k)$, 
uniformly in $n$. It can be done as follows. Observe that $\sum_{k \geq \theta} \sqrt k Q_n(k) \leq \kappa / \sqrt {\theta}$. We then compute 
$$
L ( \delta )= \sup - \sum_{k \geq 0} x_k \log x_k,
$$ 
under the linear constraints, $x_k \geq 0$, $\sum_k x_k \leq 1$ and $\sum_{ k  \geq 0} \sqrt k x_k = \delta$. Using Lagrange multipliers denoted by $\lambda$ and $\mu$, the solution of this convex optimization problem is of the form $x_k = e^{-\mu - \lambda \sqrt k}$ for $k \geq 0$ and $\sum_k x_k = 1$. It is then easy to check that as $\delta \to 0$, $\lambda \delta \to 0$ and $\mu \to 0$. It follows that $L( \delta )  = \mu + \lambda \delta \to 0$. It implies that
$- \sum_{k \geq \theta} Q_n (k) \log Q_n(k) \leq L ( \kappa / \sqrt \theta)$ goes to $0$ as $\theta \to \infty$ uniformly in $n$. Letting $n$ tend to infinity and then $\theta$, it proves that $\lim_{n \to \infty} H(Q_n) = H(Q)$. This concludes the proof of Lemma \ref{le:contH}.
\ep 

We now compute the entropy of $\UGW_h(P)$.
\begin{proposition}\label{pro:enh}
For $h\in\dN$, $P \in \cP_h$ 
and $\dE_P \deg (o) = d$: 
\begin{equation}\label{enh}
\Sigma(\UGW_h(P))=J_h(P).
\end{equation}
\end{proposition}
%
%
\begin{proof}

\noindent{\em Lower bound, finite support:} Consider first the lower bound $\underline \Sigma(\UGW_h(P))\geq J_h(P)$ when $P$ has finite support. By Lemma \ref{le:existenceG}, we may choose a sequence $\G_n\in\cG_{n,m}$ such that $U(\G_n)_h\weak P$, $n\to\infty$ and $U(\G_n)_h$ has support contained $\D:=\{t_1,\dots,t_r\}\subset \cT_h^*$ for all $n$. Let $N_h(\G_n)$ denote the number of graphs $G\in\cG_{n}$ such that $U(G)_h=U(\G_n)_h$. 
Clearly, all such graphs have the same number $m$ of edges.
From Corollary \ref{cor:treelike}, we know that $N_h(\G_n)=
n(\bD)|\cG(\bD,2h+1)|$, where $\bD$ is the neighborhood sequence associated to $\G_n$, i.e.\ if $c=(t,t')\in \cT_{h-1}^*\times \cT_{h-1}^*$,
then $D_c(i) $ is the number of $j\sim i$ in $\G_n$ such that $\G_n(i,j)_{h-1}=t'$ and $ \G_n(j,i)_{h-1}=t$; see \eqref{dcmg}.
Then, 
$$
n(\bD)=\binom{n}{\a_1n,\dots,\a_rn},
$$
where $\a_k=\a_k( n) $ stands for the probability of $t_k$ under $U(\G_n)_h$. 
Since $\a_k\to P(t_k)$ as $n\to\infty$, Stirling's formula yields
\begin{equation}\label{enh2}
\lim_{n\to\infty}\frac1n\log n(\bD) = -\sum_{t}P(t)\log P(t) = H(P) 
\end{equation}
On the other hand, from Corollary \ref{cor:Poilim} we have
\begin{equation}\label{enh3}
\log |\cG(\bD,2h+1)| = 
\frac12\sum_{c\in\cC} (S_c\log S_c - S_c) 
-\sum_{u\in[n]}\sum_{c\in\cC}\log D_c(u)! + o(n),
\end{equation}
where $\cC$ denotes the set of all pairs $c=(t,t')\in\cT_{h-1}^*\times  \cT_{h-1}^*$ 
associated to $\G_n$ as in \eqref{dcmg}, $S_c= S_{\bar c} = \sum_{u\in[n]}D_c(u)$, $\bar c=(t',t)$ if $c=(t,t')$. Note that the size of $\cC$ is finite and independent of $n$. 
For a given $c=(t,t')$, using the notation \eqref{eq:defE} one has 
$
S_c/n \to e_P(t,t').  
$ Also, writing $2m=\sum_{c\in\cC} S_c$, \eqref{enh3} can be rewritten as  
\begin{equation}\label{enh4}
m\log n - m +\frac12  n \sum_{(t,t')}e_P(t,t')\log e_P(t,t') - n  \sum_{(t,t')}\dE_P\log E_h(t,t')!
+ o(n).
\end{equation}
From \eqref{enh2} and \eqref{enh4}, 
\begin{equation}\label{enh5}
\lim_{n\to\infty}\frac1n\left(\log N_h(\G_n) - m\log n\right) = J_h(P).
\end{equation}

To prove the desired lower bound on $\underline \Sigma(\UGW_h(P))$, we may
restrict to graphs $G\in\cG_{n,m}$ with $U(G)_h=U(\G_n)_h$
to obtain
$$
|\cG_{n,m}(\UGW_h(P),\veps)|\geq  N_h(\G_n) \,\dP\left(U(G_n)\in B(\UGW_h(P),\veps)\right),
$$
where $G_n$ is uniformly distributed in $\cG(\bD,2h+1)$ with $\bD$ as above. 
From Theorem \ref{th:convlocCM}, for all $\veps>0$ one has
$$
\lim_{n\to\infty}\frac1n\log 
\dP\left(U(G_n)\in B(\UGW_h(P),\veps)\right) = 0.
$$
Using \eqref{enh5}, we have proved that for all $\veps>0$, $\underline 
\Sigma(\UGW_h(P),\veps)\geq J_h(P)$. Therefore,
\begin{equation}\label{enh6}
\underline \Sigma(\UGW_h(P))\geq J_h(P)\,. 
\end{equation}

\noindent{\em Lower bound, general case:} Set $\rho = \UGW_h(P)$. We can assume that $J_h(P) > - \infty$. For each $n \in \dN$, consider the forest $F_n$ obtained from $(T,o)$ with law $\rho$ by removing all edges adjacent to a vertex with degree higher than $n$. We may define $\rho^{(n)}$ as the the law of $(F_n(o),o)$, the connected component of the root. It is not hard to check that $\rho^{(n)}$ is a unimodular measure. We define $P_n = \rho^{(n)}_h$, the law of its $h$-neighborhood. By construction, $P_n$ is finitely supported, admissible, $P_n$ converges weakly to $P$ and $d_n = \dE_{Q_n} \deg_G (o) \leq d$ converges to $d$. We pick some fixed integer $D > d\vee 2$ and define $R = \delta_{t_\star}$ as the Dirac mass of the $h$-neighborhood of the $D$-regular tree. If $n$ is large enough, there exists $p_n \to 1$ such that $Q_n = p_n P_n + (1-p_n) R$ has mean root degree equal to $d$. Also, $Q_n \in \cP (\cT^*_h)$ is admissible (the set of admissible measures is convex) and has finite support. We apply Lemma \ref{le:lscSi} and the lower bound for finitely supported measures, to obtain 
$$
\underline \Si(\rho) \geq \limsup_{n\to\infty}\underline \Sigma(\UGW_h(Q_n)) \geq \limsup_{n\to\infty} J_h (Q_n). 
$$
By definition, $J_h(Q_n) = -s(d) + H(Q_n) - \frac d 2 H(\pi_{Q_n}) - \sum_{(s,s')} \dE_{Q_n} \log E_h (s,s') !$. We need to prove that $\limsup J_h(Q_n) \geq J_h(P)$. It suffices to prove that 
\begin{equation}\label{eq:lbif00}
\liminf_{n\to\infty} J_h (P_n) \geq J_h(P).
\end{equation}
 First, the lower semi-continuity of the entropy gives 
$
\liminf_{n\to\infty} H(P_n) \geq H(P)$. We now check that 
\begin{equation}\label{eq:lbif1}
\lim_{n \to \infty} \sum_{t,t'} \dE_{P_n} \log E_h (t,t') ! =  \sum_{t,t'} \dE_{P} \log E_h (t,t') !.  
\end{equation}
For ease of notation, we write $\cC = \cT^*_{h-1} \times \cT^* _{h-1}$, $c = (t,t') \in \cC$ and $E_h(c)(\tau)$ to make explicit the dependence in $\tau \in \cT^* _h$. As above, $F_n$ is the forest obtained from $(T,o)$ with law $\rho$, so that 
\begin{align*}
\sum_{t,t'} \dE_{P_n} \log E_h (t,t') ! & =  \dE_P  \,\varphi ( (F_n(o),o) ) ,
\end{align*}
where $\varphi(\tau) = \sum_c \log \PAR{ E_h (c ) (\tau) ! }$ satisfies:
\begin{align}\label{eq:upsumEh}
\varphi(\tau) &\leq \sum_c E_h( c)(\tau)  \log E_h( c)(\tau)  \nonumber\\
&\leq \sum_c E_h( c)(\tau)  \log \Big( \sum_{c'} E_h( c')(\tau)\Big)   = \deg_\tau(o) \log \deg_\tau(o). 
\end{align}
In particular, $\varphi ( (F_n(o),o) )\leq \bar\varphi(T,o):=\deg_T(o) \log \deg_T(o)$. The assumption $P \in \cP_h$ implies that  $\dE_P  \bar\varphi(T,o)< \infty$. Therefore \eqref{eq:lbif1} follows from the dominated convergence theorem. 

To conclude the proof of \eqref{eq:lbif00}, it remains to check that $\limsup H(\pi_{P_n}) \leq H(\pi_P)$, i.e.\
\begin{equation}\label{eq:lbif2}
\liminf_{n \to \infty} \sum_{c \in \cC} e_{P_n} (c)\log e_{P_n} (c) \geq  \sum_{c \in \cC} e_{P} (c) \log e_{P} (c) .  
\end{equation}
For $\theta \in \dN$, we denote by $\cF_\theta\subset \cT^*_h$ the subset of trees whose root vertex has degree bounded by $\theta+1$ and by $\cC_\theta \subset \cC$, the finite subset of pairs of trees with vertex degrees bounded by $\theta$. The assumption $P \in \cP_h$ and Lemma \ref{le:HwPbis} imply that $- \sum_{\tau} \deg_\tau (o)  P(\tau) \log   P (\tau) < \infty$. Also, the assumption $J_h(P) > -\infty$ implies that $H(\pi_P) < \infty$ and $\sum_c \ABS{ e_{P} (c) \log e_{P} (c) }< \infty $. It follows that for any $\veps >0$, there exists $\theta$ such that 
$$
\Big|  \sum_{c \notin \cC_\theta} e_{P} (c) \log e_{P} (c) \Big|\leq \veps  \quad \hbox{ and } \quad  -  \sum_{\tau \notin \cF_\theta } \deg_\tau (o)  P(\tau) \log   P (\tau)  \leq \veps. 
$$ 
By dominated convergence, for any $c\in \cC$, $e_{P_n} (c) \to  e_{P} (c)$. Since $\cC_\theta$ is finite, we find
$$
\limsup_{n \to \infty} \Big|   \sum_{c \in \cC_\theta } e_{P_n} (c)\log e_{P_n} (c)  -  \sum_{c \in \cC} e_{P} (c) \log e_{P} (c) \Big|  \leq \veps. 
$$  
Since $\veps >0$ is arbitrarily small, in order to complete \eqref{eq:lbif2}, it suffices to prove that for any $n \in \dN$, 
$$
\sum_{c \notin \cC_\theta } e_{P_n} (c)\log e_{P_n} (c) \geq - \veps. 
$$
We write 
\begin{align*}
e_{P_n} (c)\log e_{P_n} (c) & =  \sum_\tau P_n (\tau) E_h ( c) (\tau) \log \PAR{  \sum_{\tau'} P_n (\tau') E_h ( c) (\tau') } \\
& \geq   \sum_\tau P_n (\tau) E_h ( c) (\tau) \log \PAR{ P_n (\tau) E_h ( c) (\tau) } \\
& \geq   \sum_{\tau} P_n (\tau) E_h ( c) (\tau) \log  P_n (\tau). 
\end{align*}
It follows that 
\begin{align*}
\sum_{c \notin \cC_\theta } e_{P_n} (c)\log e_{P_n} (c) & \geq  \sum_{\tau} P_n (\tau) \log  P_n (\tau) \sum_{ c \notin \cC_\theta} E_h ( c) (\tau) \\
 & \geq   \sum_{\tau} \deg_\tau (o)  \IND ( \tau \notin \cF_\theta) P_n (\tau) \log   P_n (\tau) ,
\end{align*}
where we use that $\sum_{ c } E_h ( c) (\tau)=\deg_\tau (o)$, and that if $\tau \notin \cF_{\theta}$ and $c \in \cC_\theta$ then $E_h( c)(\tau) = 0$.  
Now, by construction, there exists a partition $\cup_i \cX^i _n $ of $\cT^*_h$ and $\tau^i_n \in \cX^i _n$ such that if $(T,o) \in \cX^i_n$ then $(F_n(o),o) = \tau^i _n$. Also,  $P_n (\tau^i_n) = P ( \cX^i_n) \geq P( \tau^i _n)$, and for all $\tau \in \cX^i _n $, $\deg_\tau (o) \geq \deg_{\tau^i_n} (o)$, $\IND ( \tau \notin \cF_\theta) \geq \IND ( \tau^i _n  \notin \cF_\theta)$. It follows that 
\begin{align*}
\sum_{c \notin \cC_\theta } e_{P_n} (c)\log e_{P_n} (c) & \geq  \sum_i   \sum_{\tau \in \cX^i _n} \deg_\tau (o)   \IND ( \tau \notin \cF_\theta) P(\tau) \log   P_n (\tau^i_n)  \\
& \geq   \sum_{\tau \notin \cF_\theta } \deg_\tau (o)  P(\tau) \log  P (\tau)  \geq  - \veps. 
\end{align*}
This concludes the proof of \eqref{eq:lbif2}.

\bigskip

\noindent{\em Upper bound:} The upper bound $\overline \Sigma(\UGW_h(P))\leq J_h(P)$ is a consequence of the general estimate of Lemma \ref{upperenh} below. 
\end{proof}

\smallskip
\begin{lemma}\label{upperenh}
Fix $h\in\dN$. 
If $\r\in\cP_u(\cT^*)$ is such that $\r_h\in\cP_h$, then
\begin{equation}\label{upperenh1}
\overline\Sigma(\r)\leq J_h(\r_h).
\end{equation}
\end{lemma}
\begin{proof}
\noindent{\em Finite support:} For clarity, we first assume that $P=\r_h$ has finite support. The definition of local weak topology implies that for any $h\in\dN$, any $\veps>0$, there
exists $\eta>0$ such that 
$B ( \rho , \eta ) \subset 
\{ \mu \in \cP ( \cG^ *) : \,d_{TV} ( \mu_h , \rho_h) \leq  \veps \}$, where $d_{TV}$ denotes the total variation distance. Define
$$
A_{n,m} ( P , \veps) = \BRA{ G \in \cG_{n,m} :\, d_{TV} (U(G)_h , P ) \leq \veps }.
$$
Therefore, \eqref{upperenh1} follows if we prove
\begin{equation}\label{upperenh2}
\lim_{\veps \to 0} \LIMSUP_{n\to\infty} \frac1n\left( \log \ABS{A_{n,m} (P, \veps) } -  m \log n \right) \leq  J_h (P).
\end{equation}
Let $\D\subset\cT_h^*$ be the support of $P$. Define $\cF\subset \cT^*_{h-1}$ as the
set of unlabeled rooted trees $t\in \cT^*_{h-1}$ such that either $T(o,v)_{h-1}=t$ or $ T(v,o)_{h-1}=t$
for some $T\in\D$. Set $L=|\cF|$.  Also, by adding a fictitious point $\star$ to $\cF$, define $\bar \cF = \cF\cup \{\star\}$, and call $\bar \cC$ the associated set of $(L+1)\times(L+1)$ colors $c=(t,t')$, $t,t'\in\bar\cF$.   
To any graph $G\in \cG_{n,m}$ we may associate a degree sequence $\bar \bD = (\bar D(1),\dots,\bar D(n))$, where $\bar D(i)$ is a $(L+1)\times(L+1)$ matrix for each $i$, obtained as in \eqref{dcmg} by identifying with $\star$ all neighborhoods that do not belong to $\cF$. The precise construction is defined as follows.
Fix an
edge $\{u,v\}$ of $G$: if 
$G(u,v)_{h-1} =  t'$ and $G(v,u)_{h-1} =t$, with $t,t'\in\cF$, then we say that 
the oriented pair $(u,v)$ has color $c=(t,t')\in\bar\cC$; if  either 
$G(u,v)_{h-1}$ or $G(v,u)_{h-1}$ are not in $\cF$, then we say that the   
oriented pair $(u,v)$ has color $(\star,\star)\in\bar\cC$. This defines a directed colored graph $\wt G$ with colors from the set $\bar\cC$. We call $\bar\bD$ the corresponding degree sequence, i.e.\ $\bar D_c(i)$ is the number of directed edges with color $c$ going out of vertex $i$. Note that by construction, if $(u,v)$ has color $c$, then $(v,u)$ has 
color $\bar c$, and that there is no edge with color $(t,\star)$ or $(\star, t)$ for any $t\in \cF$.  

In this way a graph $G\in \cG_{n,m}$ yields an element $\wt G$ of $\wcG(\bar\bD)$.
Let $\bar Q(G)$ denote the empirical degree law 
\begin{equation}\label{upperenh3}
\frac1n\sum_{i=1}^n\d_{\bar D(i)}\,.
\end{equation}
Thus $\bar Q(G)$ is a probability measure on the set $\cM_{L+1}$; see Eq.\ \eqref{dcu}.
Also, let $\bar P$ denote the probability measure on $\cM_{L+1}$ induced by $P$. Namely, $\bar P$ is the law of the random matrix $\bD\in \cM_{L+1}$ defined as follows: for all $c= (t,\star)$, or $c=(\star, t)$ or $c=(\star,\star)$, set $D_c=0$; and for $c=(t,t')$ with $t,t'\in\cF$, set $D_c=E_h(t',t)$, where $E_h(t',t)$ is defined by \eqref{eq:defE} if the rooted graph $(G,o)$ has law $P$. By contraction, one has $H(\bar P) \leq H(P)$ and 
$$
d_{TV} (\bar Q(G) ,\bar  P )\leq d_{TV} (U(G)_h , P ).
$$
Let $\cP_{n,m}(P,\veps)$ denote the set of probability measures $Q\in \cP(\cM_{L+1})$ of the form \eqref{upperenh3}, satisfying $\sum_{i\in[n]}\sum_{c\in\bar\cC}\bar D_c(i)=2m$, and 
such that $ d_{TV} (Q,\bar  P )\leq \veps$. 
The above discussion shows that if $G\in A_{n,m} (P, \veps)$, there must exist $Q\in \cP_{n,m}(P,\veps)$ such that $\bar Q(G)=Q$. 
Therefore, one obtains 
 \begin{equation}\label{upperenh4}
|A_{n,m} (P, \veps)| \leq |\cP_{n,m}(P,\veps)| \max_{Q\in\cP_{n,m}(P,\veps)} 
n(\bar\bD) \big|\wcG(\bar\bD)\big|
\end{equation}
where $n(\bar\bD)$ is defined as in Corollary \ref{cor:treelike}, and $\bar\bD$ is the degree vector associated to $Q$ as in \eqref{upperenh3}. 

Next, we claim that for each $\veps>0$,
\begin{equation}\label{eq:integerpartition}
\lim_{n\to\infty}\frac1n\log|\cP_{n,m}(P,\veps)| = 0.
\end{equation}
Indeed, let $p = |\bar\cC|$ and fix a vector $\ell \in \dZ_+^{p}$. An integer partition of the vector $\ell$ is an unordered sequence $\{d(1), \cdots, d(k)\}$, with  $d(i)\in\dZ_+^ p$ for all $i$, and such that  
$d(1) + \cdots + d(k) = \ell$ componentwise. By \cite[Lemma 4.2]{MR2759726}, if $\sum_{i=1}^ p \ell_i = 2 m$ then the number of integer partitions of $\ell$ is $\exp( o (m) )$. The number of vectors  $\ell \in \dZ_+^{p}$ such that $\sum_{i=1}^ p \ell_i = m$ is bounded by $(m+1)^p$. It follows that the number of unordered sequences  $\{d(1), \cdots, d(n)\}$ in $\dZ_+^p$ such that $\sum_{i=1}^n \sum_{c = 1} ^ p d_c(i) = 2 m$ is at most $\exp(o(n))$, for $m=O(n)$. Now, if $Q$ is of the form \eqref{upperenh3} we may define $d_c(i)=\bar D_c(i)$, for every $c\in\bar\cC$ and $i\in[n]$, which yields  an injective map from $\cP_{n,m}(P,\veps)$ to the unordered sequences $\{d(1), \cdots, d(n)\}$, with $d(i)\in \dZ_+^p$ such that $\sum_{i=1}^n \sum_{c \in\bar\cC} d_c(i) = 2 m$. This proves \eqref{eq:integerpartition}.
 
From \eqref{upperenh4} and \eqref{eq:integerpartition}, 
to prove \eqref{upperenh2}, it remains to show that
\begin{equation}\label{upperenh6}
\limsup_{n\to\infty}\max_{Q\in\cP_{n,m}(P,\veps)}\frac1n\,\left[ \log 
 \left(n(\bar\bD) \big|\wcG(\bar\bD)\big|
\right) -  m \log n\right]  \leq J_h (P) + \eta(\veps),
\end{equation}
 where we use the notation  $\eta(\veps)$ for an arbitrary function satisfying $\eta(\veps)
\to 0$ as $\veps\to 0$.  
Since $\bar\cC$ is finite, reasoning as in \eqref{enh2} and using Lemma \ref{le:contH}, it is easily seen that
\begin{equation}\label{upperenh7}
\limsup_{n\to\infty}\max_{Q\in\cP_{n,m}(P,\veps)} \frac1n\log 
 n(\bar\bD)
 \leq   H(\bar P) + \eta(\veps) \leq H(P) + \eta (\veps).
\end{equation}
Moreover, as in \eqref{enh3} one has
$$ \log \big|\wcG(\bar\bD)\big| = 
\frac12\sum_{c\in\bar\cC} (\bar S_c\log \bar S_c - \bar S_c) 
-\sum_{u\in[n]}\sum_{c\in\bar\cC}\log \bar D_c(u)! + o(n),
$$
where $\bar S_c = \sum_{i\in[n]}\bar D_c(i)$. 
 Observe that \begin{equation}\label{upperenh8}
 \limsup_{n\to\infty}
 \max_{Q\in\cP_{n,m}(P,\veps)}|\bar S_c/n - e_P(c)| \leq \eta(\veps).
 \end{equation} 
 Indeed, if $Q_n$ is a sequence with $Q_n\in\cP_{n,m}(P,\veps)$, then $\bar S_c/n = \dE_{Q_n}D_{c}$, where $D\in\cM_{L+1}$ has law $Q_n$. Then, for any $k\in\dN$, $\bar S_c/n \geq \dE_{Q_n}[D_{c}\wedge k]$, and since $Q_n\in\cP_{n,m}(P,\veps)$ and $D_{c}\wedge k$ is a bounded function, taking first $\veps\to 0$ and then $k\to\infty$, one has  $\bar S_c/n \geq e_P(c) - \eta(\veps)$ uniformly in $n$. Moreover, since $\sum_{c\in\bar\cC}\bar S_c/n = 2m/n = d + o (1)$, one has $\bar S_c/n = d + o(1) - \sum_{c'\neq c}\bar S_{c'}/n $. 
 Therefore, from the lower bound $\bar S_c/n \geq e_P(c) - \eta(\veps)$ and the fact that $\sum_c
 e_P(c)=d$, one finds $\bar S_c/n \leq e_P(c) +|\bar\cC|\eta(\veps) + o(1)$. This ends the proof of \eqref{upperenh8}. Moreover, with the same truncation argument as above one has that 
 $$\frac1n \sum_{u\in[n]}\log \bar D_c(u)!\geq \dE_P[\log E_h(c)!] - \eta(\veps),$$
 for all $c\in\bar\cC$. This, together with \eqref{upperenh7}-\eqref{upperenh8} and the argument in \eqref{enh4} allows us to conclude the proof of \eqref{upperenh6}. This ends the proof of \eqref{upperenh2}.

\noindent{\em General case:} We now come back to the case of arbitrary $P \in \cP_h$. For any finite set $\Delta \subset \cT^*_h$, we associate the sets $\cC=\cC(\D)$ and $\bar \cC$ as above. The above argument establishes that 
\begin{equation}\label{upperenh20}
\lim_{\veps \to 0} \LIMSUP_{n\to\infty} \frac1n\left( \log \ABS{A_{n,m} (P, \veps) } -  m \log n \right) \leq  J^\Delta_h (P),
\end{equation}
where $J^\Delta _h (P) := -s(d) + H(P) - \frac d 2 H ( \pi_{\bar P}) - \sum_{(t,t') \in \cC}\dE_ P \log E_h ( t,t') ! - \dE_P \log E_h (\star, \star) ! $ and $\pi_{\bar P} \in \cP ( \bar \cC)$ is defined as follows: for $c = (t,t') \in  \cC$, $\pi_{\bar P} (t,t') = \pi_ P ( t,t')$ and for $c = (\star , \star)$, 
$$
\pi_{\bar P} (\star, \star) = 1 - \sum_{(t,t') \in \cC} \pi_{P} (t,t') = \frac 1 d \dE_P\Big| \BRA{ v \stackrel{T}{\sim} {o} : T(o,v) \hbox{ or } T(v,o)_{h-1} \hbox{ is not in $\cC$}}\Big|.  
$$
Assume first that $J_h(P)  > - \infty$. Using \eqref{eq:upsumEh} at the second line, one has
\begin{align*}
J^\Delta _h (P)  &\leq  J_h (P) - \frac d 2 \sum_{(t,t') \notin \cC}  \pi_ P ( t,t') \log  \pi_ P ( t,t') +  \sum_{(t,t') \notin \cC}  \dE_P \log E_h (t,t') !  \\
& \leq   J_h (P) - \frac d 2 \sum_{(t,t') \notin \cC}  \pi_ P ( t,t') \log  \pi_ P ( t,t') + \dE_P \SBRA{  \IND ( T \notin \Delta)  \deg_T (o ) \log \deg_T (o)}.
\end{align*}
We may then consider a sequence $(\Delta_k)$ of finite subsets in $\cT^*_h$ such that $P  ( T \notin \Delta_k) \to 0$, and  $\sum_{c \notin \cC(\Delta_k)}  \pi_ P ( c) \log  \pi_ P ( c) \to 0$, as $k \to \infty$. Then as $k \to \infty$, the above expression converges to $J_h (P)$. This proves that \eqref{upperenh1} holds when $P \in \cP_h$ and $J_h (P) > -\infty$. 

If $P \in \cP_h$ and $J_h (P) = -\infty$ then either $H(\pi_P) = \infty$ or $\sum_{(t,t')} \dE_P \log E_h (t,t) ! = \infty$.  We use the upper bound
$$
J^\Delta _h (P) \leq -s(d) + H(P)  + \frac d 2 \sum_{(t,t') \in \cC}  \pi_ P ( t,t') \log  \pi_ P ( t,t') - \sum_{(t,t') \in \cC}\dE_ P \log E_h ( t,t') ! 
$$
We may consider a sequence $(\Delta_k)$ of finite subsets of $\cT^*_h$ such that, as $k\to\infty$, one has  $\sum_{c\in \cC(\Delta_k)}  \pi_ P ( c) \log  \pi_ P (c) \to -H(\pi_P)$ and $ \sum_{c \in \cC(\Delta_k)}\dE_ P \log E_h (c) ! \to \sum_{c}\dE_ P \log E_h ( c) ! $, and therefore $J^{\Delta_k} _h (P)\to-\infty$, $k\to\infty$. This completes the proof of   \eqref{upperenh1}. 
  \end{proof}

Next, we  extend Lemma \ref{upperenh} to the case $\rho_h \notin \cP_h$, i.e. $H(\r_h)=\infty$ or $\dE_\rho \deg_T (o) \log \deg_T(o) = \infty$. We start with the latter case. 

\begin{lemma}\label{le:phgap}
If $\rho \in \cP_u ( \cT^*)$ is such that $\dE_\rho \deg_T (o) = d$ and $\dE_\rho \deg_T (o) \log \deg_T(o) = \infty$ then 
$$
\overline \Si( \rho) = -\infty.
$$
\end{lemma}
\bp{}
We set $P = \rho_1$ which can be identified with a probability measure on $\dZ_+$. Since $P$ has finite first moment, $H(P)$ is finite. The proof of Lemma \ref{upperenh} can be simplified for $h=1$: since $\cT^* _{h-1}$ has a unique element (the isolated root), one has $H(\pi_P)=0$ and it is not necessary to consider the extra state $\star$.  The bound \eqref{upperenh2} gives 
$$
\lim_{\veps \to 0} \LIMSUP_{n\to\infty} \frac1n\left( \log \ABS{A_{n,m} (P, \veps) } -  m \log n \right) \leq  -s(d) + H(P) - \dE_ P \log \deg_T (o) ! 
$$
Now, from Stirling's approximation, for $n \geq 1$, $n ! \geq c \sqrt {n} e^{-n} n ^n$ for some constant $c >0$. We deduce that $\log n ! \geq  c' - n  + n \log n $ for some constant $c'>0$. In particular, from $\dE_P \deg_T(o) = d < \infty$ and $\dE_\rho \deg_T (o) \log \deg_T(o) = \infty$, we get that  $\dE_ P \log \deg_T (o) !  = \infty$. 
\ep

The following statement is the  extension of Lemma \ref{upperenh} to the case $\rho_h \notin \cP_h$.
\begin{proposition}\label{proupperenh}
If $\r\in\cP_u(\cT^*)$, then  for any $h\in\dN$, 
\begin{equation}\label{upperenh11}
\overline\Sigma(\r)\leq \overline J_h(\r_h),
\end{equation}
where $\overline J_h(\r_h)=J_h(\r_h)$ if $\r_h\in\cP_h$, and $\overline J_h(\r_h)=-\infty$ otherwise.
\end{proposition}

In view of  Lemma \ref{upperenh} and Lemma \ref{le:phgap}, Proposition \ref{proupperenh} is a consequence of the following lemma. 
\begin{lemma}\label{le:phgap2}
Let $\r\in\cP_u(\cT^*)$ be such that  $\dE_\rho \deg_T (o) \log \deg_T(o) < \infty$. Then for any $h \in \dN$, 
$
H(\rho_h) < \infty.
$ Consequently, for any $P \in \cP(\cT^*_h)$, $P \in \cP_h$ is equivalent to $P$ admissible and $\dE_P \deg_T (o) \log \deg_T(o) < \infty$.  
\end{lemma}

\begin{proof}
The second statement follows from the first applied to $\rho = \UGW_h (P)$ with $P$ admissible.  We now prove the first statement. Since $d:=\dE_\r \deg(o)$ is finite, one has $H(\r_1)<\infty$. To prove the lemma, we proceed by induction, and show that for any $h\in\dN$, if $H(\r_{h})<\infty$ and $\dE_\rho \deg_T (o) \log \deg_T(o) < \infty$ then $H(\r_{h+1})<\infty$. 
Set $P=\r_h$, $Q=\r_{h+1}$, and $Q^*=[\UGW_h(P)]_{h+1}$. Assume that $H(P)<\infty$. We are going to prove
that $H(Q)<\infty$.
Observe that 
$$
H(Q) = 
H(P)+ \sum_{\g\in\cT^*_h} P(\g)H(Q(\cdot|\g)),
$$
where 
$Q(\cdot|\g)$ stands for the conditional distribution of the $(h+1)$-neighborhood 
given the $h$-neighborhood $\g$. 
Also,
\begin{align*}
\sum_{\g\in\cT^*_h}P(\g)H(Q(\cdot|\g)) &= -\sum_{\t\in\cT^*_{h+1}}Q(\t)\log \frac{Q(\t)}{P(\t_h)}\\&
= - \sum_{\t\in\cT^*_{h+1}}Q(\t)\log \frac{Q^*(\t)}{P(\t_h)} - H(Q|Q^*)\leq - \sum_{\t\in\cT^*_{h+1}}Q(\t)\log \frac{Q^*(\t)}{P(\t_h)} .
\end{align*}
Now recall that $\t\in\cT^*_{h+1}$ determines all the coefficients $E_{h+1}(t,t')$, $(t,t')\in\cT^*_{h}\times\cT^*_{h}$, and these can be 
partitioned according to the pairs $(s,s')\in\cT^*_{h-1}\times\cT^*_{h-1}$ such that $t_{h-1}=s$, $t'_{h-1}=s'$. 
With this notation, by definition of $Q^*$, one has, for $\t\in\cT^*_{h+1}$ such that $\t_h=\g$:
\begin{equation}\label{enah06}
\frac{Q^*(\t)}{P(\t_h)} = Q^*(\t|\g) = \prod_{(s,s')\in\cT^*_{h-1}\times\cT^*_{h-1}} \binom{E_{h}(s,s')}{\{E_{h+1}(t,t')\}} \prod_{t\in\cT^*_h} \wP_{s,s'}(t)^{k_{t,s'}(\t)},
\end{equation}
where the terms $\{E_{h+1}(t,t')\}$ in the multinomial coefficient are all such that $t_{h-1}=s$, $t'_{h-1}=s'$, 
and we write $k_{t,s'}(\t):=|\{v \stackrel{\t}{\sim }o: \, \t(o,v)_h=t, \,\t(v,o)_{h-1}=s'\}|$, with $t_{h-1}=s$.
Therefore,
$$
- \sum_{\t}Q(\t)\log \frac{Q^*(\t)}{P(\t_h)} \leq - \sum_{s,s'}\sum_{t:\,t_{h-1}=s}\sum_{\t}Q(\t)k_{t,s'}(\t)\log \wP_{s,s'}(t).
$$
Moreover, unimodularity yields
\begin{equation}\label{unimk}
\sum_{\t} Q(\t)\,k_{t,s'}(\t) =\dE_\r |\{v\sim o: \, T(o,v)_{h-1}=s', \, T(v,o)_{h}=t\}|
 =\wP_{s,s'}(t)e_P(s,s')
\end{equation}
Thus,
$$
- \sum_{\t}Q(\t)\log \frac{Q^*(\t)}{P(\t_h)} \leq - \sum_{s,s'}\sum_{t:\,t_{h-1}=s}e_P(s,s')
\wP_{s,s'}(t)\log \wP_{s,s'}(t) = \sum_{s,s'}e_P(s,s')\,H( \wP_{s,s'}).
$$
In conclusion, we have obtained that
$$
H(Q) \leq H(P)+  \sum_{s,s'}e_P(s,s')\,H( \wP_{s,s'}).
$$
The proof will be complete once we show that $H(P)<\infty$ and $\dE_\rho \deg_T (o) \log \deg_T(o) < \infty$ imply that $\sum_{s,s'}e_P(s,s')\,H( \wP_{s,s'})<\infty.$

Now, by definition, if $\g = t \cup s'_+$ and $n_{t,s'} = \big|\{ v \stackrel{\g}{\sim} o : \g(v,o) = t , \g (o,v) = s' \}\big|$, we have 
\begin{align*}
\sum_{s,s'} e_P ( s,s') H ( \wP_{s,s'}) & =   \sum_{s,s'}\sum_{ t : \,t_{h-1} = s} n_{t,s'}  P ( t \cup s'_ + ) \log \frac{ e_P (s,s') }{ n_{t,s'}  P (t \cup s'_+)} \\
& \leq  \sum_{s,s'}\sum_{ t :\, t_{h-1} = s} n_{t,s'}  P ( t \cup s'_ + ) \log \frac{ d }{  P (t \cup s'_+)} \\
& =  \sum_{\g \in \cT^*_h} P(\g) \log \PAR{ \frac{d}{P(\g) }}\sum_{s,s'} E_h (s',s) (\g),
\end{align*}
where we have used that $e_P (s,s')\leq d\, n_{t,s'}$, and that 
 $1\leq n_{t,s'} \leq E_h (s',s) (\g)$ where $\g = t \cup s'_+$ and $t_{h-1} = s$. Since $\deg_\g (o) = \sum_{s,s'}E_h (s,s') (\g)$, we find
\begin{align*}
\sum_{s,s'} e_P ( s,s') H ( \wP_{s,s'}) \leq d \log d  -  \sum_{\g \in \cT^*_h} \deg_\g (o) P(\g) \log P(\g)  .
\end{align*}
It remains to apply Lemma \ref{le:HwPbis} with $\cX = \cT^*_h$ and $\ell_x = \deg_x (o)$ together with the assumption $\dE_\rho \deg_T (o) \log \deg_T(o) < \infty$.
\end{proof}

\begin{lemma}\label{strict}
Suppose $\r\in\cP_u(\cT^*)$. 
Then $\overline J_k(\r_k)$, $k\in\dN$,  is a non-increasing sequence. Assume moreover that $\r_1$ has finite support. Then for fixed $k>h$, one has 
$\overline J_k(\rho_k) < \overline J_h ( \r_h)$ if and only if $\r_k\neq [\UGW_h(\r_h)]_k$. In particular, if $\r_1$ has finite support, then
for any $h\in\dN$, one has $\overline \Sigma(\r)< \overline J_h(\r_h)$ if and only if $\r\neq \UGW_h(\r_h)$. 
\end{lemma}
\begin{proof}
Fix $h\in\dN$. To prove that $\overline J_{h+1}(\r_{h+1})\leq \overline J_h(\r_h)$, we may assume that $\r_{h+1}\in\cP_{h+1}$. In this case one has also that $\r_{h}\in\cP_{h}$. 
%
%
From Proposition \ref{pro:enh}, we know that $\Sigma(\UGW_k(\r_k))=J_k(\r_k)$ for both $k=h$ and $k=h+1$.  Therefore,
using Lemma \ref{upperenh} one has
$$
J_{h+1} (\r_{h+1}) = \Sigma(\UGW_{h+1}(\r_{h+1}))\leq J_h([\UGW_{h+1}(\r_{h+1})]_h) = J_h(\r_h)\,,
$$
where we use $[\UGW_k(\r_k)]_h=\r_h$, $k\geq h$. This proves that $\overline J_k(\r_k)$ is non-increasing in $k$.

We now assume that $\r_1$ has finite support. 
Then, by unimodularity it follows that $\r_h$ has finite support for all $h\in\dN$. In particular, $\r_h\in\cP_h$ and $J_h(\r_h)>-\infty$ for all $h\in\dN$. Fix $k>h$. Suppose that $\overline J_k(\rho_k) < \overline J_h ( \r_h)$. 
One has $\Sigma(\UGW_k(\r_k))=J_k(\r_k)$ by Proposition \ref{pro:enh}. 
From the consistency property of Lemma \ref{le:consistency}, one must then have $\r_k\neq [\UGW_h(\r_h)]_k$. 

Next, suppose that $\r_1$ has finite support and that
$\r_k\neq [\UGW_h(\r_h)]_k$
and let us show that $J_k(\rho_k) < J_h ( \r_h)$. 
If $\G_n\in\cG_{n,m}$ is a sequence with $U(\G_n)_k\weak \r_k$, then also $U(\G_n)_h\weak \r_h$ and by \eqref{enh5} one has
\begin{equation}\label{enh50}
J_k(\r_k)-J_h(\r_h)=\lim_{n\to\infty}\frac1n\left(\log N_k(\G_n) - \log N_h(\G_n)\right).
\end{equation}
Using Corollary \ref{cor:treelike}, if $\widehat G_n$ denotes a random graph with uniform distribution in $\cG(\bD^{(n)},2h+1)$, $\bD^{(n)}$ being the degree vector associated to the $h$-neighborhood of $\G_n$, one also has
\begin{equation}\label{enh52}
J_k(\r_k)-J_h(\r_h)=\lim_{n\to\infty}\frac1n\log \dP\big(U(\widehat G_n)_k=U(\G_n)_k\big).
\end{equation}
Since $\rho_k \neq \g_k:=[\UGW_h(\r_h)]_k$, there exist $\veps >0$ and an event $A$ of the form $A = \{ g \in \cG^* : g_k= t \}$ for some $t\in \cT^*_k$, such that $| \rho_k(A)  -  \g_k (A) | > \veps$. 
Therefore, $U(\G_n)_k\weak \r_k$ implies that 
$$
J_k(\r_k)-J_h(\r_h) \leq \limsup_{n\to\infty} 
\frac 1 n \log \dP \PAR{ | U(\widehat G_n)_k(A)  - \g_k (A) | > \veps/2 }.
$$
By Proposition \ref{prop:convlocCM}, $\dE U(\widehat G_n) (A)$ converges to $\g_k(A)$. It follows that 
$$
J_k(\r_k)-J_h(\r_h) \leq\limsup_{n\to\infty}  \frac 1 n \log \dP \PAR{ | U(\widehat G_n)(A)  - \dE U ( \widehat G_n)  (A) | > \veps/3 }.
$$
The desired conclusion $J_k(\r_k)-J_h(\r_h)<0$ now follows from 
\eqref{eq:concunifCM} (in Appendix).

Finally, the assertion concerning $ \overline \Sigma(\r)$ follows easily from the results above.
Indeed, from Proposition \ref{pro:enh} we know that $\overline\Sigma(\r)<J_h(\r_h)$ implies that
$\r\neq \UGW_h(\r_h)$. For the opposite direction, observe that if $\r\neq \UGW_h(\r_h)$, then 
$\r_k\neq [\UGW_h(\r_h)]_k$ for some $k>h$. From Lemma \ref{upperenh} one has
$\overline\Sigma(\r)\leq J_k(\r_k)$, and the above implies $\overline\Sigma(\r)<J_h(\r_h)$.
\end{proof}

\begin{lemma}\label{limitenh}
Suppose $\r\in\cP_u(\cT^*)$.
Then 
\begin{equation}\label{limitenh1}
\Sigma(\r)=\overline J_\infty(\r):=\lim_{k\to\infty}\overline J_k(\r_k).
\end{equation}
\end{lemma}
\begin{proof}
The limit $\overline J_\infty(\r)$ is well defined by the monotonicity in Lemma \ref{strict}.
The upper bound in Proposition  \ref{proupperenh} shows that $\overline\Sigma(\r)\leq \overline J_\infty(\r)$. Thus,
all we have to prove  is
\begin{equation}\label{limitenh11}
\underline\Sigma(\r)\geq \overline J_\infty(\r).
\end{equation}
We may assume that $\r_k\in\cP_k$ for all $k\in\dN$. 
Fix $\eta >0$ and set $\rho^ h = \UGW_h ( \rho_h)$. 
By the lower bound in Proposition  \ref{pro:enh}, for any $h \in\dN$, $\veps>0$ and $n \geq n_0(\veps,h,\eta)$, 
$$
x (n , h, \veps) := \frac1n\PAR{\log \ABS{ \cG_{n,m} (\rho^h , \veps) } - m \log n } \geq \overline  J_\infty(\rho) - \eta. 
$$ 
By diagonal extraction, there exist sequences $h_n \to \infty$ and $\veps_n \to 0$ such that 
$$
\liminf_{n\to\infty}x (n , h_n, \veps_n)  \geq\overline  J_\infty (\rho) - \eta. 
$$
Since $\rho^ {h_n} \weak \rho$, for any fixed $\veps >0$ and all $n$ large enough, 
$
B ( \rho^{h_n} , \veps_n ) \subset B ( \rho , \veps ).  
$
In particular, $\ABS{ \cG_{n,m} (\rho , \veps) } \geq \ABS{ \cG_{n,m} (\rho^{h_n} , \veps_n ) }$. It follows that 
$
\underline \Si (\rho,\veps)  \geq\overline  J_\infty (\rho)- \eta.
$
The latter holding for all $\veps>0$ and $\eta >0$, we have checked that \eqref{limitenh11} holds.
\end{proof}
All the statements in  Theorem \ref{th:entropy2} are contained in Proposition \ref{pro:enh}, Proposition \ref{proupperenh}, 
Lemma \ref{strict} and Lemma \ref{limitenh}. 
Moreover, Lemma \ref{limitenh} implies that $\Sigma(\r)$ is well defined and equals $\overline J_\infty(\r)$ 
for every $\r\in\cP_u(\cT^*)$, independently of the choice of the sequence $m=m(n)$ with $m/n\to d/2$.
This completes  the proof of Theorem \ref{th:entropy1} and Theorem \ref{th:entropy2} .

\subsection{Proof of Corollary \ref{cor:entropy}}
\label{subsec:corentropy}
In the special case $h = 1$, one has $P\in\cP(\dZ_+)$, and the condition $\sum_{n=0}^\infty nP(n)=d$ implies $H(P)<\infty$. By Proposition \ref{pro:enh} one has $\Sigma(\UGW_1(P)) = J_1(P)$.
Moreover, since $|\cT^*_0|=1$, there exists a unique type $(s,s')\in\cT^*_0\times \cT^*_0$ with $e_P(s,s')= d$ and therefore $H ( \pi_P)=0$, and  
$$
\sum_{(s,s')\in\cT^*_{h-1}\times \cT^*_{h-1}}   \dE_P \log (E_h(s,s')!) = \sum_{n=0}^\infty P(n) \log (n!).$$
It follows that 
\begin{align*}
J_1(P)&=-s(d) -\sum_{n=0}^\infty P(n) \log P(n) - \sum_{n=0}^\infty P(n) \log (n!)\\
&=-s(d)  + d - d \log d - \sum_{n=0}^\infty P(n) \log \frac{P(n)n!}{d^ne^{-d}} = s(d) - H(P\,|\,\POI(d)).
\end{align*}
This ends the proof of Corollary \ref{cor:entropy}. 

\begin{remark}\label{altrem}
{\em Fix $h\in\dN$, and suppose that $\r\in\cP_u(\cT^*)$ is such that $\r_h\in\cP_h$. One can derive the following alternative expression for $J_h(\r_h)$ in terms of relative entropies:
\begin{align}\label{enrel1}
J_h(\r)=s(d) - \sum_{k=1}^h \D_k(\r)\,,
\end{align}
where $\D_1(\r)=  H(\r_1\,|\,\POI(d))$ and, for $k\geq 2$:
\begin{align}\label{enrel2}
\D_k(\r) = 
H(\r_k \,|\,\r_k^*) - \frac{d}2 \, H(\pi_{\r_k}\,|\,\pi_{\r_k^*})\geq 0,
\end{align}
where $\r_k^*:=[\UGW_{k-1}(\r_{k-1})]_k$. 
To prove \eqref{enrel1}, thanks to Corollary \ref{cor:entropy}, it suffices to prove that the increment $J_{k-1}(\r_{k-1}) - J_k(\r_k)$ equals \eqref{enrel2} for $k\geq 2$. This in turn can be checked as follows.

Fix $h\in\dN$, and write $Q=\r_{h+1}$, $Q^*=\r^*_{h+1}$, $P=\r_h$.
Simple manipulations show that 
\begin{align}\label{enh55}
&J_h(P)-J_{h+1}(Q) \\
&\quad =  -\sum_{t} P(t) H(Q(\cdot| t))+\frac{d}2\sum_{(s,s')}\pi_{P}(s,s') H(q(\cdot|s,s'))
-\sum_{(s,s')} \dE_Q \Big[\log \binom{E_h(s,s')}{\{E_{h+1}(t,t')\}}
\Big], \nonumber
\end{align}
where $t\in\cT^*_{h}$ while $(s,s')\in\cT^*_{h-1}\times \cT^*_{h-1}$, we use the multinomial coefficients introduced in \eqref{enah06}, and we define the conditional probability $q(\cdot|s,s')$ on $\cT^*_{h-1}\times \cT^*_{h-1}$ by $\pi_Q(t,t')=\pi_{P}(s,s')q(t,t'|s,s')$. 
Using \eqref{enah06} and \eqref{unimk}, one finds 
\begin{align*}
 H(Q|Q^*)=-\sum_t P(t) H(Q(\cdot| t)) - 
\sum_{(s,s')} \dE_Q \Big[\log \binom{E_h(s,s')}{\{E_{h+1}(t,t')\}}
\Big] +d \sum_{(s,s')}\pi_P(s,s')H(\wP_{s,s'}).
\end{align*}
Therefore, 
\begin{align}\label{enh500}
&J_h(P)-J_{h+1}(Q)
= H(Q|Q^*) + \frac{d}2\sum_{(s,s')}\pi_{P}(s,s') [H(q(\cdot|s,s')) - 2 H(\wP_{s,s'})].
\end{align}
Next observe that  if $q^*(\cdot|s,s'):=\wP_{s,s'}(t)\wP_{s',s}(t')$, then $\pi_{Q^*}(t,t')=\pi_{P}(s,s')q^*(\cdot|s,s')$, see Remark \ref{boh}. Moreover, using 
$$
\sum_{t'\in\cT^*_{h}:\;t'_{h-1}=s'}q(t,t'|s,s') = \wP_{s,s'}(t) = \sum_{t'\in\cT^*_{h}:\;t'_{h-1}=s'}q^*(t,t'|s,s'),
$$
one finds 
$$H(q(\cdot|s,s')) - 2 H(\wP_{s,s'}) = - H(q(\cdot|s,s')|q^*(\cdot|s,s')).$$ It follows that
 \begin{align*}
&\frac{d}2\sum_{(s,s')}\pi_{P}(s,s') H(q(\cdot|s,s')\,|\, q^*(\cdot|s,s'))=\frac12 \sum_{(t,t')} \dE_Q (E_{h+1}(t,t'))\log \frac{\dE_Q (E_{h+1}(t,t'))}{\dE_{\bar \r }(E_{h+1}(t,t'))}
\\
&\qquad\qquad\qquad\qquad\qquad=\frac{d}2
 \sum_{(t,t')} \pi_{Q}(t,t')\log \frac{\pi_{Q}(t,t')}{\pi_{Q^* }(t,t')} = \frac{d}2 H(\pi_Q\,|\,\pi_{Q^*}),
\end{align*}
where $(s,s')\in\cT^*_{h-1}\times \cT^*_{h-1}$, while $(t,t')\in\cT^*_{h}\times \cT^*_{h}$.
From \eqref{enh500} we then obtain the desired conclusion $J_h(P)-J_{h+1}(Q) = \D_{h+1}(\r)$.
Clearly, the monotonicity in Lemma \ref{strict} implies that $\D_{h+1}(\r)\geq 0$. This yields the seemingly nontrivial inequality
$\frac{d}2 H(\pi_Q\,|\,\pi_{Q^*})\leq H(Q|Q^*)$. 
}
\end{remark}

\subsection{Discontinuity of the entropy}

\label{susbec:discentropy}

The aim of this section is to prove that the $\cP(\cG^*) \to [-\infty,\infty)$ map $\rho \mapsto \Si (\rho)$ is discontinuous for the weak topology at $\rho = \UGW_1( P)$ for any finitely supported
$P \in \cP(\dZ_+)$  with $P(0) = P(1) = 0$ and  $P(2) < 1$.

Let $P_1,P_2$ be two probability measures on $\dZ_+$ with finite positive means, say $d_1$ and $d_2$. For $i = 1 , 2$, we set $p_i = d_{\bar i} / ( d_1 + d_2)$, where $\bar 1 = 2 $ , $\bar 2 = 1$. We define $\UGW(P_1,P_2)$ as the law of the rooted tree $(T, o)$ obtained as follows. We first build a rooted multi-type Galton-Watson tree $(\check T,o)$. The vertices can be of type $1$ or of type $2$. The root has type $i$ with probability $p_i$. All offspring of a vertex of type $i$ are of type $\bar i$. Conditioned on being of type $i$, the root has a number of offspring distributed according to $P_i$. Conditioned on being of type $i$, a vertex different from the root has a number of offspring distributed according to the size-biased law $\wP_i$ given by \eqref{eq:defwP}. The tree $(T,o)$ is finally  obtained from $(\check T, o)$ by removing the types.

The distribution of $(\check T, o)$ is unimodular. It implies that $\UGW(P_1,P_2)$ is also unimodular. It can be checked directly from the definition of unimodularity or by proving that it is the local weak limit of bipartite configuration models (they are especially of interest  in coding theory, see e.g.\ Montanari and M\'ezard \cite{MR2518205}).

Now, let $S \subset \dZ_+$ be a finite set and $P$ be a probability measure on $S$. Observe that $\UGW(P,P) = \UGW_1 (P)$ and that if $P_n$ is a sequence of probability measures on $S$ such that $P_n \weak P$ then $\wP_n \to \wP$ and
$$
\UGW(P,P_n) \weak \UGW_1(P).
$$
However, we have the following discontinuity result:
\begin{proposition}\label{prop:discSi}
Assume that $ S \subset \dZ_+ \backslash\{ 0 , 1\}$, $P, P_n\in\cP(S)$, and $P_n \weak P$ as $n\to\infty$. Assume further that $P(2) < 1$ and that  $P_n \ne P$ for all $n$ large enough. Then, 
$$
\LIMSUP_{n\to\infty} \, \Si ( \UGW ( P, P_n) ) <    \Si ( \UGW (P) ) . 
$$
\end{proposition}

The proposition is a consequence of the following upper bound on 
$
\Si (\UGW(P,Q)) 
$
\begin{lemma}\label{le:discSi}
Let $ S \subset \dZ_+ \backslash\{ 0 , 1\}$ be a finite set and  $P_1 \ne P_2$ be two probability measures on $S$. We have 
$$
 \Si ( \UGW ( P_1, P_2) )  \leq H((p_1,p_2)) +\sum_{i=1}^2p_i H(P_i) + \frac{d}{2} \log\PAR{ \frac{d}{2}}  - \frac{d}{2}   -\sum_{i=1}^2 p_i \dE_{P_i} \log (D!),
$$
where $D$ is the random variable with law $P_1$, $P_2$ respectively, and $H((p_1,p_2))=-\sum_{i=1}^2p_i\log p_i$.
\end{lemma}
The idea will be to prove that if $U(G_n) \weak \UGW(P_1,P_2)$, then $G_n$ needs to be approximately bipartite. The constraint of being bipartite will be costly in terms of entropy. 

\begin{proof}[Proof of Proposition \ref{prop:discSi}]
Using Lemma \ref{le:discSi}, with $P_1 = P_n$ and $P_2 = P$, we may upper bound  $\Si ( \UGW ( P, P_n) ) $ by 
\begin{align*}
H((p_1(n),p_2(n))) + &p_1 (n) H(P_n) + p_2(n)H(P)  + \frac{d}{2} \log\PAR{ \frac{d}{2}}  -  \frac{d}{2} \\ & \qquad  - p_1 (n)\dE_{P_n} \log (D!)  -  p_2(n) \dE_{P} \log (D!)  . 
\end{align*}
Since $P_n$ and $P$ have support in the finite set $S$, $p_1(n) \to 1/2$, $p_2(n) \to 1/2$, $H(P_n) \to H(P)$ and $\dE_{P_n} \log (D!) \to   \dE_{P} \log (D!)   $. So finally 
\begin{align*}
\LIMSUP_{n\to\infty} \, \Si ( \UGW ( P, P_n) )  & \leq  \log (2) + H(P) +  \frac{d}{2} \log\PAR{ \frac{d}{2}}  - \frac{d}{2}  -  \dE_{P} \log (D!) \\
& =  \Si( \UGW_1(P) ) -\PAR{  \frac d 2 - 1} \log (2).   
\end{align*}
Since $P(0) = P(1) = 0 $ and $P(2) < 1$, we have $d > 2$. \end{proof}

\begin{proof}[Proof of Lemma \ref{le:discSi}]
Let us start by a remark. We denote by $d_i$ and $\hat d_i$ the mean of $P_i$ and $\wP_i$. Since $\{0,1\} \notin S$, the support of $\wP_i$ is included in $\{1, \cdots, \theta\}$ for some $\theta$. It follows that $\hat d_i \geq 1$. Also, $\hat d_i = 1$ implies that $\wP_i = \delta_1$, hence $P_i = \delta_2$. Since $P_1 \ne P_2$, we have that either $\wP_1$ or $\wP_2$ is different from $\delta_1$. In particular,
$$
\alpha = \sqrt{\hat d_1 \hat d_2} > 1.
$$

Let $(T,o)$ be a rooted tree with distribution $\rho = \UGW(P_1,P_2)$ obtained from a multi-type 
rooted tree $(\check T , o)$ as above whose law is denoted by $\check \rho$. We will assign to all vertices  of $T$ a type $\{a,b\}$: type $a$ (resp. $b$) is supposed to be a good approximation for type $1$ (resp. $2$) in $\check T$. 

Let $\cA_1 \cup \cA_2$ be a partition of $\cP( \dZ_+)$ such that $\wP_i$ is in the interior of $\cA_i$ (it is possible since $\wP_1 \ne \wP_2$). Now, for $v \in V(T)$ and integer $h \geq 1$, $\partial B(v,h)$ is the set of vertices at distance $h$ from $v$ in $T$. The assumption $\{0,1\} \notin S$ implies that $\partial B(V,h)$ is not empty. Hence, we define 
$$
\mu^h_v = \frac{1}{|\partial B (v , 2h)|} \sum_{u \in \partial B (v , 2h)} \delta_{ \deg_T (u) -1 }.  
$$
Moreover $\alpha > 1$ implies that $\check \rho$-a.s. 
\begin{equation}\label{eq:logB}
\lim_{h \to \infty} \frac 1 h \log | \partial B(o,h)| = \log \alpha   > 0.
\end{equation}
Indeed, we consider a tree $T'$ whose vertex set are the vertices at even distance (in $T$) from the root. $T'$ is obtained by connecting vertices at distance $2h$  from the root to their grandchildren (the offspring of its own offspring), at distance $2(h+1)$.  Then, by construction, all vertices have the same type in $T'$. Moreover, conditioned on the root being of type $i$, $T'$ is a Galton-Watson tree where the root has offspring distribution $Q_i$, the distribution of $\sum_{k=1}^{N} N_k$, where $N$ has law $ P_i$, independent of $(N_k)_k$ an i.i.d. sequence with law $\widehat P_{2}$ if $i =1$ and $\widehat P_1$ if $i=2$, and any other vertex in $T'$ has offspring distribution $Q'_i$, the distribution of $\sum_{k=1}^{\widehat N} N_k$, where $\widehat N $ has law $\widehat P_i$, independent of $(N_k)_k$ as above. By construction, $Q'_i$ has mean $\alpha^ 2 = \hat d_1 \hat d_2$ and $T'$ has extinction probability $0$.   Then \eqref{eq:logB} is a consequence of the Seneta-Heyde Theorem \cite{MR0234530,MR0254929}.  

Also, conditioned on the root being of type $i$, all vertices $u \in B(o,2h)$ are of type $i$. It follows that, conditioned on  $|\partial B (o , 2h)|$ the vector $(\deg_T (u) -1)_{u \in \partial B (o , 2h)}$ is i.i.d. with common law $\wP_i$. Hence, the strong law of large numbers implies that, $\check  \rho$-a.s.  
$$
\mu^ h _o \weak \wP_{c(o)}.
$$ 
where $c(o)$ is the type of the root.

In the sequel, we fix $\delta >0$ and take $h$ large enough such that 
$$
\min_{i = 1, 2 } \dP_{\check \rho} ( \mu^ h _o \in \cA_{i} \,|\, c(o) = i ) \geq 1 - \delta. 
$$

Now, to a locally finite graph $G = (V,E)$, we attach to each vertex $v \in V$ the type $\o(v) = a$ (resp. $\o(v) = b$) if $\partial B (v , 2h)$ is not empty, $\deg(v) \in S$ and   $\mu^ h_v \in \cA_1$ (resp. $\mu^h_v \in \cA_2$). Otherwise, we set set $\o(v) = \bullet$.

Let $\bar a  = b$, $\bar b = a$, $\theta = \max ( s \in S)$ and $\Theta = \{0,\cdots,\theta\}$. We also attach on the vertices of $G$ a new type in the set $\cR = \{ \bullet , (a,k) , (b,k) : k \in \Theta \}$ defined, for $c \in \{a,b\}$, by $\tau(u) = (c,k)$ if 
\begin{enumerate}[(i)]
\item
$\omega(u) = c$ ;
\item 
$\sum_{v \stackrel{G}{\sim} u} \IND( \omega(v) = \bar c )  = k$.
\end{enumerate}
Otherwise,  $ \omega (u) = \bullet$ and we also set $\tau(u) = \bullet$.  In words: a vertex has $\tau$-type $(c,k)$  if its $\o$-type is $c$ and it has exactly  $k$ of its neighbors having $\o$-type $ \bar c $. We may call this scalar $k$ the $ab$-degree of the vertex.

By construction, $\dP_{\check \rho} ( \o(o) = c(o) ) \geq 1 -  \delta$. Also, using the union bound and unimodularity,
\begin{align*}
\dP_{\check \rho} ( \exists \, v \stackrel{T}{\sim} o : \omega (v) \ne c(v) ) & \leq \dE_{\check \rho} \sum_{v} \IND ( v \stackrel{T}{\sim} o ) \IND ( \o(v) \ne c(v) ) \\
 & =    \dE_{\check \rho} \sum_{v} \IND (v \stackrel{T}{\sim} o  ) \IND(    \o(o) \ne c(o) ) \\
 & \leq  \theta\, \dP_{\check \rho} \PAR{\o(o) \ne c(o)}  \leq \theta  \,\delta.
\end{align*}
We thus have proved that 
$$
\dP_{\check \rho} (\tau(o) =  (c(o),\deg (o)) ) \geq 1 - (\theta +1) \delta.   
$$
It follows that, for any $k \in S$, $c\in \{a,b\}$, 
\begin{equation}\label{eq:reconstypes}
\ABS{ \dP_{\rho} ( \omega (o) = c ) -p_i } \leq \delta \quad \hbox{ and } \quad  \ABS{ \dP_{\rho} (  \tau(o) = (c,k) ) - p_i P_i(k) } \leq (\theta+1) \delta, 
\end{equation}
where $i= 1$ if $c = a$ and $i=2$ if $c = b$. Equation \eqref{eq:reconstypes} shows that we can nearly reconstruct the types and the bipartite structure from $2h$-neighborhoods.

Also, by construction, the maps $(G,o) \to \o(o)$ and $(G,o) \to \tau(o)$ are continuous for the local topology.  Hence, there exists $\eta(\veps)>0$ with $\eta (\veps) \to 0$, $\veps\to 0$, such that $\mu \in B(\rho,\veps)$ implies that 
\begin{align*}
 \max_{c \in \{a,b,\bullet\}} \ABS{  \dP_{\mu} ( \o ( o) = c ) -    \dP_{\rho} ( \o ( o) = c ) }  \leq \eta(\veps), \quad \hbox{and}  \quad \max_{r  \in \cR } \ABS{  \dP_{\mu}\PAR {  \tau(o)  = r } -    \dP_{\rho} \PAR{   \tau(o) =  r   }  }\leq \eta(\veps).
\end{align*}
For all $\veps \leq \veps( \delta)$ small enough, $\eta (\veps) \leq \delta$.

All ingredients are now in order. Consider a sequence $m = m
(n)$ such that $m(n)/n \to d /2$ where $d = 2 p_1 d_1 = 2 p_2 d_2 = 2 d_1 d_2 / (d_1 + d_2)$. Let $G_n \in \cG_{n,m} ( \rho , \veps)$ with $\veps \leq \veps(\delta)$. For $c \in \{a,b,\bullet\}$ and $r \in \cR$, we set 
\begin{align*}
n_c = \sum_{ v = 1 }^ n \IND ( \omega(v) = c )  \quad \hbox{and} \quad N_r = \sum_{ v = 1 }^ n  \IND (  \tau(v) = r ) . 
\end{align*}
From what precedes and \eqref{eq:reconstypes}, for $c \in \{a,b\}$ and $k \in \Theta$, 
\begin{equation}\label{eq:constrintpart}
\ABS{ n_c - n p_i } \leq 2 \delta n   \quad \hbox{and} \quad  \ABS{ N_{(c,k)} - n p_i  P_i(k) } \leq 2(\theta +1) \delta n,
\end{equation}
where $i = 1$ if $c =a$ and $i = 2$ if $c = b$.  We notice also that $(n_a,n_b,n_{\bullet})$ is an integer partition of $n$ of length $3$ and $(n_{r})_{ r \in \cR}$ is an integer partition of length $|\cR| = 2(\theta +1) +1$.

We now compute an upper bound for $| \cG_{n,m} ( \rho , \veps)|$. Fix $\bn = ((n_c)_{c\in \{a,b\}},(N_r)_{r \in \cR})$. We denote by $A(\bn)$ the set of vertex-labeled graphs $G = ([n],E,\omega',\tau')$ such that for any $c \in \{a,b\}$, $r\in \cR$ and $v\in [n]$, 
\begin{enumerate}[(i)]
\item $\omega'(v) \in \{a,b,\bullet\}$ and $\tau'(v) \in \cR$ ; 
\item $\tau'(v) = (c,k)$  iif $\o'(v) = c$ and $\sum_{ u \stackrel{G}{\sim} v} \IND( \o' (u) = \bar c ) = k$ ;
\item  $n_c = \sum_v \IND ( \o' (v) = c )$ and $N_r = \sum_{ v = 1 }^ n \IND ( \tau'(v) = r )$. 
\end{enumerate}

From what precedes, 
$$
| \cG_{n,m} ( \rho , \veps)| \leq n^ 2 n^{|\cR|-1} \max_{\bn} |A(\bn)|, 
$$
where the maximum is over all pairs of integer partitions $((n_c)_{c\in \{a,b,\bullet\}},(N_r)_{r \in \cR})$ satisfying \eqref{eq:constrintpart}.

We set 
$$
 m_{\circ} =  \sum_{k \in \Theta} k N_{(a,k)} = \sum_{k \in \Theta} k  N_{(b,k)} \; \hbox{ and }\; m_{\bullet} = m - m_{\circ}.  
$$
In words, $m_{\circ}$ is the number of $ab$-edges (i.e.\ adjacent to a vertex of $\o'$-type $a$ and a vertex of $\o'$-type $b$), $m_{\bullet}$ counts all the other edges. Summing \eqref{eq:constrintpart} over $c \in \{a,b\}$, $k \in \Theta$, yields 
$$
n_{\bullet} = n -n_a - n_b \leq 4 n \delta,
$$
and 
$$
\Big| m_{\circ} - \frac{n d}{2} \Big|= \Big| \sum_{k \in \Theta} \PAR{ k N_{(a,k)} - n p_1 k  P_1 (k) }\Big|   \leq  2\theta( \theta +1)^2 \delta n = O ( \delta n ). 
$$
Since $m = m_{\bullet} + m_{\circ} = nd /2 + o(n)$. It follows that 
$$
m_{\bullet} = O ( \delta n) + o(n),
$$
where $O( \cdot) $ depends only on $\theta$. 

We find 
$$
|A(  \bn) | \leq { n \choose (n_a,n_b,n_\bullet) } { n_a \choose (N_{(a,k)})_{k \in \Theta} } { n_b \choose (N_{(b,k)})_{k \in \Theta} } \frac{ m_{\circ} !}{\prod_{k\in \Theta} ( k !)^{N_{(a,k)} + N_{(b,k)}}  } { \frac{n (n -1) }{2} \choose m_{\bullet}},
$$
where: the first term counts the number of ways to partition $[n]$ into three blocks of sizes $n_a,n_b$ and $n_\bullet$; the second and third terms subdivide each of the blocks in terms of the $ab$-degrees of the vertices; the fourth term upper bounds the number of ways to realize the $ab$-degree sequence (reasoning as in Lemma \ref{le:bipmatch}); the last term bounds the number of ways to put the remaining $m_{\bullet}$ edges.  

We set $p = (p_a,p_b,p_{\bullet})$ with $p_c = n_c / n$ and for $r = (c,k) \in \cR$,  $P_c (k) = N_{(c,k)} / n_c$. Using Stirling's approximation, we obtain 
\begin{align*}
\log |A(  \bn) |   &\leq    n H(p) + n p_a H(P_a) + n p_b H(P_b)   + m_{\circ}  \log n + m_{\circ} \log\PAR{ \frac{m_{\circ}}{n}}  - m_{\circ} \\
&   \quad - np_a \dE_{P_a} \log (D!)  -  np_b \dE_{P_b} \log (D!) + m_{\bullet} \log n - m_{\bullet} \log \PAR{ \frac{2 m_{\bullet} }{n(n-1)}   } + m_{\bullet}  + o(n),
\end{align*}
where $o(\cdot)$ depends only on $\theta$. Using our estimates in terms of $\delta$, we get 
\begin{align*}
\log |A(  \bn) | - m \log n  &\leq    n H((p_1,p_2)) + n p_1 H(P_1) + n p_2 H(P_2)  + \frac{nd}{2} \log\PAR{ \frac{d}{2}}  - \frac{nd}{2}  \\
&   \quad - np_1 \dE_{P_1} \log (D!)  - n p_2 \dE_{P_2} \log (D!)  + O ( n \delta \log \delta^{-1} ) + o(n).
\end{align*}
Letting $n \to \infty$ and then $\delta \to 0$, the lemma follows. 
\end{proof}

\section{Large deviation principles}

\label{LD}

\subsection{Proof of Theorem \ref{th:LDPCM}} \label{subsec:LDPCM}
Fix a sequence $\bd=\bd^{(n)}$ as in (C1)-(C3), set $P_n := \frac 1 n \sum_{i=1} ^ n \delta_{d(i)}$.
The measure $P_n\in \cP(\dZ_+)$
may be viewed as a measure on rooted graphs with depth $1$, i.e.\ $\cG_1^*$, by 
assigning probability zero to any $g\in\cG_1^*\setminus\cT_1^*$, and by assigning the weight $P_n(k)$ to the unlabeled star with $k$ neighbors (rooted at the center of the star). Define
$$
m = \frac12\sum_{i=1}^ n d_i (n),
$$
so that $m/n\to d/2$ as $n\to\infty$, and define the set 
$$ 
\cG_{P_n} = \{ G \in \cG_{n,m} : U(G_n)_1 = P_n \}.
$$
Each element of $\cG( \bd_n)$ is isomorphic to exactly $n(\bd)$ graphs in $\cG_{P_n}$, i.e.\ $n(\bd)|\cG( \bd)|=|\cG_{P_n}|$, where 
$n(\bd)$ denotes the number of distinct vectors $(d(\pi_1),\dots,d(\pi_n))$ as 
$\pi:[n]\mapsto[n]$ ranges over permutations of the vertex labels. Since $U(G)$ is invariant under isomorphisms, Theorem \ref{th:LDPCM} is equivalent to the same statement where $G_n$ is a random graph uniformly distributed 
in  $\cG_{P_n}$ rather than in $\cG( \bd)$. Thus, for the rest of this proof $G_n$ will denote a uniform graph in $\cG_{P_n}$.

Since $U(G_n)$ is unimodular, we may restrict to the closed subspace $\cP_u(\cG^*)$. 
Let $\cK \subset \cP_u(\cG^*)$ denote the  compact set of unimodular probability measures supported by graphs with degree bounded by $\theta$. 
Unimodularity implies that 
$\r\in\cK$ is equivalent to  $\r$ being supported by graphs such that the degree at the root is bounded by $\theta$. 
By construction, $U(G_n)\in\cK$ and $P\in\cK$.  Therefore, if $\r\in\cP_u(\cG^*)$ is such that $\r_1=P$, then $\r\in\cK$. 
From general principles, see e.g.\ \cite[Ch.\ 4]{dembo}, the theorem follows if we prove that: (i) for any $\rho \in \cK$ with $\rho_1 = P$, $\d>0$,
\begin{equation}
\label{lowerb1}
 \liminf_{n \to \infty} \frac 1 n \log \dP ( U(G_n) \in B(\rho,\delta) ) \geq  \Sigma(\rho)- \Sigma(\UGW_1(P)) ;
\end{equation}
and (ii)  
for any $\rho \in \cK$ 
\begin{equation}
\label{upperb1}
\lim_{\veps \downarrow 0}  \limsup_{n \to \infty} \frac 1 n \log \dP ( U(G_n) \in B(\rho,\veps) )  \leq \left\{ \begin{array}{ll}
\Sigma(\rho) - \Sigma(\UGW_1(P)) & \hbox{ if $\rho_1 = P$}, \\
- \infty & \hbox{ otherwise.}
\end{array}\right.
\end{equation}
To prove the lower bound \eqref{lowerb1}, write
\begin{equation}
\label{erb1}
\dP ( U(G_n) \in B(\rho,\delta) )  = \frac { |\{ G \in \cG_{n,m} : U(G_n)_1 = P_n ,\; U(G_n) \in B ( \rho, \delta)\}| } { |\{ G \in \cG_{n,m} : U(G_n)_1 = P_n  \}| }.
\end{equation}
As a consequence of \eqref{upperenh2}
$$
 \frac 1 n \PAR{ \log  |\{ G \in \cG_{n,m} : U(G_n)_1 = P_n\}  | - m \log n}\leq J_1(P) + o(1).
$$
On the other hand, the lower bound in Proposition \ref{pro:enh} proves that for fixed $\d>0$, one has 
$$
 \frac 1 n \PAR{ \log  |\{ G \in \cG_{n,m} : U(G_n)_1 = P_n ,\; U(G_n) \in B ( \rho, \delta)\}  | - m \log n}\geq J_h(\r_h) + o(1)
$$
for all $h$ large enough. From Theorem \ref{th:entropy2} one has $ \Sigma(\rho)=\lim_{h\to\infty} J_h(\r_h)$, and $\Sigma(\UGW_1(P))=J_1(P)$, and \eqref{lowerb1} follows. 

We turn to the proof of the upper bound \eqref{upperb1}. We start with the case $\rho_1 \ne P$. For $\delta > 0$, consider the closure, say $F(\delta)$,  of the probability measures $\rho\in\cK$ such that $d_{TV} (\rho_1 , P) \leq \delta$. For all $n$ large enough, $U(G_n) \in F(\delta)$, since $U(G_n)_1=P_n\weak P$. If $\rho_1 \ne P$, then 
$\rho \notin F(\delta)$ for some $\delta >0$, and 
$ \dP ( U(G_n) \in B ( \rho , \veps ) ) = 0$, 
for all $\veps $ small enough and $n$ large enough. It follows that  \eqref{upperb1} is $-\infty$ in this case.
Suppose now that $\rho_1 = P$. For the upper bound  one may drop the constraint $U(G_n)_1=P_n$ in the numerator of \eqref{erb1}. Then, using the lower bound in Proposition \ref{pro:enh} for the denominator and Theorem \ref{th:entropy2} for the numerator, one has the desired estimate. 

\begin{remark}\label{re:LDPCM}
The result of Theorem \ref{th:LDPCM} can be extended with no difficulty to the case  where $G_n$ is uniformly distributed in the set of all graphs $G$ with vertex set $[n]$ satisfying $U(G)_h=P_n$, where $P_n$ is supported on some fixed set $\D=\{t_1,\dots,t_r\}\subset \cT^*_h$ for all $n$, and such that 
$P_n\weak P$ for some admissible $P$. Theorem \ref{th:LDPCM} is the special case $h=1$.
With the same proof, for any fixed $h\in\dN$, one obtains that $U(G_n)$ satisfies the large deviation principle with speed $n$ and good rate function $I(\r)=J_h(P)-\Sigma(\r)$ if $\r_h=P$, and $I(\r)=+\infty$ otherwise.   
\end{remark}

\subsection{Proof of Theorem \ref{th:LDPER1}}

We start with a proof of exponential tightness.
Let $c \geq 1$ and let $G_n$ be a random graph sampled uniformly on $\cG_{n,m}$, where $m=m(n)$ is an arbitrary sequence satisfying 
$$
\frac { m(n)}{n} \leq \frac c 2.
$$
The random probability measure $\rho_n:= U ( G_n) $ is an element of $\cP_{u} ( \cG^*) $.

\begin{lemma}\label{le:exptight}
The sequence of random variables $\rho_n$ is exponentially tight in $\cP_{u}  ( \cG^*)$, i.e.\
for any $z\geq1$, there exists a compact set $\Pi_z\subset \cP_{u} ( \cG^*)$ such that 
$$
\limsup_{n\to\infty}\frac1n \log \dP(U(G_n)\notin\Pi_z)\leq -z.
$$ 
\end{lemma}

\begin{proof}
For $y \geq 1 $ and $x \in (0,1)$, we define $$
\delta_y ( x) =  - \frac{ 2  y  } { \log ( c x ) },
$$
and consider the event,
$$
\cE_y (n) = \BRA{ \forall S \subset [n]: | \deg_{G_n} (S)|  \leq n \,\delta_y \PAR{ \frac { |S| } { n} }  },
$$
where $\deg_{G_n} (S)$ was defined in \eqref{eq:defdegS}. We are going to prove that there exists a constant $L> 0$ such that for any real $y \geq 1$, for any integer $n \geq 1$, 
\begin{equation}\label{eq:exptightEy}
\frac 1 n \log \dP \PAR{ \cE_y (n) ^c } \leq - y + L. 
\end{equation}
In view of Lemma \ref{le:tightLWC}, \eqref{eq:exptightEy} implies the lemma. 

To prove  \eqref{eq:exptightEy}, we may restrict ourself to subsets $S \subset [n]$ of cardinality at most $|S| \leq n \veps_0$, with $\veps_0 = \delta_y ^{-1} ( 1) = e^{ - 2 y} / c \leq e^{-2y}$. From the union bound, 
\begin{equation*}
 \dP \PAR{ \cE_y (n) ^c }  \leq   n \max_{ 0 <\veps \leq \veps_0 } \dP\PAR{  \exists S \subset [n]: |S| = \veps n , \;  \deg_{G_n} (S )   >  n \delta_y \PAR{ \veps } } .
\end{equation*}
By choosing $y$ large enough we may assume that $\veps_0>0$ is small enough. Choose 
$\veps\in(0,\veps_0]$ and $\d:=\delta_y \PAR{ \veps }$. 
We note that, 
as in the proof of \eqref{deg1o}, $\deg_{G_n} (S )$  is stochastically dominated by 
$2N$, where $N$ has distribution $\BIN ( \veps n^2 , 2d/n)$.
It follows that 
$$
\dP \PAR{ \cE_y (n) ^c }  \leq n \max_{ 0 <\veps \leq \veps_0 }\binom{n}{\veps n} \dP ( N \geq \d n/2 ).
$$
For $x>0$,
\begin{align*}
\dP \PAR{ N \geq \d n/2 } \leq e^{-\d n x   }\dE[e^{2xN}] 
= e^{-\d n x } \PAR{ 1 +  (2d/n)  ( e^{2x} - 1)  }^{\veps n^2} 
\leq e^{- \d n x   +   2d \veps n e^{2x} }.
\end{align*}
Taking $x=-\frac12\log(c\veps)$ one finds
$$
  \frac 1 n \log \dP ( N \geq n \delta/2  ) \leq \frac\delta2 \log\PAR{ c \veps } +\frac{2d}c = - y +\frac{2d}c .
$$
On the other hand, from Stirling's formula, there exists a constant $ C$ such that 
$$
{ n \choose n \veps } \leq C \sqrt n e^{ n H(\veps) } ,
$$
where $H(\veps)=-\veps\log\veps-(1-\veps)\log  (1-\veps)$. Since $\veps\leq \veps_0=e^{-2y}$, these bounds imply the desired conclusion \eqref{eq:exptightEy}. 
\end{proof}

We turn to the proof of Theorem \ref{th:LDPER1}. Fix $d>0$ and a sequence $m=m(n)$ such that $m/n\to d/2$, as $n\to\infty$. Thanks to 
Lemma \ref{le:exptight}, 
from general principles, see e.g.\ \cite[Ch.\ 4]{dembo}, it is sufficient to establish: 
 (i) for any $\rho \in \cP_u(\cG^*)$ and  
 $\d>0$,
\begin{equation}
\label{lowerber1}
 \liminf_{n \to \infty} \frac 1 n \log \dP ( U(G_n) \in B(\rho,\delta) ) \geq  \Sigma(\rho)- s(d) ;
\end{equation}
and (ii)  
for any $\rho \in \cP_u(\cG^*)$ 
\begin{equation}
\label{upperber1}
\lim_{\veps \downarrow 0}  \limsup_{n \to \infty} \frac 1 n \log \dP ( U(G_n) \in B(\rho,\veps) )  \leq 
\Sigma(\rho) - s(d).
\end{equation}
However, both the lower bound \eqref{lowerber1} and the upper bound \eqref{upperber1} follow immediately from the definition of $\Si(\r)$, Theorem \ref{th:entropy1} and \eqref{eq:Gnmlog}.
This ends the proof.

\subsection{Proof of Theorem \ref{th:LDPER2}}
Theorem \ref{th:LDPER2} is a simple consequence of Theorem \ref{th:LDPER1}. We argue as in \cite{MR2759726}. Let $G_n$ denote the random graph with distribution $\cG(n,\l/n)$, and let $M(n)$ be the total number of edges in $G_n$. 
Then $M(n)$ is the binomial random variable $\BIN(n(n-1)/2,\l/n)$. Conditioned on a given value $M(n)=m$, $G_n$ has uniform distribution over $\cG_{n,m}$.  It follows that $\cG(n,\l/n)$ is a mixture of the uniform distribution on $\cG_{n,m}$, where $m$ is sampled according to $\BIN(n(n-1)/2,\l/n)$. We use the following simple lemma, whose proof is omitted.
\begin{lemma}\label{le:ldpmn}
The sequence $2M(n)/n$ satisfies the LDP in $[0,\infty)$ with speed $n$ and good rate function
$$j(x) =  \frac12\left(\lambda  - x + x \log (x/\l)\right).$$
\end{lemma}
We need to prove that $\rho_n  = U(G_n)$ satisfies a LDP on $\cP_u (\cG^*)$ with speed $n$ and good rate function 
$$
I(\rho) = j(d) - \Si (\rho) + s(d) = \inf \BRA{ j(r) - \Si_r (\rho) + s(r) : r \geq 0} ,  
$$
where $\rho \in \cP_u (\cG^*)$, $d = \dE_\rho \deg_G  (o)$ and $\Si_r (\rho)$ is the entropy of $\rho$ associated to the mean degree $r$ (which is equal to $-\infty$ if $r \ne d$ by Theorem \ref{th:entropy1}). 
A simple adaptation of the proof of Lemma \ref{le:exptight} shows that the random variable $\r_n=U(G_n)$ is exponentially tight. The conclusion follows from a general result on large deviations for mixtures; see Biggins \cite[Theorem 5(b)]{MR2081460}.
%

\subsection{Proof of Corollary \ref{cor:LDPER1} and Corollary \ref{cor:LDPER2} }
The proof 
is an application of the contraction principle, cf. \cite{dembo}.
Concerning Corollary \ref{cor:LDPER1}, by Theorem \ref{th:LDPER1} one has that $u(G_n)$ satisfies the LDP in $\cP(\dZ_+)$ with speed $n$ and good rate function $$K(P)=\inf\{s(d) - \Sigma(\r)\,,\; \r\in\cP(\cG^*):\,\r_1=P\}.$$ From Theorem \ref{th:entropy2} and Corollary \ref{cor:entropy} this expression equals $s(d)-J_1(P)=H(P\,|\,\POI(d))$. 

As for Corollary \ref{cor:LDPER2}, by Theorem \ref{th:LDPER2} $u(G_n)$ satisfies the LDP in $\cP(\dZ_+)$ with speed $n$ and good rate function  $$K(P)=\inf\{\phi(\l,d) - \Sigma(\r)\,,\; \r\in\cP(\cG^*):\,\r_1=P\},$$ where $\phi(\l,d)=\frac \l2 - \frac d 2 \log \l $. Since all  $\r\in\cP(\cG^*)$ with $\r_1=P$ have the same expected degree at the root, this equals $\phi(\l,d)-J_1(P)=\phi(\l,d)-s(d)+H(P\,|\,\POI(d))$.

\appendix 
\section{Local  convergence for generalized configuration model}
In this section we prove Theorem \ref{th:convlocCM}. 
\subsection{The exploration process}
The first step  is to prove convergence of the average measure $\dE U (G_n)$, where $G_n$ has distribution $\CM( {\bD}^{(n)} )$.

\begin{proposition}\label{prop:convlocCM}
Let $G_n \in \wcG(\bD^{(n)} )$ with distribution $\CM ( {\bD}_n)$ such that  assumptions (H1)-(H2) hold (see Section \ref{H1H2}). Then, $\dE U(G_n) \weak \UGW(P)$. 
\end{proposition}

The proof of Proposition \ref{prop:convlocCM} is based on an exploration process of the neighborhood of a vertex. 
We shall use the notation of Section \ref{CM}. For ease of notation, we will often omit the dependence on $n$ from our notation. Let $\bD = (D(1), \cdots, D(n)) \in \cD_n$, $\si \in \Si$ and $G =\Gamma (\sigma)$ the associated multigraph. To be precise,  we specify the set $W = \cup_{c \in \cC} W_c$ to be $ W_c= \{ (c , i , j):   i \in [n] , 1 \leq j \leq   D_c(i)  \}$ and $W (i) =  \{ (c , i , j ):  c \in \cC ,  1 \leq j \leq   D_c (i)  \}$ the set of half-edges of all colors starting from $i$. With a slight abuse of notation, we will sometimes write for $e = (c, i,j) \in W$, $\si (e)$ in place of $\si_c ( i,j)$.

Let $\dN^f = \cup_{k \geq 0} \dN^ k$ where $\dN^0 = \so$. We consider the total order on $\dN^f$: $\ibf < \jbf $ with $\ibf = (i_1, \cdots, i_k), \jbf = (j_1, \cdots, j_{\ell})$ if either  $k < \ell$ or $k = \ell$ and $\ibf < \jbf$ for the lexicographical order. We will define a bijective map $\phi$ from a finite set $S \subset \dN^f$  to the vertex set of $G(v)$. The value of $\phi$ is defined iteratively and if $\ibf < \jbf$ are in $S$ then the value of $\phi(\ibf)$ will be determined before the value of $\phi(\jbf)$. Moreover $\phi( S \cap \dN^k)$ will be the set of vertices at distance $k$ from $v$.

The exploration is on the set of half-edges $W$ and it is defined recursively. At integer step $t$, we partition $W$ in $3$ sets: an half-edge may belong to the {\em active} set $A(t)$,  to the {\em unexplored} set $U(t)$  or to the {\em connected} set $C(t) = W \backslash ( A(t) \cup U(t))$. At stage $t$, a vertex with an half-edge in $C(t) \cup A(t)$ will have a pre-image via $\phi$ in $\dN^f$. We start with a given $v \in [n]$, and fix the initial conditions $A(0) = W (v)$, $C(0) = \emptyset$, $U(0) = W \backslash W(v) $, and $\phi(\so) = v$. 

For integer $t \geq 0$, if $A(t) \neq \emptyset$, let $e_{t+1} = (c_t,\phi(\ibf_{t}), j_t)$ be an half-edge in $A(t)$ such that $\ibf_{t}$ is minimal for the total order on $\dN^f$. Let $I(t+1) =  ( W (v_{t+1}) \backslash  \{ \sigma (e_{t+1} ) \}) \cap U(t) $ where $v_{t+1}$ is the vertex such that $\sigma(e_{t+1}) \in W (v_{t+1})$. $I_{t+1}$ is the set of new half-edges and our partition of $W$ is updated as 
\begin{equation}
\left\{
\begin{array}{lcl}
A(t+1) & =& \!\! A(t)   \backslash  \{ e_{t+1}  , \sigma ( e_{t+1})\}     \bigcup I_{t+1}  \\
U(t+1) & = & \!\! U(t)  \backslash  \left( I(t+1)  \cup \{ \sigma ( e_{t+1})\} \right) \\
C(t+1)  & = &  \!\!C(t) \cup \{ e_{t+1},\sigma ( e_{t+1})\}.
\end{array} \right.\label{eq:explorationCM}
\end{equation}
 If $\sigma(e_{t+1}) \notin A(t)$, we also set $\phi( ( \ibf_{t} , j_t) )  = v_{t+1}$. Finally, if $A(t) = \emptyset$, then the exploration process stops.

We notice that the elements in $C(t) $ are the half-edges for which we know by step $t$ their matched half-edge. It implies that $\sigma(e_{t+1}) \in A(t) \cup U(t)$. Moreover, for any vertex $u$, we cannot have simultaneously $W (u) \cap U(t) \ne \emptyset$ and  $W (u) \cap A(t) \ne \emptyset$. With a slight abuse, we may thus write $u \in U(t)$ or $u \in A(t)$ if, respectively, $W(u) \cap U(t) \ne \emptyset$ or $W(u) \cap A(t) \ne \emptyset$. Now, if $v_{t+1} \in  U(t)$, then $I(t+1) =  W (v_{t+1}) \backslash \{ \sigma(e_{t+1}) \}$, otherwise $v_{t+1} \in  A(t)$ and $I(t+1)=  \emptyset$.  Note also that for integer $k$, the image by $\phi$ of the vertices of generation $k$ in $S$, $\phi ( S \cap \dN^k)$, are the set of vertices in $G$ at distance $k$ from $v$ (by recursion, this comes from the fact that $\ibf_t$ is minimal for the total order on $\dN^f$).  

We now define $X(0) = D(v)$ and for integer $t \geq 1$, $X_c (t+1) = | \{ (i, j ): (c,i,j) \in I_{t+1} \}|$. Hence $X(t) \in\cM_L$ gives the new colored half-edges attached to $v_{t}$. For ease of notation, we also set $$\veps_c(t+1) =  \IND \PAR{ v_{t+1}  \in A(t) , c_t = \bar c} \, , \quad \delta_c ( t+1) = \IND ( c_t = c),$$ and $$\tau =  \inf \{ t: A(t) = \emptyset \}.$$ 
Setting $A_c=A\cap W_c$, $U_c=U\cap W_c$ and $C_c=C\cap W_c$, we get
\begin{eqnarray}
& | A_c(t) | = D_c(v) + \sum_{ k = 1} ^t  (X_c(k) - \delta_c(k) - \veps_c(k))  \nonumber \\
& | U_c (t) |  = | W_c | - D_c(v) -\sum_{ k = 1} ^t  (X_c(k) + \delta_{\bar c} (k) - \veps_c(k)) \nonumber  \\ 
 & | C_c(t) |  = \sum_{k=1}^t (\delta_c (k) + \delta_{\bar c} (k)). \label{eq:defAUC}
\end{eqnarray}
Note that $| C_c(t) | = | C_{\bar c}(t) |$ and, if $c \in \cC_=$, $|C_c (t)|$ is even.

Now, as in the statement of Proposition \ref{prop:convlocCM}, consider a random multi-graph $G_n$ with distribution  $\CM(\bD^{(n)})$. For integer $t \geq 0$, we consider the filtration  
$$\cF_t = \sigma( (A(0),U(0), C(0)), \cdots, (A(t),U(t),C(t))).$$ The hitting time $\tau$ is a stopping time for this filtration.  Also, given $\cF_{t}$, if $\{ t < \tau\}$ and $c_t = c \in \cC_{\ne}$, then $\sigma(e_{t+1})$ is uniformly distributed on $ U_{\bar c} (t) \cup A_{\bar c} (t) $. It follows that for $u \in [n]$, 
\begin{align*}
\dP \PAR{ v_{t+1} = u  | \cF_t } & =  \frac{ | W_{\bar c}(u)   \cap ( U_{\bar c} (t) \cup  A_{\bar c} (t) ) | }{ |U_{\bar c} (t) | + | A_{\bar c} (t) |   } \\
& =  \frac{ \IND_{ u \in U(t)} D_{\bar c} (u) }{ |U_{\bar c} (t) | + | A_{\bar c} (t) | } +  \frac{ \IND_{ u \in A(t)}  | W_{\bar c} (u) \cap A(t)  |}{ |U_{\bar c} (t)| + | A_{\bar c} (t) |}.
\end{align*}
Similarly, given $\cF_t$, if $\{ t < \tau\}$ and $c_t = c \in \cC_{=}$, $\sigma(e_{t+1})$ is uniformly distributed on $ U_{c} (t) \cup A_{c} (t) \backslash \{ e_{t+1}\} $. We find in this case, 
\begin{align*}
\dP ( v_{t+1} = u  | \cF_t ) & =  \frac{ | W_{c}(u)   \cap ( U_{ c} (t) \cup  A_{ c} (t) \backslash \{ e_{t+1} \} ) | }{ |U_{ c} (t) | + | A_{ c} (t) | - 1    } \\
& =  \frac{ \IND_{ u \in U(t)} D_{c} (u) }{ |U_{c} (t) | + | A_{c} (t) | - 1 } +  \frac{ \IND_{ u \in A(t)} ( | W_{c} (u) \cap A_c(t)  | -  \IND_{e_{t+1} \in W_c(u)})}{ |U_{c} (t)| + | A_{c} (t) | - 1 }.
\end{align*}
In either case , for $c \in \cC$, if $\sigma(e_{t+1}) \in U(t)$, then $X(t+1) = D(v_{t+1}) - E^ {\bar c}$ otherwise, $\sigma(e_{t+1}) \in A(t)$ and $X(t+1) = 0$. We recall also that $|U_c (t)| + | A_c (t) |  = | W_c | - | C_c (t) |  = | W_{\bar c} | - | C_{\bar c} (t) | $. We get, for $ M \in \cM_L$, if $c_t=c$ then
\begin{equation}
\dP\PAR{ X(t+1) = M  | \cF_t} =  \left\{ \begin{array}{ll} 
\sum_{ u \in U(t)}  \frac{ \IND_{D(u) = M+E^ {\bar c}} (M_{\bar c} +1)}{| W_c | - |C_c(t)| -  \IND( c \in \cC_{=}) } & \; M \neq 0 , \\
\sum_{ u \in U(t)}  \frac{ \IND_{D(u) = E^{\bar c}} }{| W_c | - |C_c(t)| -  \IND( c \in \cC_{=})  } +  \frac{| A_{\bar c} (t)| - \IND ( c \in \cC_= )   }{ | W_c | - |C_c(t)| -  \IND( c \in \cC_{=})  }& \; M =  0.
\end{array} \right. \label{eq:XCM}
\end{equation}

Observe that, from \eqref{eq:defAUC} and assumption (H1), we find for any $c \in \cC$, 
\begin{equation}\label{eq:boundedA}
| A_c (t) | \leq \theta (t+1)\quad  \hbox{ and } \quad  |C_c (t) | \leq 2 t. 
\end{equation}

The next lemma computes the limiting marginals of the exploration process. 
\begin{lemma}
\label{le:coupexpCM}
Under the assumption of Proposition \ref{prop:convlocCM}, let $o$ be uniformly distributed on $[n]$, independently of $G_n$, and consider the exploration process on the rooted graph $(G_n(\so),\so)$. For any integer $t \geq 0$, as $n\to\infty$:
\begin{enumerate}[(i)]
\item \label{eq:1coupCM}
$X(0)$ converges weakly to $P$.
\item \label{eq:2coupCM}
Let $c\in \cC$ be such that $\dE D_c  > 0$. Given $\cF_t$, if $\{ t < \tau\}$ and $c_t = c$, then the conditional law of $X(t+1)$ given $\cF_t$ converges weakly to $\wP^ c$. 
\item \label{eq:3coupCM}
The probability that there exist $c\in \cC$ and an integer $1 \leq s \leq t\wedge \tau $ such that $\dE D_c =0$ and $c_s = c$ goes to $0$. 
\end{enumerate}
\end{lemma}
\begin{proof}
Since  $X(0)= D(\so)$, statement \eqref{eq:1coupCM} is simply a restatement of the assumption (H2). 

For statement \eqref{eq:2coupCM}, we first note that the set $\{ i \in [n]: i \notin U(t)\}$ has cardinality bounded by $1 + \theta L^2 t$. It follows by (\ref{eq:XCM}) that, if $\{ t < \tau \}$ and $c_t = c$ hold, for any $M \in \cM_L$, 
\begin{align*}
&\Big|  \dP( X (t+1) = M | \cF_t ) - \frac { M_{\bar c} +1}{| W_c | - |C_c(t)| -  \IND( c \in \cC_{=})} \sum_{i=1} ^n \IND_{D(i) = M + E^{\bar c}}  \Big| \\
&\qquad\qquad\qquad\qquad\qquad\qquad\qquad\qquad\qquad\qquad \quad\leq  \frac { ( M_{\bar c}+1 )  \left(  1 + \theta L^2 t +  | A_{ \bar c} (t)| \right) }{| W_c | - |C_c(t)| -  \IND( c \in \cC_{=})}. 
\end{align*}
Now, assumptions (H1)-(H2) imply that $| W_c | / n$ converges to $\dE D_c$, where $D$ has law $P$. Similarly, assumption (H2) implies that 
$
\frac 1 n \sum_{i=1} ^n \IND_{D(i) = M + E^{\bar c}} 
$
converges to $P(M + E^ {\bar c})$. Hence, from \eqref{eq:boundedA}, if $\dE D_c >0$, then $\dP( X (t+1) = M | \cF_t )$ converges to $\wP^ {c } (M)$. This proves statement \eqref{eq:2coupCM}.

We now turn to statement \eqref{eq:3coupCM}. We set $\cC_0 = \{ c\in \cC: D_c \equiv 0 \}$ and $A_{\cC_0} (t) = \cup_{c\in \cC_0} A_{c} (t)$. We recall that $\dE D_{\bar c} = \dE D_{c}$, hence $c\in \cC_0$ is equivalent to $\bar c \in \cC_0$. We should prove that for any integer $t \geq 0$, $\dP ( | A_{\cC_0} (t \wedge \tau ) | \geq 1 ) \to 0$. First, statement \eqref{eq:1coupCM} and the union bound implies that $\dP ( | A_{\cC_0} (0) | \geq 1 ) \leq \sum_{c \in \cC_0} \dP (X_c(0) \geq 1) \to 0$. By recursion, it is thus sufficient to prove that for any integer $t \geq 0$, $c \in \cC_0$, $c' \in \cC \backslash \cC_0$, if $\{ t < \tau \}$ and $c_t = c'$ hold, then
$$
\dP ( X_c (t+1) \geq 1 | \cF_t ) \to 0. 
$$
The latter follows from statement \eqref{eq:2coupCM} (recall that $\bar c \in \cC_0$). 
\end{proof}

We introduce a variable that counts the number of times that two elements in the active sets are matched by step $t$:
$$
E(t) = \sum_{k=1} ^t \sum_{c \in \cC} \veps_c (k).
$$
\begin{lemma} 
\label{le:cycleexpCM}
Under the assumption of Proposition \ref{prop:convlocCM}, let $o$ be uniformly distributed on $[n]$, independently of $G_n$, and consider the exploration process on the rooted graph $(G_n(\so),\so)$.  For every integer $t \geq 0$, we have
$$
\lim_{n\to\infty} \dP \left( E ( t \wedge \tau )  \neq 0  \right) = 0. 
$$
If $t \leq \tau$ and $E(t) = 0$, the subgraph of $G_n$ spanned by the vertices with all their half-edges in $C(t )$ is an directed colored tree. 
\end{lemma}
\begin{proof}
We start with the second statement. To every vertex  $u$ with an half-edge in $C (t) \cup A(t)$, there is an element $\ibf$ in $\dN^f$ such that $\phi (\ibf) = u$. We may thus order these vertices by  the order through $\phi^{-1}$ in $\dN^f$. Every such vertex is adjacent to its parent. By construction if $ E(t) = 0$ or equivalently if for all $1 \leq s \leq t$, all $c\in \cC$, $\veps_c(s) = 0$, then every vertex with an half-edge in $C(t) \cup A(t) $ has a unique adjacent vertex with a smaller index.  It follows that there cannot be a cycle in the subgraph spanned by these vertices. 

If $  E (t\wedge \tau)  \ne 0 $, there exists an integer $1 \leq s \leq t \wedge \tau$ such that $\sigma(e_s) \in A(s-1)$. Using (\ref{eq:XCM}), it follows from the union bound and the fact that $\{ s < \tau \} \in \cF_s$, 
\begin{align*}
\dP  (  \exists 1 \leq s \leq  t \wedge \tau:  \sigma(e_s) \in A(s-1)  ) & \leq  \dE \left[  \sum_{s \geq 0} \IND_{ s < t  \wedge \tau }  \dP ( v_{s+1} \in  A(s) |  \cF_s)  \right]\\
& \leq  \dE  \sum_{s = 0} ^ {t-1} \frac{| A_{\bar c_s} (s)|}{  | W_{c_s} | - |C_{c_s}(s)| -  \IND( c_s \in \cC_{=})}.
\end{align*}
From \eqref{eq:boundedA}, for each $t \geq 0$, $|A_c (t)|  \leq \theta (t+1)$. Also, by  assumptions (H1)-(H2), $| W _c| / n$ converges  to $\dE D_c$, where $D$ has law $P$. If $\dE D_c=0$, we may appeal to  Lemma \ref{le:coupexpCM}\eqref{eq:3coupCM}. 
\end{proof}

All ingredients of the proof of Proposition \ref{prop:convlocCM} are now gathered. 
\begin{proof}[Proof of Proposition \ref{prop:convlocCM}. ] Let $o$ be uniformly distributed on $[n]$, independently of $G_n$. We set $\overline \rho_n = \dE U ( G_n)$, and $\rho = \UGW (P)$. Define $B =\{g \in \wcG^*: g_t = \gamma \}$ where $\gamma$ is the equivalence class of a finite rooted directed colored tree of depth at most $t$. It is sufficient to prove that for any integer $t \geq 1$ and any such $\gamma$, $\overline \rho_n (B) $ converges  to $\rho (B)$.

For some $m = \sum_{k = 0}^{t-1} (\theta L^ 2)^ k$, $(G_n(\so),\so)_t$ has at most $m$ vertices.  However, by Lemma \ref{le:cycleexpCM}, with high probability, $E_{m \wedge \tau} =0$ and $(G_n (\so),\so)_t$ is a rooted directed colored tree. Applying now Lemma \ref{le:coupexpCM}, we deduce that  
$$
\lim_{n\to\infty}  \left| \overline \rho_n (B)   - \rho (B )  \right| = \lim_{n\to\infty}   \left| \dP ( (G_n (\so) ,\so)_t \simeq \gamma  )    - \rho ( B )  \right|   = 0. 
$$
The conclusion follows.
 \end{proof}

\subsection{Concentration Inequalities}

We are going to state a concentration inequality for the configuration model. We use the notation of Section \ref{CM}. We fix an integer $L \geq 1$ and consider a set of colors $\cC = \{(i,j): 1 \leq i , j \leq L \}$, $\bD = (D(1),\cdots, D(n)) \in \cD_n$ and $\Si = \Si( \bD)$ be the set of configurations. We shall say that  $m \in \Si $ and $m' \in \Si$ differ by at most one switch if there exists $c \in \cC_\leq $ such that for all $c ' \ne c$, $m_{c'}  = m'_{c'}$ and a set $J \subset W_c$, with $|J| \leq 2$ if $c \in \cC_{\ne}$ or $|J| \leq 4$ if $c \in \cC_=$, and for all $x \in W_c \backslash J$, $m (x) = m'(x)$. In other words, if $c \in \cC_{\ne}$, $m'_c \circ m_c^ {-1}$ is either the identity ($|J| = 0$) or a transposition ($|J| = 2$). Similarly, for $c \in \cC_{=}$, $m'_c \circ m_c^ {-1}$ is either the identity ($|J| = 0$) or the composition of two disjoint transpositions ($|J| = 4$). 

In the special case $L=1$, the next proposition appears in Wormald \cite[Theorem 2.19]{wormald}.

\begin{proposition}\label{th:azumaCM}
Let $\bD = (D(1),\cdots, D(n)) \in \cD_n$, $\Si = \Si( \bD)$ be the set of configurations, $S = \sum_{i=1}^ n D (i)$ and  $N =  \sum_{c \in \cC} S_c$. Let $F: \Si\to \dR$ be a function such that for some $\kappa > 0$ and any $m , m' \in \Si$ which differ by at most one switch, we have 
$$
| F ( m ) - F (m') | \leq \kappa.
$$
Then, if  $\si$ is uniformly sampled from $\Si$,  for any $t > 0$, 
$$
\dP \left( \left|     F ( \si )   -  \dE F (\si)     \right|  \geq t  \right) \leq  2 \exp\left ({-  t^2 \over \kappa^2 N} \right).$$
\end{proposition}

The proof will be given in Section \ref{propa4} below. 

\begin{corollary}\label{cor:azumaCM}
Let $\bD = (D(1),\cdots, D(n)) \in \cD_n$ such that (H1) holds (see Section \ref{H1H2}). Let $k \geq 0$, $\gamma$ be a rooted directed colored multigraph and $A = \{ g :\; g_k = \gamma_k \}$. There exists a constant $\d = \d(\theta,k,L) > 0$, such that, if  $G_n \stackrel{d}{\sim} \CM(\bD)$, $\rho_n = U(G_n)$ and $t >0$,
$$
\dP \left( \left|     \rho_n ( A)   -  \dE  \rho_n ( A)     \right|  \geq t  \right) \leq  2 \exp\left (-   \d n t^2  \right).$$
\end{corollary}

\begin{proof}

By assumption we have for any $c \in \cC$, $i\in [n]$, $ D_c(i) \leq \theta$. We may thus assume without loss of generality that $\g$ has degrees bounded by $ \theta$. We set $$f(\si):=n \rho_n (A) = \sum_{i=1} ^n \IND( (G_n(i), i )_k \simeq \gamma_k ).$$ 

The number of vertices in $G_n$ which are at distance at most $k$ from both endpoints of any given edge is bounded by $\kappa = 2\sum_{s=0}^{k-1} (\theta L^ 2) ^{s}$. If two configurations $m,m'$ in $\Si$ differ by at most one switch then $|f(m) - f(m')| \leq 4 \kappa$. Indeed, a switch changes the status at most $4$ edges and the addition or the removal of an edge can modify for at most $\kappa$ vertices the value of $ \IND( (G_n(i), i )_k \simeq \gamma_k ) $. 
It remains to apply Proposition \ref{th:azumaCM}, with $F(\si)=f(\si)/n$ and $N=O(n)$. 
\end{proof}

\subsubsection{Proof of Theorem \ref{th:convlocCM}}

Let us start with the case $G_n \stackrel{d}{\sim} \CM(\bD^{(n)})$.  We set $\rho_n = U(G_n)$ and $\rho = \UGW(P)$. Let $k \geq 0$, $\gamma \in \wcG_*$ and $A = \{ g \in \wcG^ *: g_k = \gamma_k \}$. Corollary \ref{cor:azumaCM}, Proposition \ref{prop:convlocCM} and Borel-Cantelli's Lemma imply that with probability one, $\rho_n (A) \to \rho(A)$.  The collection of sets $A = A(k,\gamma)$, $k \geq 0$, $\gamma \in \wcG^*$, being a basis of the topology on  $\wcG^*$, this proves the first statement of Theorem \ref{th:convlocCM}. 

For the second statement, we notice that if $G_n$ is uniformly distributed on $\cG(\bD^{(n)},h)$ and $\widehat{G}_n \stackrel{d}{\sim} \CM(\bD^{(n)})$, then Lemma \ref{le:unifCM} implies for any subset $B \subset \cG(\bD^{(n)},h)$ that 
$$
\dP ( G_n \in B) \leq \dP ( \widehat G_n \in B) / \dP ( \widehat G_n \in \cG(\bD^{(n)},h)). 
$$
By Corollary \ref{cor:Poilim}, $\dP ( \widehat G_n \in \cG(\bD^{(n)},h))$ is lower bounded by some $\alpha >0$, uniformly in $n$ (depending of the sequence $\bD^{(n)}$ and $h$). Then, if $\rho_n = U(G_n)$ and $\widehat \rho_n = U(\widehat G_n)$, from what precedes, for any $t > 0$, and $A$ as above
$$
\sum_n \dP (\ABS{ \rho_n(A) - \dE\widehat \rho_n(A)}  > t  ) \leq \alpha^{-1} \sum_n \dP (\ABS{ \widehat \rho_n(A) - \dE \widehat \rho_n(A)}  > t  ) < \infty.
$$
It remains to apply again Borel-Cantelli's lemma and Proposition \ref{prop:convlocCM}. 
 \ep

\begin{remark}
The proof of Theorem \ref{th:convlocCM} actually shows that for a sequence $\bD^{(n)}$ satisfying (H1)-(H2) the following holds. Let $k \geq 0$, $\gamma$ fixed and $A = \{ g : \;g_k = \gamma_k \}$. There exists a constant $\d > 0$ (depending of the sequence $\bD^{(n)}$, $h$ and $k$), such that, if $G_n$ is uniformly distributed on $\cG(\bD^{(n)},h)$ and $\widehat{G}_n \stackrel{d}{\sim} \CM(\bD^{(n)})$, we have for any $t >0$ 
 \begin{equation}\label{eq:concunifCM}
\dP \left( \left|     \rho_n ( A)   -  \dE  \widehat \rho_n ( A)     \right|  \geq t  \right) \leq  \d^{-1} \exp\left (-   \d  n t^2  \right),
\end{equation}
 where $\rho_n = U(G_n)$ and $\widehat \rho_n = U(\widehat G_n)$.
\end{remark}

\subsubsection{Proof of Proposition \ref{th:azumaCM}} \label{propa4}
The proof is a consequence of Azuma-Hoeffding's inequality.

\paragraph{\em Case of random matchings}

For clarity, we start with the case $L=1$, i.e. $\cC = \{(1,1)\}$. Then $N = S_{(1,1)}$. We order the elements of $W = W_{(1,1)}$ in the lexicographic order. We may identify a matching of $W$ as the set of $N/2$ matched pairs. We order these $N/2$ pairs by the index of their smallest element. We then define $\cF_0$ as the trivial $\sigma$-algebra and for $1 \leq k \leq N/2$, we define $\cF_k$ as the $\sigma$-algebra generated by the first $k$ pairs of matched elements of $\sigma$.  We set $Z_k = \dE [ F ( \sigma ) | \cF_{k} ]$, so that $Z_0 = \dE F(\sigma)$, $Z_{N/2 - 1} = F(\sigma)$. By construction, $Z_k$ is a Doob martingale. 

If $A$ is a finite set, we  denote by $\bM(A)$ the set of perfect matchings on $A$. With our previous notation $\Si = \bM ( W)$. For $1 \leq k \leq N/2$, an element $\sigma$ of $\bM(W)$  can be uniquely decomposed into $(\sigma^-_{k-1}, \sigma^+_{k})$ where  $\sigma^-_{k-1}$ is the restriction of $\sigma$ to the $k-1$ smallest pairs and $\sigma^+_{k}$ is the rest. Let $W^{k-1}$ denote the subset of $W$ such that $ \sigma^-_{k-1}$ is a perfect matching on $W^{k-1}$.

If $v_{k}$ is the smallest element of $W \backslash W^{k-1}$, we set $w_k = \sigma(v_k)  \in W \backslash W^{k-1}$,  so that $W^k = W^{k-1} \cup \{ v_k , w_k \}$.  Now, for $w \in W \backslash ( W^{k-1} \cup \{ v_k \} )$,  let $\bM_w$ denote the set of matchings of $W \backslash W^{k-1}$ such that $m (v_k) = w$. Then for any $w,w' \in W \backslash ( W^{k-1} \cup \{ v_k \} )$, each $m \in \bM_w$ corresponds to a unique $m' \in \bM_{w'}$ through the switch $
\{\{v_k, w \}, \{w',  z \} \} \to \{\{v_k, w'\}, \{w, z\}\}$, where $m (w')  = z$. This gives a bijection between $\bM_w$ and $\bM_{w'}$, and we set $N_k = | \bM_w |$.  By assumption, we deduce that for any $w, w'$, 
$$
\Big| \sum_{ m \in \bM_{w} } F ( \sigma^-_k , m ) - \sum_{ m \in \bM_{w'} }F ( \sigma^-_k , m ) \Big| \leq \kappa. 
$$
Applying the above inequality to $w_k$, we deduce that  
\begin{align*}
 \Big| \frac{1 }{ N_k } \sum_{ m \in \bM_{w_k} } F ( \sigma^-_k , m ) -  \frac{ 1} { N - 2 k + 1 } \sum_{ w \in W \backslash ( W^{k-1} \cup \{ v_k \} )   }  \frac{1 }{ N_k }  \sum_{ m \in \bM_{w} }F ( \sigma^-_k , m ) \Big|  =  \left| Z_{k} - Z_{k-1} \right| \leq \kappa. 
\end{align*}
We may then apply Azuma-Hoeffding's inequality to the martingale $Z_k$. We obtain that for any $t >0$, 
$$
\dP\left( Z_n   - Z_0  \geq t\right) \leq \exp\left ({-   t^2 \over  2 \sum_{i=1}^{N/2 -1} \kappa^ 2} \right) \leq \exp\left ({-   t^2 \over  N  \kappa^ 2} \right). 
$$
This proves the proposition when $L =1$. 

\paragraph{\em Case of random bijections}

Now, let $\si$ be a uniformly drawn bijection from the set $W$ to $\bar W$ of common cardinality $N$. We order the elements of the set $W$ as $(x_1, \cdots, x_N)$. We introduce the filtration $\cF_k$ generated by $\si(x_1), \cdots, \si (x_k)$. Now let $F$ be a function on the set of bijections from $W$ to $\bar W$ such that for any $m$, $m'$ such that $m' \circ m^ {-1}$ is a transposition, we have $|F(m) - F(m')| \leq \kappa$. We set $Z_k = \dE [ F(\si) | \cF_k]$.  

With minor modifications, the above argument shows that, for $1 \leq k \leq N$, $|Z_{k-1} - Z_k| \leq \kappa$. Then, by Azuma-Hoeffding's inequality, we find for $t \geq 0$, 
\begin{equation}\label{AZH}
\dP\left( F ( \si)    - \dE F( \si)   \geq t\right) \leq \exp\left ({-   t^2 \over  2 N  \kappa^ 2} \right). 
\end{equation}

\paragraph{\em General case}

It suffices to combine the two results. We first order the elements of $\cC_{\leq}$ in an arbitrary way, say $c_1, \cdots, c_{\ell}$ with $\ell = L ( L+1) / 2$. Let $N_0 = 0$, for $1 \leq k \leq \ell$, $N_k = \sum_{i=1}^k S_{c_i}/(1 + \IND_{c_i \in \cC_=})$. We have $N_{\ell} = N /2$. We define the filtration, $(\cF_{t}), 0 \leq t \leq N_{\ell},$ built as follows. $\cF_0$ is the trivial $\si$-algebra, $\cF_{N_k}$ is the filtration generated by the independent variables $(\si_{c_i}), 1 \leq i \leq k$. Finally for $1 \leq i < N_{k+1} - N_k$, $\cF_{N_{k} +i}$ is the filtration generated by $\cF_{N_k}$ and the first $i$ matched pairs of $\si_{c_{k+1}}$. 

As above, we set $Z_k = \dE [ F(\si) | \cF_k]$, so that  $Z_0 = \dE F ( \si)$, $Z_{N_{\ell}} = F ( \si)$. By construction, using the independence of $(\si)_c, c\in \cC$, we find, for $1 \leq k \leq N_\ell$, $|Z_{k-1} - Z_k| \leq \kappa$. Hence, Azuma-Hoeffding's inequality implies that \eqref{AZH} holds for $\si$ uniform in $\Sigma$, with $N$ replaced by $N_\ell = N/2$. 
This ends the proof of Proposition \ref{th:azumaCM}. 

\section*{Acknowledgments}
We thank  Justin Salez for bringing reference \cite{BLS} to our attention and B\'alint Vir\'ag for a discussion on the discontinuity of the entropy. This work was supported by the GDRE GREFI-MEFI CNRS-INdAM. Partial support of the European Research Council
through the Advanced Grant PTRELSS 228032 and ANR-11-JS02-005-01 is also acknowledged.

\bibliographystyle{abbrv}
\bibliography{mat}

\end{document}